\documentclass[11pt,a4paper]{amsart}

\usepackage[norelsize, ruled, boxed, commentsnumbered]{algorithm2e}
\usepackage[latin1]{inputenc} 
\usepackage{enumerate}
\usepackage{mdwlist}
\usepackage{amsthm,amsmath,amssymb}

\newtheorem{firstthm}{Proposition}[section] 
\newtheorem{thm}[firstthm]{Theorem}
\newtheorem{lemma}[firstthm]{Lemma}
\newtheorem{claim}[firstthm]{Claim}
\newtheorem{coro}[firstthm]{Corollary} 
\newtheorem{prop}[firstthm]{Proposition}
\newtheorem{conj}[firstthm]{Conjecture}
\theoremstyle{definition}
\newtheorem{defin}[firstthm]{Definition}
\newtheorem{construct}[firstthm]{Construction}
\newtheorem{setup}[firstthm]{Setup}
\newtheorem{rem}[firstthm]{Remark}

\newcommand{\sm}{\setminus}
\newcommand{\eps}{\varepsilon}

\newcommand{\Z}{{\mathbb Z}}

\newcommand{\Prob}{{\mathbb P}}

\newcommand{\Exp}{{\mathbb E}}
\newcommand{\ab}{{\bf a}}

\newcommand{\yb}{{\bf y}}

\newcommand{\0}{{\bf 0}} 
\newcommand{\1}{{\bf 1}}
\newcommand{\Part}{\mathcal{P}}
\newcommand{\Qart}{\mathcal{Q}} 
\newcommand{\Sart}{{\mathcal S}}
\newcommand{\hPart}{{\hat{\mathcal{P}}}}
\newcommand{\hQart}{{\hat{\mathcal{Q}}}}
\newcommand{\R}{\mathbb{R}}
\newcommand{\db}{{\bf d}}
\newcommand{\ub}{{\bf u}} 
\newcommand{\vb}{{\bf v}}
\newcommand{\ib}{{\bf i}}
\newcommand{\jb}{{\bf j}}
\newcommand{\xb}{{\bf x}}
\newcommand{\nb}{{\bf n}}
\newcommand{\PM}{\mathrm{{\bf PM}}}

\newcommand{\mc}[1]{\mathcal{#1}}
\newcommand{\mb}[1]{\mathbb{#1}}

\newcommand{\brac}[1]{\left( #1 \right)}
\newcommand{\bracc}[1]{\left\{ #1 \right\}}

\newcommand{\bgen}[1]{\left\langle #1 \right\rangle}
\newcommand{\sub}{\subseteq}

\newcommand{\es}{\emptyset}

\newcommand{\N}{{\mathbb N}}

\date{\today}
\title{Polynomial-time perfect matchings in dense hypergraphs}
\author{Peter Keevash, Fiachra Knox and Richard Mycroft}
\thanks{Research supported in part by ERC grant 239696 and EPSRC grant EP/G056730/1.}

\begin{document}
\vspace*{-0.8cm}
\begin{abstract}
Let $H$ be a $k$-graph on $n$ vertices, with minimum codegree at least $n/k + cn$ for some fixed $c > 0$. In this paper we construct a polynomial-time algorithm which finds either a perfect matching in $H$ or a certificate that none exists. This essentially solves a problem of Karpi\'nski, Ruci\'nski and Szyma\'nska; Szyma\'nska previously showed that this problem is NP-hard for a minimum codegree of $n/k - cn$. Our algorithm relies on a theoretical result of independent interest, in which we characterise any such hypergraph with no perfect matching using a family of lattice-based constructions. 
\end{abstract}
\maketitle
\vspace*{-0.6cm}

\section{Introduction} \label{IntroSec}
The question of whether a given $k$-uniform hypergraph (or $k$-graph) $H$ contains a perfect matching (i.e.\ a partition of the vertex set into edges), while simple to state, is one of the key questions of combinatorics. 
In the graph case $k = 2$, Tutte's Theorem~\cite{Tutte47} gives necessary and sufficient conditions for $H$ to contain a perfect matching, and Edmonds' Algorithm~\cite{edmonds} finds such a matching in polynomial time.
However, for $k \geq 3$ this problem was one of Karp's celebrated 21 NP-complete problems~\cite{Karp72}. 
Results for perfect matchings in hypergraphs have many potential practical applications;
one example which has garnered interest in recent years is the `Santa Claus' allocation problem (see~\cite{AFS08}).
Since the general problem is intractable provided P $\neq$ NP, it is natural to seek conditions on $H$ which render the problem tractable or even guarantee that a perfect matching exists. In recent years a substantial amount of progress has been made in this direction.
One well-studied class of such conditions are minimum degree conditions. This paper provides an algorithm that essentially eliminates the hardness gap between the sparse and dense cases for the most-studied of these conditions.

\subsection{Minimum degree conditions}
Suppose that $H$ has $n$ vertices and that $k$ divides $n$ (we assume this throughout, since it is a necessary condition for $H$ to contain a perfect matching).
In the graph case, a simple argument shows that a minimum degree of $n/2$ guarantees a perfect matching. 
Indeed, Dirac's Theorem~\cite{dirac} states that this condition even guarantees that $H$ contains a Hamilton cycle.
For $k \geq 3$, there are several natural definitions of the minimum degree of $H$.
Indeed, for any set $A \subseteq V(H)$, the \emph{degree} $d(A)=d_H(A)$ of $A$ is the number of edges of $H$ containing $A$.
Then for any $1 \leq \ell \leq k-1$, the \emph{minimum $\ell$-degree} $\delta_{\ell}(H)$ of $H$ is the minimum of $d(A)$ over all subsets $A \subseteq V(H)$ of size $\ell$. Two cases have received particular attention: the minimum $1$-degree $\delta_{1}(H)$ is also known as the \emph{minimum vertex degree} of $H$, and the minimum $(k-1)$-degree $\delta_{k-1}(H)$ as the \emph{minimum codegree} of $H$.

For sufficiently large $n$, R\"odl, Ruci\'nski and Szemer\'edi~\cite{RRS09} determined the minimum codegree which guarantees a perfect matching in $H$ to be exactly $n/2 - k + c$, where $c \in \{1.5, 2, 2.5, 3\}$ is an explicitly given function of $n$ and $k$.
They also showed that the condition $\delta_{k-1}(H) \geq n/k + O(\log n)$ is sufficient to guarantee a matching covering all but
$k$ vertices of $H$, i.e.\ one edge away from a perfect matching; their conjecture that $\delta_{k-1}(H) \geq n/k$ suffices for
this was recently proved by Han~\cite{H1}. This provides a sharp contrast to the graph case, where a minimum degree of $\delta(G) \geq n/2 - \eps n$ only guarantees the existence of a matching covering at least $n - 2\eps n$ vertices.
There is a large literature on minimum degree conditions for perfect matchings in hypergraphs, see e.g.~\cite{AGS09, Alon+, CK, DH, HPS, KRS, KM11p, Khp1, Khp2, KO, KOT, LMp, LM, MR2, Pikhurko, RRS09, Szy, TZ12, TZp} and the survey by R\"odl and Ruci\'nski~\cite{RR10} for details.

Let $\PM(k, \delta)$ be the decision problem of determining whether a $k$-graph $H$ with $\delta_{k-1}(H) \geq \delta n$ contains a perfect matching.
Given the result of~\cite{RRS09}, a natural question to ask is the following: 
For which values of $\delta$ can $\PM(k, \delta)$ be decided in polynomial time? 
This holds for $\PM(k, 1/2)$ by the main result of~\cite{RRS09}.
On the other hand, $\PM(k, 0)$ includes no degree restriction on $H$ at all,
so is NP-complete by the result of Karp~\cite{Karp72}.
Szyma\'nska~\cite{Szy09, Szy} proved that for $\delta < 1/k$ 
the problem $\PM(k, 0)$ admits a polynomial-time reduction to $\PM(k, \delta)$ and hence $\PM(k, \delta)$ is also NP-complete,
while Karpi\'nski, Ruci\'nski and Szyma\'nska~\cite{KRS} showed that there exists $\eps > 0$ such that $\PM(k, 1/2-\eps)$ is in P.
This left a hardness gap for $\PM(k, \delta)$ when $\delta \in \left[1/k, 1/2 - \eps\right)$.

In this paper we provide an algorithm which eliminates this hardness gap almost entirely. 
Moreover, it not only solves the decision problem, but also provides a perfect matching or a certificate that none exists.

\begin{thm} \label{main}
Fix $k \geq 3$ and $\gamma > 0$. 
Then there is an algorithm with running time $O(n^{3k^2 - 7k + 1})$, which given any $k$-graph $H$ on $n$ vertices with $\delta_{k-1}(H) \geq (1/k + \gamma)n$, finds either a perfect matching or a certificate that no perfect matching exists. 
\end{thm}

A preliminary version of this algorithm (with a slower running time) appeared as an extended abstract \cite{KKMabs}.

\subsection{Lattices and divisibility barriers}

Theorem~\ref{main} relies on a result of Keevash and Mycroft~\cite{KM11p} giving fairly general sufficient conditions which ensure a perfect matching in a $k$-graph. 
In this context, their result essentially states that if $H$ is a $k$-graph on $n$ vertices, and $\delta_{k-1}(H) \geq n/k + o(n)$, then $H$ either contains a perfect matching or is close to one of a family of lattice-based constructions termed `divisibility barriers'. 
These constructions play a key role in this paper, so we now describe them in some detail.

The simplest example of a divisibility barrier is the following construction, given as one of the two extremal examples in~\cite{RRS09}.

\begin{construct} \label{RRSconstruct}
Let $A$ and $B$ be disjoint sets such that $|A|$ is odd and $|A\cup B| = n$, and let $H$ be the $k$-graph on $A \cup B$ whose edges are all $k$-sets which intersect $A$ in an even number of vertices.
\end{construct}

We consider a partition to include an implicit order on its parts.
To describe divisibility barriers in general, we make the following definition.

\begin{defin} \label{def:lattices0} Let $H$ be a $k$-graph and let $\Part$ be a partition of $V(H)$ into $d$ parts. 
Then the \emph{index vector} $\ib_\Part(S) \in \Z^d$ of a subset $S \subseteq V(H)$ with respect to $\Part$ is the vector whose coordinates are the sizes of the intersections of $S$ with each part of $\Part$, i.e.\ $\ib_\Part(S)_X = |S \cap X|$ for $X \in \Part$. Further,
\begin{enumerate}[(i)]
\item $I_\Part(H)$ denotes the set of index vectors $\ib_\Part(e)$ of edges $e \in H$, and
\item $L_\Part(H)$ denotes the lattice (i.e.\ additive subgroup) in $\Z^d$ generated by $I_\Part(H)$.
\end{enumerate}
\end{defin}

A \emph{divisibility barrier} is a $k$-graph $H$ which admits a partition $\Part$ of its vertex set $V$ such that $\ib_\Part(V) \notin L_\Part(H)$; the next proposition shows that such an $H$ contains no perfect matching.
To see that this generalises Construction~\ref{RRSconstruct}, let $\Part$ be the partition into parts $A$ and $B$; then $L_\Part(H)$ is the lattice of vectors $(x, y)$ in $\Z^2$ for which $x$ is even, and $|A|$ being odd implies that $\ib_\Part(V) \notin L_\Part(H)$.

\begin{prop} \label{divbarrier}
Let $H$ be a $k$-graph with vertex set $V$. If there is a partition $\Part$ of $V$ 
with $\ib_\Part(V) \notin L_\Part(H)$ then $H$ does not contain a perfect matching.
\end{prop}

\begin{proof}
Suppose $M$ is a matching in $H$. Then $\ib_\Part(V(M)) = \sum_{e \in M} \ib_\Part(e) \in L_\Part(H)$.
But $\ib_\Part(V) \notin L_\Part(H)$, so $V(M) \ne V$, i.e.\ $M$ is not perfect.
\end{proof}

A special case of the main theoretical result of this paper is the following theorem, which states that the converse of Proposition~\ref{divbarrier} holds for sufficiently large $3$-graphs as in Theorem~\ref{main}. Thus we obtain an essentially best-possible strong stability `Andrasfai-Erd\H{o}s-Sos analogue' for the result of R\"odl, Ruci\'nski and Szemer\'edi~\cite{RRS09} in the case $k=3$.

\begin{thm} \label{PMNeccSuff3case} For any $\gamma > 0$ there exists $n_0 = n_0(\gamma)$ such that the following statement holds. Let $H$ be a $3$-graph on $n \geq n_0$ vertices, such that $3$ divides $n$ and $\delta_{2}(H) \geq (1/3 + \gamma)n$, and suppose that $H$ does not contain a perfect matching. Then there is a subset $A \subseteq V(H)$ such that $|A|$ is odd but every edge of $H$ intersects $A$ in an even number of vertices.
\end{thm}

Theorem~\ref{PMNeccSuff3case} can be used to decide $\PM(3,1/3+\gamma)$, as the existence of a subset $A$ as in the theorem can be checked using (simpler versions of) the algorithms in Section~2. However, the case $k=3$ is particularly simple because there is only one maximal divisibility barrier; for $k \geq 4$, the next construction shows that the converse of Proposition~\ref{divbarrier} does not hold for general $k$-graphs as in Theorem~\ref{main}.

\begin{construct} \label{nopm}
Let $A$, $B$ and $C$ be disjoint sets of vertices with $|A\cup B \cup C| = n$, $|A|,|B|,|C| = n/3 \pm 2$ and $|A| = |B| + 2$. Fix some vertex $x \in A$, and let $H$ be the $k$-graph with vertex set $A \cup B \cup C$ whose edges are
\begin{enumerate}
\item any $k$-set $e$ with $|e \cap A| = |e \cap B|$ modulo 3, and
\item any $k$-set $(x, z_1, \dots, z_{k-1})$ with $z_1, \dots, z_{k-1}$ in $C$.
\end{enumerate}
\end{construct}

Construction~\ref{nopm} satisfies $\delta_{k-1}(H) \geq n/3 - k - 1$, so if $k \geq 4$ then $H$ meets the degree condition of
Theorem~\ref{main}. Moreover, for any partition $\Part$ of $V(H)$ we have $\ib_\Part(V(H)) \in L_\Part(H)$. 
To see this, suppose on the contrary that it is false, and fix a counterexample $\Part$ with as few parts as possible.
Then $L_\Part(H)$ cannot contain any transferral (see Definition \ref{def:full}),
as otherwise we could merge the corresponding parts of $\Part$ to obtain a counterexample with fewer parts. 
It follows that $A$, $B$ and $C$ must each be contained within some part of $\Part$. 
We consider the case $\Part=(A,B,C)$ and omit the easy cases where $\Part$ has two parts 
(clearly the partition with one part is not a counterexample, as $k \mid n$).
Write $\ib_\Part(V(H)) = (x,y,z)$, where $x + y + z = n$.
Note that $(0,0,k)$, $(1,1,k-2)$ and $(1,0,k-1)$ are all in $I_\Part(H)$.
Then $(x,y,z) = y(1,1,k-2) + (x-y)(1,0,k-1) + (n/k-x)(0,0,k) \in L_\Part(H)$,
so $\Part$ is not a counterexample.

However, $H$ does not contain a perfect matching. To see this, let $M$ be a matching in $H$, and note that any edge $e \in M$ has
$|e \cap A| = |e \cap B|$ modulo 3, except for at most one edge of $M$ which has $|e \cap A| = |e \cap B| + 1$ modulo 3. 
Then $\ib_\Part(V(M))_1 - \ib_\Part(V(M))_2 \in \{0, 1\}$ modulo 3, 
whereas $\ib_\Part(V(H))_1 - \ib_\Part(V(H))_2 = |A| - |B| = 2$,
so $V(M) \neq V(H)$, that is, $M$ is not perfect. 

\subsection{Approximate divisibility barriers}

Our starting point will be (a special case of) a result of Keevash and Mycroft~\cite{KM11p} on approximate divisibility barriers. First we introduce a less restrictive degree assumption, which follows from the assumption in Theorem~\ref{main} when $\gamma>0$ is small. (The reason for doing so will become clear at the end of Section 2.) 
We will use the following setup throughout the paper.

\begin{setup} \label{setup}
Suppose that $k \ge 3$, that $1/n_0 \ll \eps \ll \gamma \ll 1/k$ and that $n \ge n_0$ satisfies $k \mid n$.

Let $H$ be a $k$-graph on $n$ vertices such that
\begin{itemize}
\item[(deg)]  $\delta_1(H) \geq \gamma n^{k-1}$, and
\item[(codeg)] at most $\eps n^{k-1}$ $(k-1)$-sets $A \sub V(H)$ have $d_H(A) < (1/k + \gamma)n$.
\end{itemize}
\end{setup}

A result from~\cite{KM11p} (stated here as Theorem~\ref{HypergraphMatching}), combined with Lemma~\ref{PrunePartiteComplex}, implies that under Setup~\ref{setup}, if $H$ does not contain a perfect matching then we can delete $o(n^k)$ edges from $H$ to obtain a subgraph $H'$ for which there exists a partition $\Part$ of $V(H')$ such that $\ib_\Part(V(H')) \notin L_\Part(H')$. Thus if $H$ is far from a divisibility barrier then it has a perfect matching. On the other hand, if $H$ is itself a divisibility barrier then Proposition~\ref{divbarrier} implies that $H$ does not have a perfect matching. However, our algorithm cannot search directly for such partitions $\Part$, since the number of edges to be deleted, while small compared to $n^k$, is still large from a computational perspective. 

The main theoretical contribution of this paper is to fill the gap between these cases,
by giving a condition for the existence of a perfect matching under Setup~\ref{setup}
that is necessary and sufficient, and also efficiently checkable.

\subsection{Definitions}
Before giving the statement we require several definitions. 
Firstly, it will be sufficient to consider the following special classes of lattices.
To motivate this definition, we remark that all of our lattices will be edge-lattices,
any incomplete lattice can be simplified to one that is transferral-free,
and our assumptions on $H$ will imply that if its edge-lattice with respect to some partition
is transferral-free then it is full.
Note also that our definitions depend on $k$,
but we consider this to be fixed throughout the paper.
We write $|\Part|$ for the number of parts of a partition $\Part$.

\begin{defin} \label{def:full}
Suppose $L$ is a lattice in $\Z^d$.
\begin{enumerate}[(i)]
\item We say that $\ib \in \Z^d$ is an \emph{$r$-vector} if it has non-negative co-ordinates that sum to $r$.
We write $\ub_j$ for the `unit' $1$-vector that is $1$ in co-ordinate $j$ and $0$ in all other co-ordinates.
\item We say that $L$ is an \emph{edge-lattice} if it is generated by a set of $k$-vectors.
\item We write $L_{\max}^d$ for the lattice generated by all $k$-vectors.
\item We say that $L$ is \emph{complete} if $L = L_{\max}^d$, otherwise it is \emph{incomplete}.
\item A \emph{transferral} is a non-zero difference $\ub_i-\ub_j$ of $1$-vectors.
\item We say that $L$ is \emph{transferral-free} if it does not contain any transferral.
\item We say that a set $I$ of $k$-vectors is \emph{full} if 
for every $(k-1)$-vector $\vb$ there is some $i \in [d]$ such that $\vb + \ub_i \in I$.
\item We say that $L$ is \emph{full} if it contains a full set of $k$-vectors and is transferral-free.
\end{enumerate}
\end{defin}

Note that $L_{\max}^d$ can be equivalently defined as the lattice of all vectors in $\Z^d$ whose coordinates sum to a multiple of $k$.
Often the dimensions of $\Z^d$ will correspond to parts of a partition $\Part$ of a set; in this case we refer to $\Z^\Part$ and $L_{\max}^\Part$ instead of $\Z^d$ and $L_{\max}^d$.

Now we come to the structures that appear in our characterisation.
To motivate this definition, 
we remark that the first step in our strategy for finding a perfect matching 
will be to identify a `canonical' partition and lattice,
then delete the vertices covered by a small matching so that the index of the remaining set is on the lattice.

\begin{defin} \label{Certificate} Let $H$ be a $k$-graph. 
\begin{enumerate}[(i)]
\item A \emph{full pair} $(\Part, L)$ for $H$ consists of a partition $\Part$ of $V(H)$ into $d \le k-1$ parts 
and a full edge-lattice $L \subseteq \Z^d$ (possibly distinct from $L_\Part(H)$).
\item A (possibly empty) matching $M$ of size at most $|\Part| - 1$ is a \emph{solution} for $(\Part, L)$ (in $H$)
if $\ib_\Part(V(H) \backslash V(M)) \in L$; we say that $(\Part, L)$ is \emph{soluble} if it has a solution, otherwise \emph{insoluble}.
\item A full pair $(\Part, L)$ is \emph{$C$-far} for $H$ if some set of $C$ vertices intersects every edge $e \in H$ with $\ib_\Part(e) \notin L$.
\item A \emph{$C$-certificate} for $H$ is an insoluble $C$-far full pair for $H$.
\end{enumerate}
\end{defin}

We will see later that there is no loss of generality in considering partitions into at most $k-1$ parts 
because of our codegree assumption, or in requiring solutions to have size at most $|\Part| - 1$, 
as if there is any matching $M$ such that $\ib_\Part(V(H) \backslash V(M)) \in L$ then there is one with $|M| \le |\Part| - 1$.
In particular, if $H$ has a perfect matching then any full pair is soluble, so there is no $C$-certificate for $H$ for any $C \geq 0$.

\subsection{Main structural result} 
Now we can state our structural characterisation 
for the perfect matching problem under Setup~\ref{setup}.

\begin{thm} \label{PMNeccSuff}
Under Setup~\ref{setup}, $H$ has a perfect matching if and only if there is no $2k(k-3)$-certificate for $H$.
\end{thm}

We remark that most of the work in proving Theorem~\ref{PMNeccSuff} lies in establishing the ``if'' part of the statement; the ``only if'' part follows easily from the results in Section~\ref{FullSection}.
Our algorithm for the decision problem is essentially an exhaustive search for a $2k(k-3)$-certificate, although we also need to provide an algorithm to efficiently list the partitions $\Part$ that may arise in Definition~\ref{Certificate}.
Thus Theorem~\ref{PMNeccSuff} is required to prove correctness of the algorithm.
We remark that the constant $2k(k-3)$ is within a factor $2$ of being best-possible (see Conjecture~\ref{StrongerPMNS}).
Furthermore, the constant $1/k$ in Setup~\ref{setup} is best-possible, as shown by the `space barrier' construction
(see the concluding remarks for the definition and further discussion).

\subsection{Contents}
In the next section we present the algorithmic details of the results in this introduction. 
That is, we assume Theorem~\ref{PMNeccSuff} and deduce Theorem~\ref{main}. 
The rest of the paper is devoted to the proof of Theorem~\ref{PMNeccSuff}, beginning with Section~\ref{SketchSection} in which we briefly sketch the ideas of the proof and its key lemmas
(Lemmas~\ref{NonPartiteLemma} and~\ref{kPartiteLemma}).
In Section~\ref{PrelimSection} we introduce a number of necessary preliminaries,
including results from~\cite{KM11p}, some convex geometry and well-known probabilistic tools. 
Section~\ref{RobMaxSection} focuses on an important definition, that of `robust maximality', and some of its properties;
this turns out to be the correct notion for the key lemmas.
In Section~\ref{FullSection} we establish various properties of full lattices,
including a characterisation in terms of finite abelian groups.
After these preparations, we prove the key lemmas in Section~\ref{KeyLemmaSection}.
In Section~\ref{deduceThm} we deduce Theorem~\ref{PMNeccSuff} from Lemma~\ref{NonPartiteLemma}
and some additional lemmas on subsequence sums in finite abelian groups.
Section~\ref{DeferredProofsSection} contains some technical proofs which were deferred from earlier sections.
Finally, in Section~\ref{PartiteSection} we state the multipartite versions of our results (omitting the similar proofs)
and make some concluding remarks.

\subsection{Notation and terminology}
A \emph{hypergraph} $H$ consists of a vertex set $V(H)$ and a set $E(H)$ of edges $e \subseteq V(H)$. 
We frequently identify a hypergraph $H$ with its edge set, for example, writing $e \in H$ for $e \in E(H)$ and $|H|$ for $|E(H)|$. A \emph{$k$-uniform} hypergraph, or \emph{$k$-graph}, is a hypergraph in which every edge is a \emph{$k$-set}, that is, a set of size $k$. 
Given a hypergraph $H$ and $A \sub V(H)$, the \emph{neighbourhood} of $A$ in $H$ is $H(A) = \{B \sub V(H) \sm A: A \cup B \in
H\}$. 
Note that $|H(A)|=d_H(A)$ is the degree of $A$ in $H$.

The hypergraph $H[A]$ (as distinct from $H(A)$) is the hypergraph with vertex set $A$ 
whose edges are the edges of $H$ which are contained in $A$.
We use $H-A$ and $H \sm A$ interchangeably to denote the hypergraph obtained from $H$ 
by deleting $A$ and all edges which intersect $A$ (this is identical to $H[V(H) \sm A]$). 
If $H$ is a hypergraph and $H' \sub H$
we say that $H'$ is a \emph{subgraph} of $H$ (we prefer to avoid the terms `subhypergraph' and `sub-$k$-graph').

We use bold font for vectors and normal font for their co-ordinates, e.g.\ $\vb=(v_1,\dots,v_d)$. We write $\0$ and $\1$ for the vectors whose co-ordinates are all zero and one respectively (the dimension of the vector will be clear from the context). $B^d(\vb, r)$ denotes the ball of radius $r$ around $\vb \in \Z^d$; we sometimes omit the dimension $d$.
We will often work with vectors $\vb$ in $\Z^d$, for some $d$, in which the coordinates are indexed by some ordered 
partition $\Part$ of a set $V$ with $d$ parts. 
If $X$ is the $j$th
part of $\Part$ for some $j \in [d]$ we write $v_X = v_j$. 

We say that an event $E$ holds \emph{with high probability} 
if $\mb{P}(E) = 1-e^{-\Omega(n^c)}$ for some $c>0$ as $n \to \infty$; note that when $n$ is sufficiently large,
by union bounds we can assume that any specified polynomial number of such events all occur. We write $[r]$ to denote the set of integers from $1$ to $r$, and $x \ll y$ to mean for any $y \geq 0$ there exists $x_0 \geq 0$ such
that for any $x \leq x_0$ the following statement holds. Similar statements with more constants are defined
similarly. Also, we write $a = b \pm c$ to mean $b - c \leq a \leq b + c$.
Throughout the paper we omit floor and ceiling symbols where they do not affect the argument. 

\section{Algorithms and analysis} \label{AlgSection}

We start with the following theorem, which can be used to solve the decision problem of determining whether or not $H$ has a perfect matching. Note that the running time for $k=3$ is $O(n^3)$, which we cannot reasonably expect to improve, as a faster algorithm would not be able to query all edges of $H$.

\begin{thm} \label{polytime}
Under Setup~\ref{setup}, Procedure~\ref{mainProc} determines whether $H$ contains a perfect matching in time $O(n^{3k^2 - 8k})$.
\end{thm}

\begin{procedure}
\SetAlgoVlined
\SetAlgoNoEnd
\caption{DeterminePM()}
\KwData{A $k$-graph $H$ as in Setup~\ref{setup}.}
\KwResult{Determines whether $H$ has a perfect matching.}
\If{$n < n_0$}{Examine every set of $n/k$ edges in $H$, and halt with appropriate output.}
\ForEach{set $S \subseteq V(H)$ of size at most $2k(k-3)$,
 integer $d \in [k-1]$, full edge-lattice $L \subseteq \Z^d$
 and partition $\Part$ of $V(H)$ into $d$ parts so that any edge $e \in H$ which does not intersect $S$ has $\ib_\Part(e) \in L$}
{\If{there is no matching $M \subseteq H$ of size at most $d-1$ such that $\ib_\Part(V(H) \sm V(M)) \in L$}{Output ``no perfect matching'' and halt.}
}
Output ``perfect matching'' and halt.
\label{mainProc}
\end{procedure}

Procedure~\ref{mainProc} is essentially an exhaustive search for a $2k(k-3)$-certificate. It is clear that the ranges of $S$, $d$, $L$ and $M$ in the procedure can be listed by brute force in polynomial time. However, brute force cannot be used for $\Part$, as there are potentially exponentially many possibilities to consider, so first we provide an algorithm to construct all possibilities for $\Part$.

We imagine each vertex class $V_j$ to be a `bin' to which vertices may be assigned, and keep track of a set $U$ of vertices yet to be assigned to a vertex class. So initially we take each $V_j$ to be empty and $U = V(H)$. The procedure operates as a \emph{search tree}; at certain points the instruction is to branch over a range of possibilities. 
This means to select one of these possibilities and continue with this choice, then, when the algorithm halts, to return to the branch point, select the next possibility, and so forth. 
Each branch may produce an output partition; the output of the procedure consists of all output partitions. An informal statement of our procedure is that we generate partitions by repeatedly branching over all possible assignments of a vertex to a partition class, exploring all consequences of each assignment before branching again. 
Furthermore, we only branch over assignments of vertices which satisfy the following condition. Given a set of assigned vertices, we call an unassigned vertex $x$ \emph{reliable} if there exists a set $B$ of $k-2$ assigned vertices such that $d(x \cup B) \geq (1/k + \gamma)n$.

\begin{lemma} \label{OmegaPolyTime} 
Under Setup~\ref{setup}, for any $d \in [k-1]$ and full edge-lattice $L \subseteq \Z^d$,
there are at most $d^{2k-2}$ partitions $\Part$ of $V(H)$ such that $\ib_\Part(e) \in L$ for every $e \in H$, 
and Procedure~\ref{OmegaProc} lists them in time $O(n^{k+1})$.
\end{lemma}

\begin{procedure}
\SetAlgoVlined
\SetAlgoNoEnd
\caption{ListPartitions()}
\KwData{A $k$-graph $H$ and a full edge-lattice $L \subseteq \Z^d$.}
\KwResult{Outputs all partitions $\Part$ of $V(H)$ with $\ib_\Part(e) \in L$ for every $e \in H$.}
\BlankLine
{Set $U = V(H)$ and let $\Part = (V_1, \ldots, V_d)$ be a partition of $V \sm U$ (so initially $V_i = \emptyset$).\\}
{Choose arbitrarily $A \subseteq V(H)$ of size $k-1$ such that $d(A) \geq (1/k+\gamma) n$.\\
Branch over all possible assignments of vertices in $A$ to $V_1, \ldots, V_d$.}
\BlankLine
\While{$U \neq \emptyset$}
{\eIf{$xy_1\dots y_{k-1} \in H$ for some vertices $x \in U$ and $y_1,\dots, y_{k-1} \notin U$}
{Fix $j \in [k]$ such that $\ib_\Part(y_1 \dots y_{k-1}) + \ub_j \in L$.\\
Assign $x$ to $V_j$ and remove $x$ from $U$.}
{Choose $x \in U$ which is reliable.\\
Branch over all possible assignments of $x$.
}
}
\lIf{$\ib_\Part(e) \in L$ for every $e \in H$}{halt with output $\Part$.}
\label{OmegaProc}
\end{procedure}

\begin{proof} 
First we note that since $L$ is full, the instruction ``Fix $j \in [k]$ such that $\ib_\Part(y_1 \dots y_{k-1}) + \ub_j \in L$'' in Procedure~\ref{OmegaProc} is well-defined, since for any $(k-1)$-vector $\vb$ there is precisely one $j \in [k]$ such that $\vb + \ub_j \in L$.

Next we show that if the number of assigned vertices is at least $(1/k + \gamma) n$ and at most $(1 - \gamma)n$ then there is always a reliable unassigned vertex. To see this, note that the number of sets $x \cup B$, where $x$ is unassigned and $B$ is a set of $k-2$ assigned vertices, is at least $\gamma n\binom{n/k + \gamma n}{k-2} > \eps n^{k-1}$ 
(recall from Setup~\ref{setup} that $\eps \ll \gamma$). 
Hence some such $x \cup B$ has degree at least $(1/k + \gamma)n$, and so $x$ is reliable.

Observe that after branching initially over all possible assignments of $A$, at least $(1/k + \gamma)n$ vertices will be assigned (all the neighbours of $A$) before the procedure branches again.
Further, the procedure will no longer branch once there are fewer than $\gamma n$ unassigned vertices remaining.
Indeed, in this case for any unassigned vertex $x$ there are fewer than $\gamma n^{k-1}$ edges containing $x$ and another unassigned vertex. 
Thus condition~(deg) implies that every unassigned vertex $x$ is contained in some edge $x y_1 \dots y_{k-1}$ of $H$ where $y_1, \dots, y_{k-1}$ are assigned. 
Hence when branching, we may always choose a reliable vertex as stated in the procedure.

The final line of the procedure ensures that any partition $\Part$ of $V(H)$ which is output has that property that $\ib_\Part(e) \in L$ for every $e \in H$. The converse is also true: any partition $\Part$ of $V(H)$ such that $\ib_\Part(e) \in L$ for every $e \in H$ will be output by some branch of the procedure.
To see this, consider the branch of the procedure in which, at each branch point, the vertex $x$ under consideration is assigned to the vertex class in which it lies in $\Part$.
By our initial remark, every other vertex of $H$ must also be assigned to the vertex class in which it lies in $\Part$.
We conclude that Procedure~\ref{OmegaProc} indeed runs correctly, returning all partitions $\Part$ of $V(H)$ such that $\ib_\Part(e) \in L$ for every $e \in H$.

It remains to bound the number of such partitions.
Consider some $x$ over which the procedure branches.
Then there can be no edge $x y_1 \dots y_{k-1}$ of $H$ where $y_1, \dots, y_{k-1}$ are assigned.
Suppose that this is the case, and let $B = y_1y_2\dots y_{k-2}$ be a set of assigned vertices 
such that $d(\{x\} \cup B) \geq (1/k+\gamma) n$.
(Such a $B$ must exist since we chose $x$ to be reliable.)
None of the $(1/k + \gamma)n$ vertices $v$ such that $xvy_1 \dots y_{k-2}$ is an edge of $H$ can have been assigned, and each will be assigned before the next branch of the procedure. 
We conclude that after any branch in the procedure and before the next branch, at least $(1/k + \gamma) n$ vertices are assigned.

Hence the search tree has depth at most $k$. 
Since the degree of the root is at most $d^{k-1}$ and every other vertex has $d$ children, the search tree has at most $d^{2k-2}$ leaves. 
At most one partition is output at each leaf, so at most $d^{2k-2}$ partitions $\Part$ will be output by the algorithm, as required. Furthermore, over all branches there will be at most $d^{2k-2} n$ iterations of the \emph{while} loop, and the condition of the first \emph{if} statement takes $O(n^k)$ operations to check, and so the overall running time is $O(n^{k+1})$.
\end{proof}

\begin{rem} \label{AlgRemark} In fact, the step in which the algorithm finds $xy_1\dots y_{k-1} \in H$ for some vertices $x \in U$ and $y_1,\dots, y_{k-1} \notin U$ can be made more efficient by maintaining a queue of `new' vertices and only considering $k$-tuples for which $y_1$ is the vertex at the front of the queue. 
Initially the queue consists of the vertices of $A$.
Whenever a vertex is assigned it is added to the back of the queue, and the front vertex $y_1$ is removed from the queue (i.e.~ceases to be new) once we have considered all $k$-tuples $xy_1 \dots y_{k-1}$ with $x \in U$ and $y_1, \dots, y_{k-1} \notin U$ in which $y_1$ is the only new vertex.
With this modification, the running time of the algorithm can be reduced to $O(n^k)$.
Also, for the purpose of Lemma~\ref{OmegaPolyTime} we could replace the constant 
$1/k+\gamma$ in Setup~\ref{setup} by any positive constant independent of $n$
and still deduce that the number of possible partitions is independent of $n$.
\end{rem}

\begin{proof}[Proof of Theorem~\ref{polytime}.]
Let $H$ be a $k$-graph as in Setup~\ref{setup}.
We begin by noting that Procedure~\ref{mainProc} determines whether $H$ has a perfect matching. 
Indeed, this is trivial if $n < n_0$, and if $n \geq n_0$ it follows from Theorem~\ref{PMNeccSuff},
as Procedure~\ref{mainProc} determines whether there is a $2k(k-3)$-certificate $(\Part, L)$ for $H$.
We estimate the running time as follows.
There are at most $n^{2k(k-3)}$ choices of sets $S$, and these can be generated in time $O(n^{2k(k-3)})$. 
Also, there are only a constant number of choices for $d$ and $L$, and these can be generated in constant time. 
Indeed, since $L$ is an edge-lattice, it is generated by a set of $k$-vectors,
and the number of such generating sets is bounded by a function of $k$.

For each choice of $S, d$ and $L$, generating the list of choices for $\Part$ takes time $O(n^{k+1})$ by Lemma~\ref{OmegaPolyTime} applied to $H \sm S$.
Further the number of choices for $\Part$ is constant, and for each one it takes time $O(n^{k(k-2)})$ to check for the existence of the matching $M$.
When $k > 3$, $k(k-2) > k+1$ and so we conclude that the running time is $O(n^{3k^2 - 8k})$, as required.
In the case $k = 3$, we use an improved algorithm with running time $O(n^3)$ (which exists as noted in Remark~\ref{AlgRemark}) to generate the list of choices for $\Part$, thus obtaining the result.
\end{proof}

Now we will motivate the proof of our main result, Theorem~\ref{main}. 
We start by using Procedure~\ref{mainProc} to check whether $H$ contains a perfect matching. 
If $H$ has no perfect matching then this is certified by the $2k(k-3)$-certificate $(\Part, L)$ found by Procedure~\ref{mainProc}. 
So suppose that $H$ does contain a perfect matching. How can we find it? A naive attempt at a proof is the following well-known idea. We examine each edge $e$ of $H$ in turn and use the same procedure to test whether deleting the vertices of $e$ would still leave a perfect matching in the remainder of $H$, in which case we say that $e$ is \emph{safe}. There must be some safe edge $e$, which we add to our matching, then repeat this process, until the number of vertices falls below $n_0$, at which point we can find a perfect matching by a constant-time brute force search. 

The problem with this naive attempt is that as we remove edges, the minimum codegree may become too low to apply Procedure~\ref{mainProc}, and then the process cannot continue. To motivate the solution to this problem, suppose that we have oracle access to a uniformly random edge from some perfect matching. Such an edge is safe, and if we repeatedly remove such random edges, standard large deviation estimates show that with high probability the minimum codegree condition is preserved (replacing $\gamma$ by $\gamma/2$, say). As an aside, we note that since Linear Programming has a polynomial time algorithm, we can construct a distribution $(p_e)$ on the safe edges such that $\sum_{e: x \in e} p_e = k/n$ for every vertex $x$; using this distribution instead of the oracle provides a randomised algorithm for finding a perfect matching. 

Our actual algorithm is obtained by derandomising the oracle algorithm. Instead of a minimum codegree condition, we bound the sum of squares of codegree `deficiencies', which is essentially the condition (codeg) of Setup~\ref{setup}. 
We also need to introduce the vertex degree condition (deg), otherwise we do not have an effective bound on the number of partitions $\Part$ in Lemma~\ref{OmegaPolyTime}. 
Conditions (i) and (ii) in the following lemma effectively serve as proxies for (codeg) and (deg) respectively.
Note that by $H-e$ we mean the $k$-graph obtained from $H$ by deleting all vertices of $e$ and all edges incident to them.

\begin{lemma} \label{FindPMLemma} Suppose that $1/n \ll \eps \ll \gamma \ll 1/k$, that $k \geq 3$ and that $H$ is a $k$-graph on $n$ vertices. For each set $A$ of $k-1$ vertices of $H$ let $t_A = \max(0, (1/k + \gamma)n - d_H(A))$. Suppose that
\begin{enumerate}[(i)]
\item $\sum_{A \in \binom{V(H)}{k-1}} t_A^2 < \eps \gamma^2 n^{k+1}/4 + 3kn^k$,
\item $n^{k-1}/3k! - \delta_1(H) 
+ \sum_{A \in \binom{V(H)}{k-1}} \tfrac{t_A^2}{\sqrt{\eps}\gamma^2 n^2} < \sqrt{\eps} n^{k-1}$,
\item $H$ contains a perfect matching.
\end{enumerate}
Then we can find, in time at most $O(n^{3k^2 - 7k})$, an edge $e \in H$ such that (i), (ii) and (iii) also hold for $H-e$ with $n-k$ in place of $n$.
\end{lemma}

The proof of Lemma~\ref{FindPMLemma} involves messy calculations, so we defer the full proof until Section~\ref{DeferredProofsSection} and for now just describe the idea. 
While $\delta_1(H) > n^{k-1}/3k!$ we only need to maintain condition (i); then an averaging argument shows that the required edge exists. On the other hand, if there is a vertex of small degree we can remove any edge containing it: this so greatly decreases $\sum_{A \in \binom{V(H)}{k-1}} t_A^2$ as to compensate for any further decrease in $\delta_1(H)$.

\medskip

\begin{proof}[Proof of Theorem~\ref{main}.]
We begin by running Procedure~\ref{mainProc} to confirm that $H$ contains a perfect matching; if it does not, then we obtain a $2k(k-3)$-certificate for $H$, which by Theorem~\ref{PMNeccSuff} is a certificate that no perfect matching exists.
If $H$ does contain a perfect matching, then we repeatedly apply Lemma~\ref{FindPMLemma} to delete edges of $H$ (along with their vertices). Initially, condition (iii) of Lemma~\ref{FindPMLemma} holds by assumption, and conditions (i) and (ii) hold as the codegree assumption implies $t_A=0$ for any $(k-1)$-set $A$ and $\delta_1(H) \ge k^{-1} \binom{n}{k-1}$.
Since Lemma~\ref{FindPMLemma} ensures that its conditions are preserved after an edge is deleted, we may repeat until $n$ is too small for further application of Lemma~\ref{FindPMLemma}. 
At this point the number of vertices remaining in $H$ is bounded by a constant depending only on $\eps, \gamma$ and $k$ (this follows from the definition of $\ll$), so we can use the brute-force algorithm to find a perfect matching in the remainder of $H$ in constant time. 
Together with the deleted edges this forms a perfect matching in $H$.
Since $k$ vertices are deleted from $H$ with each application of Lemma~\ref{FindPMLemma}, at most $n/k$ applications are required in total.
Thus the total running time of the algorithm is $O(n^{3k^2 - 7k + 1})$.
\end{proof}

\section{Outlines of the proofs} \label{SketchSection}

In this section we briefly sketch the proof of Theorem~\ref{PMNeccSuff}. 
The easier direction is that if there is a perfect matching then there is no $2k(k-3)$-certificate.
Let $H$ be a $k$-graph as in Setup~\ref{setup} that contains a perfect matching $M^*$. 
To show that there is no $2k(k-3)$-certificate for $H$, 
it suffices to show that every $2k(k-3)$-far full pair $(\Part, L)$ for $H$ is soluble.
To accomplish this, we first show that $L$ is a subgroup of $L_{\max}^\Part$ with index at most $|\Part|$.
Then by a simple application of the pigeonhole principle, we can find a matching $M$ in $H$ of size at most $|\Part|-1$ 
for which $\ib_\Part(M)$ lies in the same coset of $L$ in $L_{\max}^\Part$ as $\ib_\Part(M^*)$. 
Since $\ib_\Part(V(H) \sm V(M^*)) = \0 \in L$ we deduce that $\ib_\Part(V(H) \sm V(M)) \in L$, 
and so $M$ solves $(\Part, L)$. 

Let us now turn our attention to the remaining implication of the theorem,
i.e.\ that if every $2k(k-3)$-far full pair is soluble then there is a perfect matching.
We reduce this to an easier version of the theorem in which we assume that every full pair is soluble.
The proof of this reduction relies on properties of subsequence sums in abelian groups, 
that (roughly speaking) allow us to convert any insoluble full pair into a $2k(k-3)$-far full pair by merging parts.
To prove the easier version, we begin by identifying a canonical partition $\Part$ of $V(H)$:
we choose $\Part$ to be `robustly maximal' with respect to $H$,
which roughly speaking means that it is an approximate divisibility barrier 
with no other approximate divisibility barrier `hidden' inside. 
By reassigning a small number of vertices to different parts 
we can further ensure that every vertex lies in many edges $e \in H$ with $\ib_\Part(e) \in L$. 
We will see that $(\Part, L)$ is a full pair, so by our assumption it has a solution $M$. 
Our key lemma will show that under these conditions $H \sm V(M)$ has a perfect matching, 
which together with $M$ forms a perfect matching in $H$.

We prove the key lemma by reducing it to a partite version, using a random $k$-partition of $H$.
Thus for much of the paper we work with a $k$-partite $k$-graph $H$ whose vertex classes each have size $n$.
Our codegree conditions will apply only to $(k-1)$-sets which contain at most one vertex from any vertex class. 
We will work with the `$\mu$-robust' edge-lattice $L^\mu_\Part(H)$, which is defined similarly to $L_\Part(H)$, except that we require many ($\mu |V(H)|^k$) edges of a given index vector for it to be included in the generating set (see Definition~\ref{def:lattices}).

A key idea in the proof of is that of splitting $H$ into a number of $k$-partite $k$-graphs, 
each of which satisfies a stronger version of the properties of $H$.
More precisely, if almost all partite $(k-1)$-sets of vertices of $H$ have degree at least $n/\ell + \gamma n$ and $L_\Part^\mu(H)$ is transferral-free, then we can find a set of $k$-partite subgraphs whose vertex sets partition $V(H)$, 
such that each satisfies a stronger codegree condition: 
almost all partite $(k-1)$-sets have degree at least $1/(\ell-1)$ proportion of the size of each vertex class.
By induction on $\ell$ we will obtain a perfect matching in each subgraph,
which together give a perfect matching in $H$.

\subsection{An analogue for tripartite 3-graphs} \label{3CaseSketch} To illustrate this idea better, we now outline the proof of an analogous result to Theorem~\ref{PMNeccSuff3case} for $3$-partite $3$-graphs. Indeed, let $H$ be a $3$-partite $3$-graph with vertex classes $V_1, V_2$ and $V_3$ each of size $n$. Suppose that any partite pair in $H$ has codegree at least $(1/3 + \gamma) n$. We shall prove that either $H$ contains a perfect matching, or $V(H)$ can be partitioned into two parts $A$ and $B$ such that every edge of $H$ has an even number of vertices in $A$, but $|A|$ is odd (i.e.\ there is a divisibility barrier).

We begin by applying Theorem~\ref{HypergraphMatching} to $H$. 
This implies that either $H$ contains a perfect matching (in which case we are done) or that there is a partition $\Part$ of $V(H)$ into parts of size at least $(1/3 + \gamma/2)n$ which refines the $3$-partition of $V(H)$ and is such that $L_\Part^\mu(H)$ is incomplete and transferral-free. Note that this means that $\Part$ refines each vertex class $V_i$ into at most two parts; in fact, it is not hard to see that $\Part$ must partition each vertex class $V_i$ into precisely two parts, $W_{i}^1$ and $W_{i}^2$.

We now decompose $H$ into eight $3$-partite subgraphs $(H^{ijk})_{i, j, k \in [2]}$, where $H^{ijk}$ consists of all edges of $H$ which contain one vertex from each of $W_{1}^i, W_{2}^j$ and $W_{3}^k$. So every edge of $H$ lies in precisely one subgraph $H^{ijk}$.
Since $L_\Part^\mu(H)$ is transferral-free, we may assume without loss of generality that almost all of the edges of $H$ lie in one of $H^{111}, H^{122}, H^{212}$ or $H^{221}$ (the remaining subgraphs have very low density).
As described earlier, we now observe that almost all partite pairs in each subgraph have codegree at least $1/2 + \gamma/2$ proportion of the size of the remaining vertex class in that subgraph. 
Indeed, our assumption on $H$ tells us that any pair of vertices $xy$ with $x \in W_{1}^1$ and $y \in W_{2}^1$ has at least $(1/3 + \gamma)n$ neighbours in $V_3$ (i.e.\ vertices $z \in V_3$ such that $xyz \in H$).  
But since $H_{112}$ has very low density, very few such pairs $xy$ can have many neighbours in $W_{3}^2$.
So almost all of these pairs must have $(1/3 + \gamma/2)n \geq (1/2 + \gamma/2)|W_{31}|$ neighbours in $W_{31}$. So $H^{111}$ satisfies a significantly stronger codegree condition than that satisfied by $H$, ignoring the fact that a small number of pairs, which we will call \emph{bad} pairs, fail this codegree condition. A similar argument applies to $H^{122}$, $H^{212}$ and $H^{221}$.

At this point we delete a (small) matching $M$ in $H$ to achieve two aims. Firstly, $M$ will cover all vertices which lie in many bad pairs. Since there are few bad pairs there will only be a small number of such vertices. Secondly, after deleting the vertices covered by $M$ there will be an \emph{even} number of vertices remaining in $W^2 := W_1^2 \cup W_2^2 \cup W^2_3$. This can be done unless $|W^2|$ is odd and every edge of $H$ intersects $W^2$ in an even number of vertices; in this case we are done, and $H$ contains no perfect matching. 

Now we delete the vertices covered by $M$ from $H$; this only slightly weakens the codegree condition on $H^{111}$, $H^{122}$, $H^{212}$ and $H^{221}$ since $M$ does not contain many edges. Then we choose a random partition of the remaining vertices of $H$ into subsets $S_{111}, S_{122}, S_{212}$ and $S_{221}$ under the constraints that 
(i) each subset contains equally many vertices from each vertex class $V_j$, and
(ii) $S_{ijk} \subseteq W_1^i \cup W_2^j \cup W_3^k$. 
Such partitions exist since each part $W_j^i$ has size greater than $n/3$ and $|W^2|$ is even (this is a special case of Proposition~\ref{applyfarkas}). Let $H'_{111}$ be the $3$-partite subgraph of $H$ induced by $S_{111}$, and define $H'_{122}$, $H'_{212}$, and $H'_{221}$ similarly. The fact that the partition was chosen randomly implies that, in each subgraph, almost every partite pair has codegree at least $(1/2 + \gamma/2)n'$, where $n'$ is the number of vertices in each part, and no vertex lies in many bad pairs. 
It follows by~\cite[Theorem 2]{AGS09} (or by Theorem~\ref{HypergraphMatching}) 
that each subgraph contains a perfect matching; together with the deleted matching this yields a perfect matching in $H$.

\subsection{The general case} 

The case $k=3$ is particularly simple because there is only one maximal divisibility barrier; 
recall that for $k \geq 4$, Construction~\ref{nopm} shows that a naive generalisation of Theorem~\ref{PMNeccSuff3case}
does not hold. For the general case of Theorem~\ref{PMNeccSuff}, we use the ideas outlined in Section~\ref{3CaseSketch},
but there are additional complications. Firstly, there are now many possibilities for the partition $\Part$ returned by applications
of Theorem~\ref{HypergraphMatching}. Secondly, multiple applications of the above technique are required to successively
improve the codegree condition on the $k$-graphs into which $H$ is decomposed. 

These complications will be handled by our partite key lemma (Lemma~\ref{kPartiteLemma}), which can be seen as a
variation on Theorem~\ref{HypergraphMatching} specialised to $k$-partite $k$-graphs with vertex classes of size $n$ in
which almost all partite $(k-1)$-sets have codegree at least $(1/\ell + \gamma)n$, with two key differences.
One is that, instead of the condition of Theorem~\ref{HypergraphMatching} that $L = L_\Part^\mu(H)$ must be complete, 
we now only require that $\ib_\Part(V(H)) \in L$. 
The other is that, whereas the condition of Theorem~\ref{HypergraphMatching} must hold for any partition $\Part$ of $V(H)$ into sufficiently large parts, we now only require that $\ib_\Part(V(H)) \in L$ for a single `canonical' partition $\Part$ 
which meets two additional requirements.

Firstly, we require that every vertex must lie in many edges $e \in H$ with $\ib_\Part(e) \in L$. 
This condition can be seen as ensuring that each vertex of $H$ lies in the `correct part' of $\Part$. 
For example, consider the extremal example described in Construction~\ref{RRSconstruct}. 
If we fix some small $\mu > 0$ and move a single vertex $v$ from part $A$ to part $B$ (but do not change the edge set of $H$), 
then the only edges whose index changes are the fewer than $n^{k-1}$ edges which contain~$v$. 
So $L^\mu_\Part(H)$ is unchanged, and $H$ doesn't have a perfect matching (since $H$ itself is unchanged), 
but we now have $\ib_\Part(V(H)) \in L^\mu_\Part(H)$, due to $v$ being in the `wrong part' of $\Part$.
Secondly, we must assume that $L$ is transferral-free. 
However this requirement, while necessary, is not sufficient, as the following example will show.

\begin{construct} \label{NestedConstruct}
Fix $k \geq 5$ and let $\Part = \{W_{1}, W_{2}\}$ be a partition of $V$. 
Let $\Qart$ be a refinement of $\Part$ which divides $W_{1}$ into $V_{11}$ and $V_{12}$ and $W_2$ into $V_{21}$ and $V_{22}$.
Suppose that $|W_1|$ is even, but that $|V_{11} \cup V_{21}|$ is odd.
Let $H$ be the $k$-graph on $V$ whose edges are precisely the $k$-subsets of $V$ 
which contain an even number of vertices in $W_1$ and an even number of vertices in $V_{11} \cup V_{21}$.
\end{construct}

In Construction~\ref{NestedConstruct}, $L_\Part^\mu(H)$ is transferral-free,
$\delta_{k-1}(H) \geq (1/k + \gamma)n$ and $\ib_\Part(V(H)) \in L_\Part^\mu(H)$.
Thus the conditions relating to $\Part$ do not preclude the existence of a perfect matching, 
but the conditions relating to $\Qart$ do.
Indeed, $H$ cannot contain a perfect matching, since it is a subgraph of the $k$-graph described in Construction~\ref{RRSconstruct}, where $A = V_{11} \cup V_{21}$ and $B = V_{12} \cup V_{22}$. 
To avoid this kind of situation, we insist that there is no strict refinement $\Qart$ of $\Part$ 
into not-too-small parts such that $L_{\Qart}^\mu(H)$ is transferral-free. 
Note that the trivial partition into a single part satisfies this requirement if and only if $H$ is not close to a divisibility barrier.
In fact, we require the stronger property that for some $\mu' \gg \mu$ (i.e.\ even at a much weaker `detection threshold') the lattice $L_\Qart^{\mu'}(H)$ is not transferral-free for any strict refinement $\Qart$ of $\Part$ into not-too-small parts.
This property is `robust' in the sense that when we delete a small number of vertices from $H$ (as we will need to do), the property is preserved, albeit with weaker constants. (Section~\ref{RobMaxSection} is devoted to the study of robust maximality.)
If $\Part$ satisfies both of these conditions, then Lemma~\ref{kPartiteLemma} states that $H$ must contain a perfect matching. 

We prove Lemma~\ref{kPartiteLemma} by induction on $\ell$ by a similar argument to that in Section~\ref{3CaseSketch} (roughly speaking, there we reduced the $\ell = 3$ case to the $\ell = 2$ case, for $k = 3$). That is, we apply Theorem~\ref{HypergraphMatching} to yield a perfect matching (in which case we are done), or a partition $\Part$ of $V(H)$ into parts of size at least $(1/\ell + \gamma) n$ such that $L^\mu_{\Part}(H)$ is incomplete and transferral-free. For the base case of the induction we observe that the latter outcome is impossible for $\ell = 2$. 
Next, for each $\ib \in I^\mu_\Part(H)$, we define $H_\ib$ to be the `canonical' induced subgraph of $H$ on the union of the parts $W \in \Part$ such that $i_W = 1$, and $\Part_\ib$ to be the restriction of $\Part$ to $V(H_\ib)$.
By a similar argument to the previous section (formalised in Proposition~\ref{properties}), 
we find that each $H_\ib$ satisfies a codegree condition similar to that on $H$, but with $\ell - 1$ in place of $\ell$. 
Our aim is then to find vertex-disjoint subgraphs $\hat{H}_\ib \subseteq H_\ib$ whose vertex sets partition $V(H)$, each
of which satisfies the conditions of Lemma~\ref{kPartiteLemma} 
with the stronger codegree condition ($\ell - 1$ in place of $\ell$). 
We can then apply the inductive hypothesis to find a perfect matching in each $\hat{H}_\ib$; 
together these form a perfect matching in $H$.

To do this, we first take a robustly maximal refinement $\Qart_\ib$ of $\Part_\ib$ for each $\ib$,
and delete a small matching that covers all `bad' vertices.
We then need to ensure that our partition into subgraphs $H_\ib$ satisfies 
$\ib_{\Qart_\ib}(V(\hat{H}_\ib)) \in L_{\Qart_\ib}^{\mu}(H_\ib)$ for each $\ib$;
accomplishing this is the most technical part of the proof.
Observe that whether this condition is satisfied depends only on the values 
of $|V(\hat{H}_\ib) \cap Y|$ for parts $Y \in \Qart_\ib$.
Thus we need to choose the sizes of the intersections $V(\hat{H}_\ib) \cap Z$ for parts $Z \in \Qart^\cap$,
where $\Qart^\cap$ is the `meet' of the partitions $\{\Qart_\ib \mid \ib \in I\}$.
We achieve this in three stages.
Firstly, in Claim~\ref{rhoclaim}, we choose rough targets for the size of each $V(\hat{H}_\ib)$; here we use a geometric method which relies on Farkas' Lemma (Theorem~\ref{farkas}).
Secondly, in Claim~\ref{cupclaim}, we choose how many vertices each $\hat{H}_\ib$ will take from each part of the `join' $\Qart^\cup$ of the partitions $\Qart_\ib$ (this step deals with the problem illustrated by Construction~\ref{NestedConstruct}).
Thirdly, in Claim~\ref{capclaim}, we use Baranyai's Matrix Rounding Theorem 
to refine this choice to obtain the intersection sizes we require.
Finally, we choose the vertex set of each $\hat{H}_\ib$ at random with the given number of vertices from each part of $\Qart^\cap$. 
In Claim~\ref{choiceclaim} we demonstrate that with high probability this random selection does indeed give subgraphs $\hat{H}_\ib$ to which we can apply the inductive hypothesis, completing the proof.

\section{Preliminaries} \label{PrelimSection}

This section contains theoretical background needed for the rest of the paper, 
organised into the following subsections:
\ref{subsec:pil} Partitions, index vectors and lattices,
\ref{HypMatchThrySec} Hypergraph matching theory,
\ref{subsec:cg} Convex geometry,
\ref{subsec:baranyai} Baranyai's Matrix Rounding Theorem,
\ref{subsec:conc} Concentration of probability.

\subsection{Partitions, index vectors and lattices} \label{subsec:pil}

We will frequently speak of a partition $\Qart$ of a vertex set $V$. We use this term in a slightly non-standard way to mean a family of pairwise-disjoint subsets of $V$ whose union is $V$ (so it is possible for a part to be empty, though this will rarely be the case). Furthermore, we implicitly fix an order on the parts of the partition, so we may consistently speak of, for example, the $i$th part of $\Qart$. Given $V' \sub V$, the \emph{restriction} of $\Qart$ to $V'$, denoted $\Qart[V']$, is the partition of $V'$ 
with parts $X \cap V'$ for $X \in \Qart$. 

Many of our results will apply specifically to partite $k$-graphs. Let $\Part$ partition a set $V$. 
Then we say that a set $S \subseteq V$ is \emph{$\Part$-partite} if $S$ has at most one vertex in any part of $\Part$. 
We say that a hypergraph $H$ on $V$ is \emph{$\Part$-partite} if every edge of $H$ is $\Part$-partite. 
In this case we refer to the parts of $\Part$ as the \emph{vertex classes} of $H$.
Usually, we will consider $k$-graphs $H$ which are \emph{$k$-partite},
i.e.\ $\Part$-partite for some partition $\Part$ of $V(H)$ into $k$ parts.

Given a $k$-graph $H$ and a partition $\Part$ of $V(H)$, the lattice $L_\Part(H)$ in Definition~\ref{def:lattices0} can be seen as `detecting' where edges of $H$ lie with respect to $\Part$. However, the information conveyed by $L_\Part(H)$ is by itself insufficient, as shown by the $k$-graph formed in Construction~\ref{nopm}. 
Indeed, in that instance $L_\Part(H)$ was complete, but some index vectors did not represent enough edges:
specifically, $H$ did not contain two disjoint edges with different numbers of vertices in $A$ and $B$ modulo $3$.
Thus in the proof of Theorem~\ref{PMNeccSuff} we will frequently want to know which index vectors are represented by many edges; this is achieved by the following definition.

\begin{defin}[Robust edge-lattices] \label{def:lattices} Let $H$ be a $k$-graph and $\Part$ be a partition of $V(H)$ into $d$ parts. Then for any $\mu > 0$,
\begin{enumerate}[(i)]
\item $I_\Part^\mu(H)$ denotes the set of all $\ib \in \Z^d$ such that at least $\mu |V(H)|^k$ edges $e \in H$ have $\ib_\Part(e) = \ib$.
\item $L^\mu_\Part(H)$ denotes the lattice in $\Z^d$ generated by $I_\Part^\mu(H)$.
\end{enumerate}
\end{defin}

The constant $\mu$ can be viewed as the `detection threshold': 
it specifies the number of edges of a given index which are required for this index to contribute to $L_\Part^\mu(H)$.
For any $\mu < \mu'$ it is clear that $L_\Part^{\mu'}(H) \subseteq L_\Part^\mu(H) \subseteq L_\Part(H)$.
Next we show that robust edge-lattices are insensitive to small perturbations. 

\begin{defin} \label{close}
Let $\Part$ be a partition of a set $V$ of size $n$ and $H$ be a $k$-graph on $V$.
\begin{enumerate}[(i)]
\item We say that a $k$-graph $H'$ on a set $V' \sub V$ is \emph{$\alpha$-close} to $H$
if we have $|V'| \ge (1-\alpha)n$ and $|H \triangle H'| \leq \alpha n^k$.
\item We say that a partition $\Part'$ of a set $V' \sub V$ is \emph{$\alpha$-close} to $\Part$
if for some $d$ we have $|\Part'|=|\Part|=d$ and $\sum_{i \in [d]} |V_i \triangle V'_i| \le \alpha n$, 
where $\Part = (V_1, \ldots, V_d)$ and $\Part' = (V'_1, \ldots, V'_d)$.
\end{enumerate}
\end{defin}

Note that Definition~\ref{close}(i) is asymmetric; that is, it does not imply that $H$ is $\alpha$-close to $H'$.

\begin{prop} \label{similarlattices}
Under the setup of Definition~\ref{close}, 
we have $I_\Part^\mu(H) \subseteq I_{\Part'}^{\mu-2\alpha}(H')$, 
so $L_\Part^\mu(H) \subseteq L_{\Part'}^{\mu-2\alpha}(H')$.
Also, if $\mu \le 1/k$ then $I_{\Part'}^\mu(H') \subseteq I_{\Part}^{\mu-3\alpha}(H)$, 
so $L_{\Part'}^\mu(H') \subseteq L_{\Part}^{\mu-3\alpha}(H)$.
\end{prop}

\begin{proof}
Consider $\ib \in I_\Part^\mu(H)$. By definition, at least $\mu n^k$ edges $e \in H$ have $\ib_\Part(e) = \ib$.
Of these edges, at most $\alpha n^k$ are not in $H'$, and at most $\alpha n^k$ of those in $H'$ have $\ib_{\Part'}(e) \ne \ib$.
Therefore $\ib \in I_{\Part'}^{\mu-2\alpha}(H')$. 
Similarly, consider $\ib \in I_{\Part'}^\mu(H')$.
At least $\mu (n-\alpha n)^k > (\mu-\alpha)n^k$ edges $e \in H'$ have $\ib_{\Part'}(e) = \ib$,
at most $\alpha n^k$ are not in $H$, and at most $\alpha n^k$ have $\ib_{\Part}(e) \ne \ib$.
Therefore $\ib \in I_{\Part}^{\mu-3\alpha}(H)$. 
\end{proof}

Let $\Part$ and $\Part'$ be partitions of a set $V$. 
We say that $\Part'$ \emph{refines} $\Part$ if every part of $\Part'$ is a subset of a part of $\Part$.

\begin{defin} \label{def:Restriction}
Let $\Part$ and $\Part'$ be partitions of a set $V$ such that $\Part$ refines $\Part'$. 
Let $\ib$ be an index vector with respect to $\Part$. 
Then the \emph{projection} $(\ib \mid \Part')$ of $\ib$ on $\Part'$ is defined by
$$(\ib \mid \Part')_W = \sum_{X \in \Part, X \subseteq W} \ib_X$$
for each $W \in \Part'$. We also write $(I \mid \Part') = \{ (\ib \mid \Part') : \ib \in I\}$
when $I$ is a set of index vectors with respect to $\Part$. 
\end{defin}

As an example under the setup of Definition~\ref{def:Restriction}, note that 
if $S \subseteq V$ then $(\ib_{\Part}(S) \mid \Part') = \ib_{\Part'}(S)$.

Our next result concerns the behaviour of edge-lattices with respect to projection.
An informal statement is that if $\Qart$ refines $\Part$, then the projection of a robust edge-lattice in $\Z^\Qart$ is contained in a robust edge-lattice in $\Z^\Part$,
and a `weak converse' holds, namely that any vector in a robust edge-lattice in $\Z^\Part$ is the projection of some vector in a robust edge-lattice in $\Z^\Qart$.

\begin{prop} \label{InverseRestriction} Let $\Part$ and $\Qart$ be partitions of a set $V$ of $n$ vertices such that $\Qart$ refines $\Part$, and let $H$ be a $k$-graph on $V$. 
\begin{enumerate}[(i)]
\item If $\ib \in L_{\Qart}^\mu(H)$ then $(\ib \mid \Part) \in L_{\Part}^\mu(H)$.
\item For any $\ib \in L_{\Part}^\mu(H)$ there exists $\ib' \in L_{\Qart}^{\mu/m}(H)$ such that $(\ib' \mid \Part) = \ib$, where $m := (k+1)^{|\Qart| - |\Part|}$.
\end{enumerate}
\end{prop}

\begin{proof} For (i), consider $\ib \in I_{\Qart}^\mu(H)$. 
At least $\mu n^k$ edges $e \in H$ have $\ib_{\Qart}(e) = \ib$. 
Each such edge has $\ib_{\Part}(e) = (\ib \mid \Part)$, so $(\ib \mid \Part) \in I_{\Part}^\mu(H)$. 
Since projection is linear the result follows.

For (ii), consider $\ib \in I_{\Part}^\mu(H)$. At least $\mu n^k$ edges $e \in H$ have $\ib_\Part(e) = \ib$. 
Note that there are at most $m$ index vectors $\ib'$ with respect to $\Qart$ which satisfy $(\ib' \mid \Part) = \ib$, since for every part $X \in \Part$ which contains $r$ parts $Y_1, \ldots, Y_r$ of $\Qart$ there are at most $(k+1)^{r-1}$ ways of partitioning $i_X$ into $i'_{Y_1}, \ldots, i'_{Y_r}$. 
So by the pigeonhole principle there is some $\ib'$ with $(\ib' \mid \Part) = \ib$ for which at least $\mu n^k/m$ edges $e \in H$ have $\ib_{\Qart}(e) = \ib'$, that is, $\ib' \in I_{\Qart}^{\mu/m}(H)$. Again the result follows by linearity of projection.
\end{proof}

\subsection{Hypergraph matching theory} \label{HypMatchThrySec} 
In this subsection we describe a central theorem from~\cite{KM11p} that plays a key role in this paper. 
First we need the following definitions. A \emph{$k$-complex $J$} is a hypergraph such that $\emptyset \in J$, every edge of $J$ has size at most $k$, and $J$ is closed downwards, that is, if $e \in J$ and $e' \subseteq e$ then $e' \in J$. We write $J_r$ to denote the $r$-graph on $V(J)$ formed by edges of size $r$ in $J$. Also, we use the following notion of partite degree in a multipartite $k$-complex from~\cite{KM11p}. Let $\Part$ partition a set $V$ into $k$ parts $V_1, \dots, V_k$,
and let $J$ be a $\Part$-partite $k$-complex on $V$. For each $0 \leq j \leq k-1$ the
\emph{partite minimum $j$-degree} $\delta^*_j(J)$ is defined to be the largest $m$ such that any $j$-edge $e$ has
at least $m$ extensions to a $(j+1)$-edge in any part not used by $e$, i.e.\
$$\delta^*_j(J) := \min_{e \in J_j} \min_{i: e \cap V_i = \emptyset} |\{v \in V_i : e \cup \{v\} \in J\}|.$$
The \emph{partite degree sequence} is then $\delta^*(J) = (\delta_0^*(J), \dots, \delta_{k-1}^*(J))$.
Note that we suppress the dependence on $\Part$ in our notation: this will always be clear from the context. 
Minimum degree sequence conditions on a $k$-complex take the form $\delta^*(J) \geq (a_0, \dots, a_{k-1})$, where the inequality is to be interpreted pointwise. Finally, observe that if $J$ is a $\Part$-partite $k$-complex on $V$, and $\Qart$ is a partition of $V$ which refines $\Part$ into $d$ parts, then every edge $e \in J_k$ has $(\ib_\Qart(e) \mid \Part) = \ib_\Part(e) = \1$. So we say that a lattice $L \subseteq \Z^d$ is \emph{complete with respect to $\Qart$} if $\ib \in L$ for every $\ib \in \Z^d$ with $(\ib \mid  \Part) = \1$, and \emph{incomplete with respect to $\Qart$} otherwise. The following is a simplified form of~\cite[Theorem 2.13]{KM11p}.

\begin{thm}\cite{KM11p}\label{HypergraphMatching}
Suppose that $1/n \ll \mu \ll \gamma, 1/k$. 
Let~$\Part'$ partition a set $V$ into parts $V_1, \dots, V_k$ each of size~$n$.
Suppose that $J$ is a $\Part'$-partite $k$-complex on $V$ with $$\delta^*(J) \geq \left(n, \left(\frac{k-1}{k} + \gamma \right)n, \left(\frac{k-2}{k} + \gamma \right) n, \dots, \left(\frac{1}{k} + \gamma \right) n\right).$$
Then either $J_k$ contains a perfect matching, or $J_k$ is close to a divisibility barrier,
in that there is some partition $\Part$ of $V(J)$ into $d \le k^2$
parts of size at least  $\delta^*_{k-1}(J) - \mu n$ such that $\Part$ refines $\Part'$
and $L^\mu_\Part(J_k)$ is incomplete with respect to $\Part'$ and transferral-free.
\end{thm}

\subsection{Convex geometry} \label{subsec:cg}
Given points $\vb_1, \dots, \vb_r \in \R^d$, we define their {\em positive cone} as
$$PC(\{\vb_1, \dots, \vb_r\}) := \bracc{\sum_{j \in [r]} \lambda_{j} \vb_j : \lambda_1, \dots, \lambda_r \geq 0}.$$
A classical result, commonly known as Farkas' Lemma, shows that any point $\vb$ outside of the positive cone of these points 
can be separated from them by a separating hyperplane.

\begin{thm}[Farkas' Lemma] \label{farkas}
Suppose $\vb \in \R^d \sm PC(Y)$ for some finite set $Y \subseteq \R^d$.
Then there is some $\ab \in \R^d$ such that $\ab \cdot \yb \ge 0$ for every $\yb \in Y$ and $\ab \cdot \vb < 0$.
\end{thm} 
 
We also need the following (slightly rephrased) simple proposition from~\cite{KM11p}.

\begin{prop}[\cite{KM11p}, Proposition 4.8] \label{vector_as_small_sum}
Suppose that $1/s \ll 1/r, 1/d$. Let $X \subseteq \Z^d \cap B^d(\0, r)$, and let $L_X$ 
be the sublattice of $\Z^d$ generated by $X$. Then for any vector 
$\xb \in L_X \cap B^d(\0, r)$ we may choose integers~$a_\ib$ with $|a_\ib| \leq s$ 
for each $\ib \in X$ such that $\xb = \sum_{\ib \in X} a_\ib \ib$.
\end{prop}

\subsection{Baranyai's Matrix Rounding Theorem} \label{subsec:baranyai}

The proof of Lemma~\ref{kPartiteLemma} (on which the key lemma, Lemma~\ref{NonPartiteLemma}, relies) will use Baranyai's Matrix Rounding Theorem~\cite{B}
(see also~\cite[Theorem 7.5]{VLW}).

\begin{thm} (Baranyai's Matrix Rounding Theorem) \label{Baranyai} 
Let $A$ be a real matrix. Then there exists an integer matrix $B$ whose entries, row sums, column sums and the sum of all the entries are the entries, row sums, column sums and the sum of all the entries respectively of $A$, each rounded either up or down.
\end{thm}

\subsection{Concentration of probability} \label{subsec:conc}

We will need the following inequalities, known as Chernoff bounds, as applied to sums of Bernoulli random variables
(i.e. random variables which take values in $\{0, 1\}$) and hypergeometric random variables. The hypergeometric random
variable $X$ with parameters $(N,m,n)$ is defined as $X = |T \cap S|$, where $S \subseteq [N]$ is a fixed set of size
$m$, and $T \subseteq [N]$ is a uniformly random set of size $n$. If $m=pN$ then $X$ has mean $pn$. 
The following is~\cite[Theorem 2.8]{JLR}.

\begin{lemma} \label{VaryingChernoff} 
Let $X$ be a sum of independent Bernoulli random variables and $0<a<3/2$. 
Then $\Prob(|X - \Exp X| \geq a\Exp X) \le 2\exp{(-\frac{a^2}{3}\Exp X)}$.
\end{lemma}

\begin{coro} \label{SumOfHyper}
Let $X$ be a sum of independent hypergeometric random variables and $0<a<3/2$. 
Then $\Prob(|X - \Exp X| \geq a\Exp X) \le 2\exp{(-\frac{a^2}{3}\Exp X)}$.
\end{coro}

\begin{proof} We follow Remark 2.11 of~\cite{JLR}. By Lemma 1 of~\cite{VM}, each of the hypergeometric random variables may be expressed as a sum of independent (but not identically distributed) Bernoulli random variables. Hence Lemma~\ref{VaryingChernoff} implies the result. \end{proof}

We will also need a form of the well-known Azuma-Hoeffding inequality. A sequence $(X_0, \dots, X_n)$ of random
variables is a \emph{martingale} if $\Exp(|X_j|)$ is finite and $\Exp(X_j | X_0, \dots, X_{j-1}) = X_{j-1}$ for any $j \in [n]$. 
The following is (a weaker form of)~\cite[Theorem 2.25]{JLR}.

\begin{thm} \label{Azuma} 
Let $A_0, A_1, \ldots, A_n$ be a martingale such that $|A_{i} - A_{i-1}| \leq C$ for every $i \in [n]$. 
Then $\Prob(|A_n - A_0| > tC) \leq 2\exp\left(-t^2/2n\right)$ for any $t > 0$.
\end{thm}

We will apply Theorem~\ref{Azuma} via the following corollary.

\begin{coro} \label{ApplyAzuma}
Let $c>0$ and $H$ be a $k$-graph on $n$ vertices with $|H| \ge cn^k$.
Let $\Qart$ be a partition of $V(H)$ and let $(n_Z)_{Z \in \Qart}$ be integers such that $c|Z| \le n_Z \le |Z|$ for each $Z \in \Qart$.
Suppose that $S \sub V(H)$ is chosen uniformly at random
subject to the condition that $|S \cap Z| = n_Z$ for each $Z \in \Qart$.
Then with high probability $|H[S]| = (1 \pm c)\mb{E}|H[S]|$.
\end{coro}

\begin{proof}
We consider each $S \cap Z$ to be chosen as the first $n_Z$ elements of a random permutation $\pi_Z$ of $Z$.
We apply Theorem~\ref{Azuma} with $A_i = \mb{E}(|H[S]| \mid \mc{F}_i)$, 
where $\mc{F}_i$ is the algebra of events generated by revealing
the values of the permutations on the first $i$ elements of $V(H)$.
Note that $A_0 = \mb{E}|H[S]|$ and $A_n = |H[S]|$.
Also, $|H[S]|$ is `$2n^{k-1}$-Lipschitz', in that its value can change by at most $2n^{k-1}$
when any transposition is applied to any of the permutations.
It is not hard to deduce that $|A_{i} - A_{i-1}| \leq 2n^{k-1}$ for all $i \in [n]$. 
Finally, we claim that $\mb{E}|H[S]| \ge (c/2)^{k+1} n^k$.
Indeed, at least $c n^k/2$ edges of $H$ have at most one vertex 
within any part of size less than $cn/2$, and for any such edge $e$,
we have $\mb{P}(e \sub S) \ge (c/2)^k$. 
So by Theorem~\ref{Azuma}, applied with $C=2n^{k-1}$ and $t = (c/2)^{k+2}n$,
with high probability $|H[S]|$ and $\mb{E}|H[S]|$ do not differ by more than $c\mb{E}|H[S]|$.
\end{proof}

\section{Robust Maximality} \label{RobMaxSection}

Given a $k$-graph $H$, there may be many approximate divisibility barriers,
i.e.\ partitions $\Part$ of $V(H)$ such that $L_\Part^\mu(H)$ is incomplete for some constant $\mu>0$;
for example, any refinement of such a partition has this property.
In this section we will identify a canonical such partition, which is `robustly maximal'.
The definition at which we will arrive has a number of advantages.
Firstly, any approximate divisibility barrier in $H$ will imply the existence of at least one robustly maximal partition $\Part$.
Secondly, we can give a condition which implies the existence of a perfect matching in $H$
and refers to index vectors with respect to $\Part$ alone.
Thirdly, as the term `robustly' suggests, the property is insensitive to small modifications.
All of these properties will be stated precisely and proved later in the section.

To build up to the definition, we first discuss one way to avoid taking a partition that is too fine.
Recall that a lattice $L \subseteq \Z^d$ is \emph{transferral-free} 
if $L$ does not contain any difference of unit vectors $\ub_i - \ub_j$ with $i,j \in [d]$ and $i \neq j$.
Given any approximate divisibility barrier, we can repeatedly merge parts to obtain one 
with a transferral-free robust edge-lattice, which we may consider to be its `simplest' version.
Indeed, let $H$ be a $k$-graph and consider a partition $\Part$ of $V(H)$ 
such that $L_\Part^\mu(H)$ is incomplete 
and contains $\ub_X - \ub_Y$, for some distinct parts $X$, $Y$ of $\Part$.
Let $\Part'$ be the partition formed from $\Part$ by merging the parts $X$ and $Y$.
Then $L_{\Part'}^{(k+1)\mu}(H)$ is also incomplete. Indeed, if this is not the case,
then for any $\ib \in L_{\max}^\Part$ we have $(\ib \mid \Part') \in L_{\Part'}^{(k+1)\mu}(H)$,
so by Proposition~\ref{InverseRestriction}(ii) there is $\ib' \in L_\Part^\mu(H)$ with $(\ib' \mid \Part') = (\ib \mid \Part')$.
Since $\ub_X - \ub_Y \in L_\Part^\mu(H)$ we deduce that $\ib \in L_\Part^\mu(H)$.
However, this implies that $L_\Part^\mu(H)$ is complete, which is a contradiction.

Next we recall from Construction~\ref{NestedConstruct} that a transferral-free approximate divisibility barrier
may have another one `hidden' inside it, so we will require a maximality property to rule this out.
A first attempt might be to say that $\Part$ should be maximal such that $L_\Part^\mu(H)$ is transferral-free, 
but then we would obtain a rather fragile property that is sensitive to small modifications of $\Part$, $\mu$ and $H$.
Indeed, in the course of our proof we will need to remove small matchings from $H$ 
in order to cover vertices which are exceptional in various ways.
We also want the property to be preserved with high probability 
by taking an induced subgraph on a randomly chosen set of vertices.
These requirements lead naturally to the following key definition.

\begin{defin} \label{def:RobustMax} 
Let $H$ be a $k$-graph. 
We say that a partition $\Part$ of $V(H)$ is 
\emph{$(c, c', \mu, \mu')$-robustly maximal} with respect to $H$ if
\begin{enumerate}[(i)]
\item $L_\Part^\mu(H)$ is transferral-free and 
all parts of $\Part$ have size at least $c|V(H)|$,
\item $L_{\Part'}^{\mu'}(H)$ is not transferral-free for any partition $\Part'$ of $V(H)$ 
with parts of size at least $c'|V(H)|$ that strictly refines $\Part$.
\end{enumerate}
\end{defin}

We will see in the key lemma that robustly maximal partitions are `canonical', 
in that they capture all necessary information on approximate divisibility barriers.
The next proposition allows us to refine any transferral-free approximate divisibility barrier
to obtain a robustly maximal partition. 
(Note that if $\Part$ is trivial and $\Part'=\Part$ 
then $H$ does not have any approximate divisibility barrier).

\begin{prop} \label{RobustExistence} Let $k \geq 2$ be an integer and $c > 0$ be a constant. 
Let $s = \lfloor 1/c \rfloor$ and fix constants $0 < \mu_1 < \dots < \mu_{s+1}$ and $c_1, \dots, c_{s+1} \geq c$. 
Suppose that $H$ is a $k$-graph on $n$ vertices, and $\Part$ is a partition of $V(H)$ with parts of size at least $c_1 n$ such that $L_\Part^{\mu_1}(H)$ is transferral-free. Then there exists $t \in [s]$ and a partition $\Part'$ of $V(H)$ that refines $\Part$ and is $(c_t, c_{t+1}, \mu_t, \mu_{t+1})$-robustly maximal with respect to $H$.
\end{prop}

\begin{proof} 
We start by setting $\Part^{(1)} = \Part$, then we proceed iteratively. 
At step $t$, if there is a partition $\Part^+$ which strictly refines $\Part^{(t)}$ into parts of size at least $c_{t+1} n$
such that the lattice $L_{\Part^+}^{\mu_{t+1}}(H)$ is transferral-free, then we let $\Part^{(t+1)} = \Part^+$.
Otherwise we terminate with output $\Part^{(t)}$. 
By definition each $\Part^{(t)}$ refines $\Part$ and has parts of size at least $c_t n$. 
Furthermore, if the algorithm terminates at time $t$, then $L_{\Part^{(t)}}^{\mu_{t}}(H)$ is transferral-free by choice of $\Part^{(t)}$, but for any $\Part^+$ which strictly refines $\Part^{(t)}$ into parts of size at least $c_{t+1} n$ the lattice $L_{\Part^+}^{\mu_{t+1}}(H)$ is not transferral-free. 
That is, $\Part^{(t)}$ is $(c_t, c_{t+1}, \mu_t, \mu_{t+1})$-robustly maximal with respect to $H$. 
It remains to show that this algorithm must terminate with $t \leq 1/c$.
Indeed, $|\Part^{(i)}| > |\Part^{(i-1)}|$ for $i \ge 2$, 
but $\Part^{(t)}$ can have at most $1/c$ parts, 
since each part has size at least $c_{t} n \geq cn$. 
\end{proof}

Next we show that robust maximality is insensitive to small modifications.

\begin{lemma} \label{RobustInherit} 
Suppose that $1/n \ll \mu, \mu', \alpha \ll c', c, 1/k$. 
Let $\Part$ be a partition of a set $V$ of size $n$ and $H$ be a $k$-graph on $V$.
Let $V' \subseteq V$, let $\Part'$ be a partition of $V'$ that is $\alpha$-close to $\Part$, and let $H'$ be a $k$-graph on $V'$ that is $\alpha$-close to $H$.
If $\Part$ is $(c, c', \mu, \mu')$-robustly maximal with respect to $H$
then $\Part'$ is $(c-\alpha, c'+3\alpha, \mu+3\alpha, \mu'-2\alpha)$-robustly maximal with respect to $H'$.
\end{lemma} 

\begin{proof}
Since $\Part$ has parts of size at least $cn$, and $\Part'$ is $\alpha$-close to $\Part$, each part of $\Part'$ has size at least $(c - \alpha)n$.
Also, Proposition~\ref{similarlattices} implies that $L_{\Part'}^{\mu+3\alpha}(H') \sub L_\Part^\mu(H)$.
Since $L_\Part^\mu(H)$ is transferral-free, so is $L_{\Part'}^{\mu+3\alpha}(H')$. 
All that remains is to show that there is no partition $\Part^+$ of $V'$ which strictly refines $\Part'$ into parts of size at least $(c'+3\alpha)|V'|$ such that the lattice $L_{\Part^+}^{\mu'-2\alpha}(H')$ is transferral-free. 

Suppose for a contradiction that some such $\Part^+$ exists, with $d$ parts $\Part^+_1, \dots, \Part^+_d$. We use $\Part^+$ to form a partition $\Part^*$ of $V$ whose parts $\Part^*_1, \dots, \Part^*_d$ correspond to those of $\Part^+$. Note that all but at most $\alpha n$ vertices $u \in V(H)$ are members of $V(H')$ which lie in the same part of $\Part'$ as $\Part$. We include each such vertex $u$ in the part $\Part^*_j$ corresponding to the part $\Part^+_j$ of $\Part^+$ which contains $u$. Next, we assign each of the at most $\alpha n$ remaining vertices $v \in V(H)$ to an arbitrary part of $\Part^*$ which is a subset of the part of $\Part$ containing $v$. Observe that $\Part^*$ is then a refinement of $\Part$. Furthermore, $\Part^+$ is $\alpha$-close to $\Part^*$, so $L_{\Part^*}^{\mu'}(H) \subseteq L_{\Part^+}^{\mu'-2\alpha}(H')$ by Proposition~\ref{similarlattices}, and this implies that
 $L_{\Part^*}^{\mu'}(H)$ is transferral-free. 
However, $\Part^*$ has parts of size at least $(c'+3\alpha)|V'|- \alpha n \geq c'n$, so the existence of $\Part^*$ contradicts the robust maximality of $\Part$ with respect to $H$, completing the proof.
\end{proof}

Our next result shows that with robust maximality
we can improve the `weak converse' of projection from Proposition~\ref{InverseRestriction} 
to a (genuine) converse (with weaker parameters). 

\begin{prop} \label{robmaxinverse} 
Suppose that $1/n \ll \mu \ll \mu' \ll c', c, 1/k$. 
Let $H$ be a $k$-graph on $n$ vertices,  
$\Part$ be a partition of $V(H)$ that is $(c, c', \mu, \mu')$-robustly maximal with respect to $H$, 
and $\Qart$ be a partition of $V(H)$ which refines $\Part$ into parts of size at least $c'n$. 
Suppose that $\ib \in L_\Part^{\mu'}(H)$. 
Then $\ib' \in L_\Qart^\mu(H)$ for any index vector $\ib'$ 
with respect to $\Qart$ such that $(\ib' \mid \Part) = \ib$. 
\end{prop}

\begin{proof}
Let $p$ be the number of parts of $\Part$. 
We prove the following statement by induction on $q$: if $\Qart$ is a refinement of $\Part$ into $q$ parts of size at least $c'n$ and $\ib'$ is an index vector with respect to $\Qart$ such that $(\ib' \mid \Part) = \ib$ then $\ib' \in L_\Qart^{\mu'/(k+1)^{q-p}}(H)$. 
The base case $q = p$ is trivial, since then we must have $\Qart = \Part$, so $\ib' = \ib \in L_\Qart^{\mu'}(H)$ by assumption. 
Assume therefore that we have proved the statement for any refinement $\Qart'$ of $\Part$ into $q-1$ parts of size at least $c'n$, and that $\Qart$ refines $\Part$ into $q$ parts of size at least $c'n$. 
Since $\Part$ is $(c, c', \mu, \mu')$-robustly maximal we know that $L_\Qart^{\mu'}(H)$ is not transferral-free, and so there are distinct parts $X, X' \in \Qart$ such that $\ub_X - \ub_{X'} \in L_\Qart^{\mu'}(H)$. 
Note that $(\ub_X - \ub_{X'} \mid \Part) \in L_\Part^{\mu'}(H)$ by Proposition~\ref{InverseRestriction}(i); since $L_\Part^{\mu'}(H)$ is transferral-free, this implies that $(\ub_X - \ub_{X'} \mid \Part) = \0$ and hence that $X$ and $X'$ are contained in the same part of $\Part$.
Let $\Qart'$ be formed from $\Qart$ by merging $X$ and $X'$ into a single part.
Then $\Qart'$ is a refinement of $\Part$ into $q-1$ parts of size at least $c'n$,
so $(\ib' \mid \Qart') \in L_{\Qart'}^{\mu'/(k+1)^{q-1-p}}(H)$ by our inductive hypothesis. 
By Proposition~\ref{InverseRestriction}(ii) there exists $\ib^* \in L_{\Qart}^{\mu'/(k+1)^{q-p}}(H)$ with $(\ib^* \mid \Qart') = (\ib' \mid \Qart')$. 
Now $\ib^*$ differs from $\ib'$ only in the co-ordinates corresponding to $X$ and $X'$; 
as $\ub_X - \ub_{X'} \in L_\Qart^{\mu'}(H) \subseteq L_{\Qart}^{\mu'/(k+1)^{q-p}}(H)$ 
it follows that $\ib' \in L_{\Qart}^{\mu'/(k+1)^{q-p}}(H)$, completing the induction. 
Since $\Qart$ as in the statement can have at most $1/c'$ parts, we conclude that 
$\ib' \in L_{\Qart}^{\mu'/(k+1)^{1/c'}}(H) \subseteq L_{\Qart}^{\mu}(H)$.
\end{proof}

Another important property of robust maximality is that it is preserved by random selection.

\begin{lemma} \label{RobustRandom} 
Suppose that $1/n \ll \mu \ll \mu' \ll c', c \ll \eta, 1/k$. Let $H$ be a $k$-graph on $n$ vertices and $\Part$ be
a partition of $V(H)$ that is $(c, c', \mu, \mu')$-robustly maximal with respect to $H$. 
Let $\Part'$ be a partition of $V(H)$ that refines $\Part$ and let $(n_Z)_{Z \in \Part'}$ be integers such that $\eta |Z| \le n_Z \le |Z|$ for each $Z \in \Part'$.
Suppose that $S \sub V(H)$ is chosen uniformly at random subject to the condition that $|S \cap Z| = n_Z$ for each $Z \in \Part'$.
Then with high probability $\Part[S]$ is $(\eta c, 3c'/\eta, \mu/c, (\mu')^3)$-robustly maximal with respect to $H[S]$.
\end{lemma}

The proof of this lemma requires the weak hypergraph regularity lemma,
so we defer the details to Section~\ref{DeferredProofsSection}.

\section{Fullness} \label{FullSection}

This section develops the theory of full lattices.
In the first subsection we give a characterisation in terms of finite abelian groups.
In the second subsection we give an equivalent definition of solubility,
via an application of the pigeonhole principle that will also be useful in later sections.
In the last subsection we consider a partite form of fullness
that will be needed for the partite form of our key lemma.

\subsection{The structure of full lattices}
In this subsection we characterise full lattices in terms of finite abelian groups.
Recall that a set $I$ of $k$-vectors of dimension $d$ is full if 
for every $(k-1)$-vector $\vb$ there is some $i \in [d]$ such that $\vb + \ub_i \in I$.
Recall also that a lattice $L$ is full if 
it contains a full set $I$ of $k$-vectors and is transferral-free.
We start by showing that $L$ is generated by $I$, so is an edge-lattice.

\begin{lemma} \label{NonParproperties0} 
Suppose $k \geq 3$ and $L$ is a full lattice in $\Z^\Part$, where $\Part$ is a partition of a set $V$. 
Let $I \sub L$ be a full set of $k$-vectors and let $L'$ be the lattice generated by $I$.
Then 
\begin{enumerate}[(i)]
\item for any $X_1, X'_1, X_2 \in \Part$, there exists $X'_2 \in \Part$ 
such that $\ub_{X_1} + \ub_{X_2} - \ub_{X'_1} - \ub_{X'_2} \in L'$,
\item\label{NonParproperties:index} 
for any $\ib \in L_{\max}^{\Part}$ and $X \in \Part$ there is $X' \in \Part$ such that $\ib - \ub_X + \ub_{X'} \in L'$, and
\item\label{EdgeLattice} $L=L'$ is an edge-lattice.
\end{enumerate}
\end{lemma}

\begin{proof}
To prove (i), we start by fixing any $(k-3)$-vector $\ib'$. 
Since $I$ is full, we can find $Y \in \Part$ such that $\ib' + \ub_{X_1} + \ub_{X_2} + \ub_Y \in I$. 
Similarly, we can find $X'_2 \in \Part$ such that $\ib' + \ub_{X'_1} + \ub_Y + \ub_{X'_2} \in I$. 
Then $\ub_{X_1} + \ub_{X_2} - \ub_{X'_1} - \ub_{X'_2}$ is the difference of these two index
vectors and hence lies in $L'$. 

For (ii), consider $\ib' \in L'$ that minimises $\sum_{Z \in \Part} |i'_Z - i_Z|$
subject to $\sum_{Z \in \Part} i'_Z = \sum_{Z \in \Part} i_Z$. We claim that $\sum_{Z \in \Part} |i'_Z - i_Z| \leq 2$.
Indeed, if not we may choose $X_1, X'_1, X_2$ such that either 
\begin{enumerate}[(a)]
\item $X_1 \neq X_2$, $i_{X_1} - i'_{X_1} > 0$, $i_{X_2} - i'_{X_2} > 0$ 
and $i_{X'_1} - i'_{X'_1} < 0$, or
\item $X_1 = X_2$, $i_{X_1} - i'_{X_1} > 1$ and $i_{X'_1} - i'_{X'_1} < 0$.
\end{enumerate}
In either case, we may apply (i) to choose $X'_2 \in \Part$ such that $\ib^* = \ub_{X_1} + \ub_{X_2} - \ub_{X'_1} - \ub_{X'_2} \in L'$. 
Then $\ib' + \ib^*$ contradicts our choice of $\ib'$. 
Thus the claim holds, so $\ib' = \ib - \ub_Y + \ub_Y'$ for some $Y, Y' \in \Part$. 
By (i) again, we can choose $X'$ such that 
$\ib^{**} = \ub_X + \ub_Y - \ub_{X'} - \ub_{Y'} \in L'$. 
Therefore $\ib + \ub_X - \ub_{X'} = \ib' + \ib^{**} \in L'$, as claimed. 

To see (iii), consider any $\ib \in L$. By (ii) we have $\ib' = \ib - \ub_X + \ub_{X'} \in L'$ for some parts $X$ and $X'$.
Then $\ub_X - \ub_{X'} = \ib-\ib' \in L$. Since $L$ is transferral-free we have $X=X'$. Therefore $\ib = \ib' \in L'$.
\end{proof}

Full lattices have the following maximality property.

\begin{prop}\label{FullMax}
Suppose $L$ and $L'$ are edge-lattices in $\Z^d$
such that $L$ is full, $L'$ is transferral-free and $L \sub L'$.
Then $L=L'$.
\end{prop}

\begin{proof}
Let $I$ and $I'$ be full generating sets for $L$ and $L'$. Suppose that $L \ne L'$, so $I \ne I'$.
Choose $\ib \in I' \sm I$ and write $\ib = \ib' + \ub_i$,
where $\ib'$ is a $(k-1)$-vector and $i \in [d]$.
Since $I$ is full, we have $\ib'+\ub_j \in I$ for some $j \in [d]$.
Note that $j \ne i$, as $\ib \notin I$. But $I \sub L'$,
so $L'$ contains the transferral $\ub_i-\ub_j = \ib - (\ib'+\ub_j )$.
This contradiction shows that $L=L'$.
\end{proof}

Next we define a group of cosets that is naturally associated with any lattice.

\begin{defin}
Suppose $L$ is an edge-lattice in $\Z^\Part$, where $\Part$ is a partition of a set $V$.
\begin{enumerate}[(i)]
\item The \emph{coset group} of $(\Part, L)$ is $G = G(\Part, L) = L_{\max}^{\Part}/L$.
\item For any $\ib \in L_{\max}^{\Part}$, the \emph{residue} of $\ib$ in $G$ is $R_G(\ib)=\ib+L$.
For any $A \sub V$ of size divisible by $k$, 
the \emph{residue} of $A$ in $G$ is $R_G(A)=R_G(\ib_\Part(A))$.
\end{enumerate}
\end{defin}

Next we show that $|G(\Part,L)| = |\Part|$ when $L$ is full and $k \ge 3$.
Note that this may not hold if $k=2$, as shown by the example
where $|\Part|=2$ and $L$ is generated by $(1,1)$.

\begin{lemma} \label{GroupSize} 
Suppose $k \geq 3$ and $L$ is a full lattice in $\Z^\Part$, where $\Part$ is a partition of a set $V$. 
Then $|G(\Part,L)| = |\Part|$.
\end{lemma}

\begin{proof}
Fix any $(k-1)$-vector $\ib'$. We claim that every coset $L + \vb$ of $L$ in
$L_{\max}^{\Part}$ contains an index vector $\ib' + \ub_{Y}$ for some $Y \in \Part$.
To see this, note that since $L$ is full there exists $X \in \Part$ such that $\ib' + \ub_{X} \in L$. 
Also, by Lemma~\ref{NonParproperties0}(\ref{NonParproperties:index}) 
we can choose $Y \in \Part$ such that $-\vb - \ub_X + \ub_{Y} \in L$. 
But now $\ib' + \ub_{Y} = (\ib' + \ub_X) + (-\vb - \ub_X + \ub_{Y}) + \vb \in L + \vb$, as claimed.
Furthermore, $L + \vb$ cannot contain $\ib' + \ub_{Y}$ and $\ib' + \ub_{Y'}$ for distinct parts $Y$, $Y'$ of $\Part$,
as then $\ub_{Y}-\ub_{Y'} \in L$, contradicting the fact that $L$ is transferral-free. 
Therefore  $|G(\Part,L)| = |\Part|$.
\end{proof}

Now we describe the structure of $L$ in terms of its coset group.
If $G$ is an abelian group, $g \in G$ and $r$ is a non-negative integer
then we write $rg$ for the sum of $r$ copies of $g$.
We show that any full lattice arises from the following construction.

\begin{construct} \label{FullConstruct}
Let $G$ be an abelian group. Let $\Part$ be a partition of a set $V$ into $|G|$ parts, identified with $G$.
Fix $g_0 \in G$. Let $I(G,g_0)$ be the set of $k$-vectors $\ib \in L_{\max}^{\Part}$ with 
\[\sum_{g \in G} i_g g = g_0.\]
Let $L(G,g_0)$ be the lattice generated by $I(G,g_0)$.
\end{construct} 

\begin{lemma} \label{FullConstructProps}
$I(G,g_0)$ and $L(G,g_0)$ are full, and $L(G,g_0)$ is the set of index vectors $\ib \in L_{\max}^{\Part}$ with 
\[\sum_{g \in G} i_g g = (k^{-1} \sum_{g \in G} i_g) g_0.\]
\end{lemma}

\begin{proof}
First we show that $I(G,g_0)$ is full, i.e.\ that for every $(k-1)$-vector $\vb$
there is $h \in G$ such that $\vb + \ub_h \in I(G,g_0)$.
Indeed, we can take $h = g_0 - \sum_{g \in G}v_g g$.
Next let $L'$ be the lattice consisting of all $\ib \in L_{\max}^{\Part}$ with $\sum_{g \in G} i_g g = (k^{-1} \sum_{g \in G} i_g) g_0$.
Note that $L(G,g_0)$ is contained in $L'$.
Furthermore, $L'$ is transferral-free, as for any distinct $g_1,g_2 \in G$
we have $\sum_{g \in G}(\ub_{g_1}-\ub_{g_2})_g g = g_1 - g_2 \ne 0_G = 0g_0$, so $\ub_{g_1} - \ub_{g_2} \notin L'$.
Thus $L(G,g_0)$ is transferral-free, and so is full.
Finally, $L(G,g_0)=L'$ by Proposition~\ref{FullMax}.
\end{proof}

\begin{thm} \label{FullStructure}
Let $k \geq 3$ and suppose $L$ is a full edge-lattice in $\Z^\Part$, where $\Part$ is a partition of a set $V$. 
Then there is an identification of $\Part$ with $G = G(\Part,L)$ and 
some $g_0 \in G$ such that $L = L(G,g_0)$ and 
$R_G(\ib) = \sum_{g \in G} i_g g - (k^{-1} \sum_{g \in G} i_g)g_0$ 
for any $\ib \in L^\Part_{\max}$.
\end{thm}

\begin{proof}
We fix an arbitrary part $X^0 \in \Part$ and identify $X^0$ with the identity $0_G \in G$.
Next we identify each $X \in \Part$ with $R_G(\ub_X-\ub_{X^0})$.
Note that for distinct parts $X,X'$ of $\Part$ we have $R_G(\ub_X-\ub_{X^0}) \ne R_G(\ub_{X'}-\ub_{X^0})$,
otherwise we would have $\ub_{X}-\ub_{X'} \in L$, contradicting the fact that $L$ is transferral-free. 
Furthermore, the identification is bijective by Lemma~\ref{GroupSize}.
 
Let $g_0 = - k \ub_{X^0} + L \in G$. 
Consider any $\ib \in L_{\max}^{\Part}$ and write $r = k^{-1} \sum_{g \in G} i_g  \in \mb{Z}$.
Then $\ib - rk \ub_{X^0} = \sum_{X \in \Part} i_X (\ub_X - \ub_{X^0}) \in \sum_{g \in G} i_g g$. 
Now if 
$\sum_{g \in G} i_g g = rg_0$ 
then $\ib \in L$. This shows that $L(G,g_0) \sub L$.
Furthermore, $L(G,g_0)$ is full by Lemma~\ref{FullConstructProps} 
and $L$ is transferral-free, so $L=L(G,g_0)$ by Proposition~\ref{FullMax}. Finally, we saw above that $R_G(\ib - rk\ub_{X^0}) = \sum_{g \in G} i_g g$. Since $R_G(-k\ub_{X^0}) = g_0$, we have $R_G(\ib) = \sum_{g \in G} i_g g - rg_0$.
\end{proof}

\begin{rem} \label{Translation}
The identification of $\Part$ with $G = G(\Part,L)$ 
in the proof of Theorem~\ref{FullStructure} is determined up to translation by $G$,
and the element $g_0$ is determined up to translation by $kG$. 
To see this, consider two identifications as in Theorem~\ref{FullStructure},
say $\pi^i: \Part \to G$ defined by $\pi^i(X) = R_G(\ub_X-\ub_{X^i})$ 
for some $X^i \in \Part$ for $i=0,1$.
Then for any $X \in \Part$ we have $\pi^1(X)-\pi^0(X) = g'$,
where $g' = R_G(\ub_X-\ub_{X^1}) - R_G(\ub_X-\ub_{X^0}) = R_G(\ub_{X^0}-\ub_{X^1}) \in G$ 
is independent of $X$.
Now set $g_0 = - k \ub_{X^0} + L \in G$ and $g_1 = - k \ub_{X^1} + L \in G$; then $g_1 - g_0 = kR_G(\ub_{X^0} - \ub_{X^1}) = kg'$, as claimed.
\end{rem}

\subsection{Solubility} 
Recall that a full pair $(\Part, L)$ for a $k$-graph $H$ is soluble 
if there exists a matching $M$ in $H$ of size at most $|\Part| - 1$ 
such that $\ib_\Part(V(H) \sm V(M)) \in L$.
Here we show that omitting the size condition on $M$ 
gives an equivalent condition.

\begin{lemma} \label{EquivSol}
Let $(\Part, L)$ be a full pair for a $k$-graph $H$, where $k \geq 3$.
Then $(\Part, L)$ is soluble if and only if there exists a matching $M$ in $H$ 
such that $\ib_\Part(V(H) \sm V(M)) \in L$.
\end{lemma}

The proof uses the following application of the pigeonhole principle.

\begin{prop} \label{abeliangroup}
Let $G =(X, +)$ be an abelian group of order $m$, and suppose that elements $x_i \in X$ for $i \in [r]$ are such that $\sum_{i \in [r]} x_i = x'$. Then $\sum_{i \in I} x_i = x'$ for some $I \subseteq [r]$ with $|I| \leq m-1$. 
\end{prop}

\begin{proof}
It suffices to show that if $r > m-1$ then $\sum_{i \in [r]} x_i = \sum_{i \in I} x_i $ for some $I \subseteq [r]$ with $|I| < r$. To see this, note that there are $r+1 > m$ partial sums $\sum_{i \in [j]} x_i$ for $0 \leq j \leq r$, so by the pigeonhole principle some two must be equal, that is, there exist $j_1 < j_2$ so that $\sum_{i \in [j_1]} x_i = \sum_{i \in [j_2]} x_i$. Then 
$$\sum_{i \in [r]} x_i = 
\sum_{i \in [r] \sm \{j_1+1, \dots, j_2\}} x_i + \sum_{i \in [j_2]} x_i - \sum_{i \in [j_1]} x_i = 
\sum_{i \in [r] \sm \{j_1+1, \dots, j_2\}} x_i,$$
as required.
\end{proof}

\begin{proof}[Proof of Lemma~\ref{EquivSol}]
If $(\Part, L)$ is soluble then such a matching $M$ exists by definition.
Conversely, suppose such a matching $M$ exists.
Write $G = G(\Part, L)$ and note that $|G| = |\Part|$ by Lemma~\ref{GroupSize}.
Since $\ib_\Part(V(H) \sm V(M)) \in L$ we have $\sum_{e \in M} R_G(e) = R_G(V(H))$.
By Proposition~\ref{abeliangroup} there exists a submatching $M'$ of $M$ 
of size at most $|G| - 1 = |\Part|-1$ such that $\sum_{e \in M'} R_G(e) = R_G(V(H))$.
Hence $R_G(V(H) \sm V(M')) = 0$, i.e.\ $\ib_\Part(V(H) \sm V(M')) \in L$.
\end{proof}

\subsection{Partite fullness}
In the partite form of the key lemma we need the following partite form of fullness,
in which we consider index vectors of sets that are partite with respect to some fixed partition.

\begin{defin}
Let $\Part'$ be a partition of a set $V$ into $k$ parts and $\Part$ be a refinement of $\Part'$.
\begin{enumerate}[(i)]
\item We say that a set $I$ of $k$-vectors with respect to $\Part$ is \emph{full} with respect to $\Part'$ 
if $(\ib \mid \Part') = \1$ for every $\ib \in I$, and for every $X \in \Part'$
and $(k-1)$-vector $\vb$ with respect to $\Part$ such that $(\vb \mid \Part') = \1 - \ub_X$,
there is some $Y \in \Part$ with $Y \sub X$ such that $\vb + \ub_Y \in I$.
\item We write $L^{\Part\Part'}_{\max}$ for the lattice of vectors $\ib \in \Z^\Part$ 
such that $(\ib \mid \Part')$ is a multiple of $\1$.
\item We say that a lattice $L \sub L^{\Part\Part'}_{\max}$ is \emph{full} with respect to $\Part'$
if it is transferral-free and contains a set of $k$-vectors that is full with respect to $\Part'$.
\end{enumerate}
\end{defin}

The next proposition records some properties of partite fullness.
We just give the proofs of (i) and (ii),
as the proofs of (iii-v) are very similar to those of
Lemmas~\ref{NonParproperties0} and~\ref{GroupSize}.

\begin{prop} \label{properties0} 
Let $k \ge 3$, let $\Part'$ be a partition of a set $V$ into $k$ parts, $\Part$ be a refinement of $\Part'$,
and $L \sub L^{\Part\Part'}_{\max}$ be a full lattice with respect to $\Part'$.
Let $I \sub L$ be a set of $k$-vectors that is full with respect to $\Part'$
and let $L'$ be the lattice generated by $I$.
Then the following properties hold for some integer~$r$.
\begin{enumerate}[(i)]
\item\label{properties:parts} Each part of $\Part'$ is refined into exactly $r$ parts by $\Part$.
\item\label{properties:indexing} For every part $X$ of $\Part$, there are exactly $r^{k-2}$ vectors $\ib \in I$ such that $i_X = 1$.
In particular, $|I|=r^{k-1}$.
\item\label{properties:index} For any $\ib \in  L^{\Part\Part'}_{\max}$ and $X\in \Part$, 
there is $X' \in \Part$ such that $\ib - \ub_X + \ub_{X'} \in L'$.
\item\label{parGroupSize} $|L^{\Part\Part'}_{\max}/L| = r$.
\item $L=L'$ is an edge-lattice.
\end{enumerate}
\end{prop}

\begin{proof}[Proof of (i) and (ii).]
Fix $X \in \Part'$ and let $r$ be the number of parts into which $\Part$ refines $X$. 
For any other $X' \in \Part'$, we will construct a bijection between 
the parts of $\Part$ contained in $X$ and those contained in $X'$.
Let $\vb$ be a non-negative index vector with respect to $\Part$ such that $(\vb \mid \Part') = \1 - \ub_{X} - \ub_{X'}$.
For each $Y \in \Part$ which is contained in $X$, we apply the property that $I$ is full to the vector $\vb + \ub_Y$ 
to obtain a part $Y' \in \Part$ which is contained in $X'$, such that $\vb + \ub_Y + \ub_{Y'} \in I$.
Now observe that since $L$ is transferral-free, $Y'$ must be unique.
Thus $Y \mapsto Y'$ is a bijection, so (i) holds.
For (ii), we observe that the number of such $(k-2)$-vectors $\vb$ is $r^{k-2}$ 
and there is a one-to-one correspondence between them and vectors $\ib \in I$ such that $i_Y = 1$.
By considering the $r$ parts of $\Part$ contained in $X$ we obtain $|I|=r^{k-1}$.
\end{proof}
 
Finally, we require the following consequence of Farkas' Lemma.

\begin{prop} \label{applyfarkas}
Let $\Part'$ partition a set $V$ into parts $V_1, \dots, V_k$ each of size $n$, 
and let $\Part$ be a partition refining $\Part'$ into parts of size at least $n/k$. 
Suppose $I$ is a set of $k$-vectors with respect to $\Part$ which is full with respect to $\Part'$.
Then $\ib_\Part(V) \in PC(I)$. 
\end{prop}

\begin{proof}
Suppose for a contradiction that $\ib_\Part(V) \notin PC(I)$. Then by Theorem~\ref{farkas} we may fix $\ab \in \R^{|\Part|}$ such that $\ab \cdot \ib \ge 0$ for every $\ib \in I$ and $\ab \cdot \ib_\Part(V) < 0$. For each $i \in [k]$, let $X_i^1, \dots, X_i^{b_i}$ be the parts of $\Part$ which are subsets of $V_i$, and let $a_i^1, \dots, a_i^{b_i}$ be the corresponding coordinates of $\ab$, with the labels chosen so that $a_i^1 \leq \dots \leq a_i^{b_i}$. Fix $i \in [k]$ for which $a_i^{b_i} - a_i^1$ is minimised, so in particular $a_i^{b_i} - a_i^1 \leq \frac{1}{k}\sum_{j \in [k]} a_j^{b_j} - a_j^1$. By assumption we may choose $\ib \in I$ such that $\ib = \ub_{X_i^s} + \sum_{j \neq i} \ub_{X_j^1}$ for some $s \in [b_i]$. Then
\begin{align*}
0 &> \ab \cdot \ib_\Part(V) \geq \sum_{j \in [k]} na_j^1 + \frac{n}{k}(a_j^{b_j} - a_j^1) \\
&\geq n \left(\sum_{j \in [k]} a_j^1 + (a_i^{b_i} - a_i^1)\right) \geq n \ab \cdot \ib \geq 0,
\end{align*}
a contradiction. So $\ib_\Part(V) \in PC(I)$.
\end{proof}
 
\section{The key lemmas} \label{KeyLemmaSection}

The following key lemma will be used in the proof of Theorem~\ref{PMNeccSuff}. It provides a simple condition for finding a perfect matching under Setup~\ref{setup} when we have a robustly maximal partition: it suffices that the index of the vertex set is in the robust edge-lattice and every vertex is in many edges with index in the robust edge-lattice.

\begin{lemma} \label{NonPartiteLemma} 
Suppose that $k \geq 3$ and $1/n \ll \eps \ll \mu \ll \mu' \ll c, d \ll \gamma, 1/k$. 
Let $H$ be a $k$-graph on a set $V$ of size $kn$ and $\Part$ be a partition of $V$. 
Suppose that 
\begin{enumerate}[(i)]
\item at most $\eps n^{k-1}$ $(k-1)$-sets $S \subseteq V$ have $d_H(S) < (1+\gamma)n$,
\item $\Part$ is $(c, c, \mu, \mu')$-robustly maximal with respect to $H$,
\item any vertex is  in at least $dn^{k-1}$ edges $e \in H$ with $\ib_\Part(e) \in L_\Part^\mu(H)$, and
\item $\ib_\Part(V) \in L_\Part^\mu(H)$.
\end{enumerate}
Then $H$ contains a perfect matching.
\end{lemma} 

We will prove Lemma~\ref{NonPartiteLemma} by taking a random $k$-partition of $H$ and deducing it from the partite key
lemma that we prove in this section. However, showing that the conditions on $H$ transfer to the $k$-partite subgraph
induced by the random partition is technical and non-trivial, and so we defer the proof of Lemma~\ref{NonPartiteLemma}
to Section~\ref{DeferredProofsSection}. In this section, after some preparatory results in the first subsection,
we prove the partite key lemma (Lemma~\ref{kPartiteLemma}) in the second subsection.

\subsection{Preliminaries.} The following result forms the base case of the partite key lemma (Lemma~\ref{kPartiteLemma}). This is the case where our $k$-graph $H$ is far from any divisibility barrier. Note that here any edge $e \in H$ must have $\ib_\Part(e) = \1$. So condition (iv) is trivial, and condition (iii) simply states that every vertex lies in at least $dn^{k-1}$ edges. However, we state the lemma in this form for ease of comparison to the full version of Lemma~\ref{kPartiteLemma}.

\begin{lemma} \label{KeyLemmaBaseCase} 
Suppose that $k \ge 3$ and $1/n \ll \eps \ll \mu \ll \mu' \ll c, d \ll \gamma, 1/k$. 
Let $\Part$ partition a set $V$ into vertex classes $V_1, \dots, V_k$ of size $n$
and let $H$ be a $\Part$-partite $k$-graph on $V$.
Suppose that 
\begin{enumerate}[(i)]
\item at most $\eps n^{k-1}$ $\Part$-partite $(k-1)$-sets $S \subseteq V$ have $d_H(S) < (1/k + \gamma)n$,
\item $\Part$ is $(c, c, \mu, \mu')$-robustly maximal with respect to $H$,
\item  any vertex is  in at least $dn^{k-1}$ edges $e \in H$ with $\ib_\Part(e) \in L_\Part^\mu(H)$, and
\item $\ib_\Part(V) \in L_\Part^\mu(H)$.
\end{enumerate}
Then $H$ contains a perfect matching.
\end{lemma} 

The proof requires the following lemma that will enable us to apply Theorem~\ref{HypergraphMatching}.

\begin{lemma} \label{PrunePartiteComplex} 
Suppose that $1/n \ll \eps \ll \alpha \ll 1/k$.
Let $\Part$ partition a vertex set $V$ into parts $V_1, \ldots, V_k$ each of size $n$.
Suppose $H$ is a $\Part$-partite $k$-graph on $V$ such that at most $\eps n^{k-1}$ 
$\Part$-partite $(k-1)$-sets $S \subseteq V$ have $d_H(S)<D$. 
Then there exists a $k$-complex $J$ on $V$ such that $J_k \subseteq H$ and 
$\delta^*(J) \geq ((1 - \sqrt{\eps})n, (1 - \alpha)n, \ldots, (1 - \alpha)n, D - \alpha n)$.
\end{lemma}

\begin{proof} 
Let $\beta = \alpha/k^k$. 
Call a $\Part$-partite $(k-1)$-set $A \subseteq V$ \emph{bad} if $d_H(A) < D$, and \emph{good} otherwise. 
We define bad $\Part$-partite sets $A$ of size $i$ recursively for $i = k-2,k-3,\dots,0$ by saying
that $A$ is \emph{bad} if it there are more than $\beta n$ vertices $x$ in some part of $\Part$ such that $A \cup \{x\}$ is bad.
If a $\Part$-partite set $A$ is not bad we say it is \emph{good}. 
We claim that for any $0 \leq i \leq k-1$ the number of bad $i$-sets is at most $\sqrt{\eps}\binom{n}{i}$.
To see this, we show by induction on $i$ that the number of bad $(k-i)$-sets is at most $\eps (k/\beta)^i \binom{n}{k-i}$.
The base case $i=1$ holds by assumption on $H$, since $\beta \ll 1/k$. The induction step follows since there are at least $\beta n/k$ times as many bad $(k-i)$-sets as bad $(k-i-1)$-sets. Thus the claim holds.

We define the $k$-complex $J$ as follows. We take $J_0 = \{\emptyset\}$, which is good by the above claim, and
define $J_i$ recursively for $1 \leq i < k$ as the family of good $i$-sets $A$ 
such that every $(i-1)$-subset of $A$ is an element of $J_{i-1}$. Since the number of bad singletons (i.e. bad 1-sets) is at most $\sqrt{\eps}n$, we then have $\delta^*_0(J) \geq (1 - \sqrt{\eps})n$. 
For the remaining inequalities on $\delta^*(J)$ we now prove by induction that $\delta_i^*(J) \geq (1 - k^i\beta)n \geq (1 - \alpha)n$ for each $0 \leq i \leq k-2$. 
The base case $i=0$ is already done. 
For the induction step, suppose that $\delta_i^*(J) \geq (1 - k^i\beta)n$ for some $0 \leq i \leq k-3$. 
Suppose for a contradiction that there exists $e \in J_{i+1}$ such that $|J_{i+2}(e) \cap V_j| < (1 - k^{i+1}\beta)n$ for some $j \in [k]$ such that $e$ is disjoint from $V_j$.

Let $F$ be the set of good sets $e^+ = e \cup \{w\}$ such that $w \in V_j$ and $e^+ \notin J_{i+2}$.
Then $|F| \ge (k^{i+1}-1) \beta n$, as by choice of $e$ there are at least $k^{i+1} \beta n$ choices of $w$ 
such that $e^+ \notin J_{i+2}$, and at most $\beta n$ choices make $e^+$ bad, since $e \in J_{i+1}$ is good.
Next note that any $e^+ \in F$ contains some $(i+1)$-set $e^* \notin J_{i+1}$, 
otherwise we would have added $e^+$ to $J_{i+2}$. 
Let $F'$ be the set of such $(i+1)$-sets $e^*$ for all $e^+ \in F$ 
(choosing arbitrarily if there is more than one choice for some $e^+$).
Then $|F'|=|F|$, as each $e^+$ is determined by its choice of $e^*$ (we have $e^+ = e \cup e^*$).

Note that each $e^* \in F'$ intersects $e$ in $i$ vertices, so there is some $i$-set $e^- \sub e$ which is contained in at least $|F'|/(i+1) > k^i \beta n$ sets of $F'$. Then $|J_{i+1}(e^-) \cap V_j| < (1 - k^i\beta)n$, as $F'$ is disjoint from $J_{i+1}$.
But since $e \in J_{i+1}$ we have $e^- \in J_i$, and so $\delta_i^*(J) < (1 - k^i\beta)n$.
This contradiction establishes the induction step.

Now we define $J_k$ to be the set of all edges $e \in H$ such that $e^- \in J_{k-1}$ for every $e^- \subseteq e$ of size $k-1$. 
It remains to show that $\delta^*_{k-1}(J) \geq D - \alpha n$, which we do similarly to the induction step.
Suppose there exists $e \in J_{k-1}$ such that $|J_k(e)| < D - \alpha n$ and let $F$ be the set of edges $e^+ \in H \sm J_{k}$ such that $e \subseteq e^+$.
Since $e$ is good, we must have $|F| \geq \alpha n$.
Next note that any $e^+ \in F$ contains some $(k-1)$-set $e^* \notin J_{k-1}$, otherwise we would have added $e^+$ to $J_{k}$. 
Again let $F'$ be the set of such $(k-1)$-sets $e^*$ for all $e^+ \in F$.
Then $|F'|=|F|$ (as before, each $e^+$ is determined by its choice of $e^*$).

Note that each $e^* \in F'$ intersects $e$ in $k-2$ vertices, so there is some $(k-2)$-set $e^- \sub e$ which is contained in at least $|F'|/(k-1) > k^{k-1} \beta n$ sets of $F'$. 
Then $|J_{k-1}(e^-)| < (1 - k^{k-1} \beta)n$, as $F'$ is disjoint from $J_{k-1}$.
But since $e \in J_{k-1}$ we have $e^- \in J_{k-2}$, and so $\delta_{k-2}^*(J) < (1 - k^{k-1} \beta)n$.
This contradiction completes the proof.
\end{proof}

\begin{proof}[Proof of Lemma~\ref{KeyLemmaBaseCase}.]
Introduce a new constant $\alpha$ with $\eps \ll \alpha \ll \mu$.
We apply Lemma~\ref{PrunePartiteComplex} with $D=(1/k + \gamma)n$ 
to obtain a $\Part$-partite $k$-complex $J$ on $V$ with $J_k \subseteq H$ and 
\[\delta^*(J) \geq \left((1 - \sqrt{\eps})n, (1 - \alpha)n, \ldots, (1 - \alpha)n, (1/k + \gamma -\alpha)n\right).\] 
Next we choose a matching $M$ in $H$ of size at most $k \sqrt{\eps} n$ which includes all of the
at most $k\sqrt{\eps} n$ vertices $x \in V$ for which $\{x\}$ is not an edge of $J$;
we can construct $M$ greedily by the vertex degree assumption.
We write $V' = V \sm V(M)$, $n'=n-|M|$, and $J' = J[V']$ 
and verify that $J'$ satisfies the conditions of Theorem~\ref{HypergraphMatching} 
with $n'$, $2\mu$ and $\gamma/2$ in place of $n$, $\mu$ and $\gamma$ respectively. 

The degree sequence condition holds as
\[\delta^*(J') \geq \left(n', (1 - 2\alpha)n', \ldots, (1 - 2\alpha)n', (1/k + \gamma - 2\alpha)n'\right).\]
(Note that only the first and last co-ordinates are close to being tight).
Furthermore, $\Part[V']$ is $(c/2, 2c, 2\mu, \mu'/2)$-robustly maximal with respect to $J'_k$ by Lemma~\ref{RobustInherit}.
Now suppose that $\Part'$ is a partition of $V'$ into parts of size 
at least $\delta^*_{k-1}(J') - 2\mu n' \geq n'/k \geq 2c|V'|$ which refines $\Part[V']$. 
Since $J'_k$ is $\Part$-partite every edge $e \in J'_k$ has $\ib_{\Part}(e) = \1$; 
in particular we have $\1 \in L_{\Part[V']}^{\mu'/2}(J'_k)$. 
So by Proposition~\ref{robmaxinverse} we have $\ib \in L^{2\mu}_{\Part'}(J'_k)$ 
for any index vector $\ib$ with respect to $\Part'$ such that $(\ib \mid \Part) = \1$,
i.e.\ $L_{\Part'}^{2\mu}(J'_k)$ is complete with respect to $\Part[V']$.

Now $J'_k$ has a perfect matching by Theorem~\ref{HypergraphMatching}.
Together with $M$, we have a perfect matching in $H$.
\end{proof}

Next we recall from the proof outline the form of our inductive approach,
which uses the following `canonical' induced subgraphs.

\begin{defin} \label{canonical} (Canonical subgraphs)
Suppose that $H$ is a $k$-graph on $V$, $\Part$ is a partition of $V$, and $\mu > 0$. For each $\ib \in I^\mu_\Part(H)$, we define $H_\ib$ to be the induced subgraph of $H$ on the
union of the parts $W \in \Part$ such that $i_W > 0$, and $\Part_\ib$ to be the restriction of $\Part$ to $V(H_\ib)$. 
\end{defin}

Note that $H_\ib$ contains all the edges of $H$ of index $\ib$ and is $\Part_\ib$-partite whenever $H$ is $\Part$-partite. 
In the following proposition we establish various properties of $H$ and these subgraphs $H_\ib$.
In (i), we show that the robust edge-lattice of $H$ is full, so the properties of Proposition~\ref{properties0} also hold.
In (ii), we show that $H_\ib$ inherits a similar codegree condition to that of $H$ (with slightly weaker constants), 
which is essential for the induction. 
Note that a minimum codegree condition would not be inherited, so this is one reason why for much of this paper 
we work with a condition on the codegree of most $(k-1)$-sets (we also need this condition for Lemma~\ref{FindPMLemma}).
The next two parts are useful boosting properties for the edge detection parameter and the part sizes. We will apply these with $d=1/k$ and $c \ll 1/k$, so (iii) shows that the part sizes of $\Part$ are in fact much larger than our original assumption, and (iv) shows that the robust edge-lattice is unchanged for a wide range of $\mu$.
Finally, recall from our proof outline that we choose matching edges to cover bad vertices;
(v) will show that there are few bad vertices.

\begin{prop} \label{properties}
Suppose that $1/n \ll \mu, \eps \ll \psi, d, c, 1/k$. Let $\Part'$ partition a set $V$ into $k$ parts $V_1, \dots, V_k$,
where $cn \leq |V_i| \leq n$ for each~$i \in [k]$. Let $H$ be a $\Part'$-partite $k$-graph on $V$ 
and $\Part$ be a partition of $V$ which refines $\Part'$. Suppose that
\begin{enumerate}[({$\alpha$}1)] 
\item at most $\eps n^{k-1}$ $\Part'$-partite $(k-1)$-sets $S \subseteq V$ have $d_H(S) < dn$, and
\item $\Part$ has parts each of size at least $cn$ and $L^\mu_\Part(H)$ is transferral-free. 
\end{enumerate}
Then we have the following properties.
\begin{enumerate}[(i)]
\item\label{properties:full} $I_\Part^\mu(H)$ is full with respect to $\Part'$.
\item\label{properties:codegree} For any $\ib \in I^\mu_\Part(H)$ at most $\psi n^{k-1}$ 
$\Part_\ib$-partite $(k-1)$-sets $S \subseteq V(H_\ib)$ have $|H_\ib(S)| < (d - \psi) n$. 
\item\label{properties:final} Each part of $\Part$ has size at least $(d - \psi) n$.
\item\label{properties:lattice} $I^{dc^{k-1}/2k^k}_\Part(H) = I^\mu_\Part(H)$.
\item\label{properties:vdegree} For any $\ib \in I^\mu_\Part(H)$ at most $\psi n$ vertices of $V(H_\ib)$ lie in fewer than $c^{k-2}dn^{k-1}/2$ edges of $H_\ib$.
\end{enumerate}
\end{prop} 

\begin{proof}
Let $\theta$ satisfy $\mu, \eps \ll \theta \ll \psi, d, c, 1/k$.
For (i), let $\vb$ be a non-negative index vector with respect to $\Part$ such that $(\vb \mid \Part') = \1 - \ub_{V_k}$.
Then it suffices to show that there is some part $X_k$ of $\Part$ with $X_k \subseteq V_k$ 
such that $H$ contains at least $\mu |V|^k$ edges of index $\vb + \ub_{X_k}$.
To see this, write $\vb = \sum_{j \in [k-1]} \ub_{X_j}$ for parts $X_j \subseteq V_j$ of $\Part$ 
and note that the number of $(k-1)$-sets $S = \{x_1, \dots, x_{k-1}\}$ with $x_j \in X_j$ for each $j \in [k-1]$
is at least $\prod_{j \in [k-1]} |X_j| \geq (cn)^{k-1}$.
Hence at least $(c^{k-1} - \eps)n^{k-1}$ such $(k-1)$-sets have $d_H(S) \geq dn$, 
and so at least $(c^{k-1} - \eps)dn^k$ edges of $H$ contain one vertex from each of $X_1, \dots, X_{k-1}$. 
Since $\Part$ refines $V_k$ into at most $1/c$ parts, some part $X_k \subseteq V_k$ of $\Part$ must be as required.

For (ii), fix $\ib \in I^\mu_\Part(H)$, and for $j \in [k]$ let $W_j \sub V_j$ be the part of $\Part$ for which $i_{W_j} = 1$ 
(so $V(H_\ib) = \bigcup_{j \in [k]} W_j$). Let $\mc{S}$ be the family of sets $S = \{x_1, \dots, x_{k-1}\}$ 
with $x_i \in W_i$ for each $i \in [k-1]$. Also define $\mc{S}_1$ to consist of those $S \in \mc{S}$ with $d_H(S) \geq dn$ but $d_{H_\ib}(S) < (d - \theta) n$, and $\mc{S}_2$ to consist of those $S \in \mc{S}$ with $d_H(S) < dn$. Then as above we have $|\mc{S}| \geq (cn)^{k-1}$, whilst $|\mc{S}_2| < \eps n^{k-1}$ by ($\alpha 1$). To bound $|\mc{S}_1|$, note that each $S \in \mc{S}_1$ is contained in at least $\theta n$ edges $e \in H$ with $\ib_\Part(e)$ equal to some $\ib' \ne \ib$,
and since $\ib'-\ib$ is a transferral for any such $\ib'$ we have $\ib' \notin I^\mu_\Part(H)$.
Now there are at most $c^{-k} \mu (kn)^k$ such edges of $H$, as each part $V_j$ of $\Part'$ is refined into 
at most $1/c$ parts of $\Part$, so there are at most $c^{-k}$ possible values of $\ib_\Part(e)$ for an edge $e$ of $H$.
We conclude that $|\mc{S}_1| \le c^{-k} \mu (kn)^k / \theta n \leq \theta n^{k-1}/(k+1)$, so $\mc{T} = \mc{S}_1 \cup \mc{S}_2$ has
size $|\mc{T}| \leq \theta n^{k-1}/k < |\mc{S}|$. Since any $S \in \mc{S} \sm \mc{T}$ has $d_{H_\ib}(S) \geq (d - \theta) n$, 
by symmetry this proves a stronger form of (ii), with $\theta$ in place of $\psi$. 

For (iii), consider any $Y \in \Part$ and fix $\ib \in I_\Part^\mu(H)$ such that $i_Y = 1$;
this exists by (i) and Proposition~\ref{properties0}(\ref{properties:indexing}).
Without loss of generality $Y \sub V_k$. Then with notation as in (ii), for any $S \in \mc{S} \sm \mc{T}$
we have $|Y| \ge d_{H_\ib}(S) \geq (d - \theta) n$.
The same argument implies that $|H_\ib| \geq |\mc{S} \sm \mc{T}| (d-\theta)n \geq c^{k-1}d|V(H)|^k/2k^k$, 
and so we have (iv).

Finally, note that the number of vertices in $W_1$ that belong to at least $(cn)^{k-2}/3$ sets $S \in \mc{T}$
is at most $|\mc{T}| \cdot 3/(cn)^{k-2} \le \psi n/k$. Since $\Part$ has parts each of size at least $cn$,
any other vertex of $W_1$ lies in at least $2(cn)^{k-2}/3 \cdot (d-\theta)n \geq dc^{k-2}n^{k-1}/2$ edges of $H_\ib$.
The same argument applies to any part $W_j$ with $j \in [k]$, so this proves (v).
\end{proof}

\subsection{The partite key lemma} In the remainder of this section we prove the following partite form of the key lemma.

\begin{lemma} \label{kPartiteLemma} 
Let $k$ and $\ell$ be integers with $k \geq \ell \geq 2$ and $k \geq 3$, 
and suppose that $1/n \ll \eps \ll \mu \ll \mu' \ll c, d \ll \gamma, 1/k$. 
Let $\Part'$ partition a set $V$ into vertex classes $V_1, \dots, V_k$ of size $n$, 
let $H$ be a $\Part'$-partite $k$-graph on $V$ and let $\Part$ be a refinement of $\Part'$. 
Suppose that 
\begin{enumerate}[({$\beta$}1)]
\item at most $\eps n^{k-1}$ $\Part'$-partite $(k-1)$-sets $S \subseteq V$ have $d(S) < (1/\ell + \gamma)n$,
\item $\Part$ is $(c, c, \mu, \mu')$-robustly maximal with respect to $H$,
\item any vertex is in at least $dn^{k-1}$ edges $e \in H$ with $\ib_\Part(e) \in L_\Part^\mu(H)$, and
\item $\ib_\Part(V) \in L_\Part^\mu(H)$.
\end{enumerate}
Then $H$ contains a perfect matching.
\end{lemma} 

We remark that the case $\ell = k$ implies the cases $\ell = 2,\ldots,k-1$, and hence we set $\ell = k$ in all later applications of the lemma. The other cases are only required for the proof of the lemma itself.
 
\begin{proof} We fix $k$ and proceed by induction on $\ell$. Let $s = \lfloor 1/c \rfloor$; then we use the following hierarchy of constants.
\begin{align*}
&1/n \ll 1/D \ll \eps \ll \mu \ll \mu'_0 \ll \mu'_1 \ll \dots \ll \mu'_{s^{k^k}} \ll \mu' \\
&\ll c, d \ll \eta, 1/K \ll \gamma, 1/k.
\end{align*} 
In several places during the proof we will delete small matchings from $H$;
for convenient notation we let $\Part'$ and $\Part$ also denote
the restrictions of $\Part'$ and $\Part$ to the undeleted vertices.
\medskip 

{\noindent \bf Step 1: The case $\Part = \Part'$.} 
In this case we have a perfect matching by Lemma~\ref{KeyLemmaBaseCase}.
In particular, this gives the base case $\ell = 2$ of the induction, 
as by Proposition~\ref{properties}(\ref{properties:final}) each part of $\Part$ has size at least $n/\ell + \gamma n/2 > n/2$,
so we must have $\Part = \Part'$. \medskip

{\noindent \bf Step 2: The canonical subgraphs.}
For the rest of the proof we assume that $\Part \neq \Part'$, that $3 \leq \ell \leq k$ and that the lemma holds with $\ell-1$ in place of $\ell$. We write 
\[I = I_\Part^\mu(H) \ \ \text{ and } \ \ L = L_\Part^\mu(H).\]
Our strategy will be to split $H$ up randomly into a number of vertex-disjoint $k$-partite subgraphs, each of which satisfies the conditions of the lemma (with weaker constants) when $\ell$ is replaced by $\ell-1$. The inductive hypothesis will then imply the existence of a perfect matching in each subhypergraph, and taking the union of these matchings gives a perfect matching in $H$.
In this step we lay the groundwork by analysing the canonical subgraphs from which the random subgraphs will be chosen;
the analogues of properties $(\beta 1)$, $(\beta 2)$ and $(\beta 3)$ for these subgraphs will follow from this analysis.

We first observe that $H$, $\Part'$ and $\Part$ meet the conditions of Proposition~\ref{properties} 
with $1/\ell+\gamma$ in place of $d$. 
Therefore $I$ is full with respect to $\Part'$. We can therefore apply Proposition~\ref{properties0} to obtain an integer $r$ such that $|I|=r^{k-1}$, each part of $\Part'$ is partitioned into exactly $r$ parts by $\Part$, and for any part $X$ of $\Part$ there are exactly $r^{k-2}$ vectors $\ib \in I$ such that $i_X =1$.
Also, by Proposition~\ref{properties}(\ref{properties:final}), every part of $\Part$ has size at least $(1/\ell + \gamma/2)n$. 
Since $\Part \neq \Part'$ this implies $2 \leq r < \ell$. 

Next we recall Definition~\ref{canonical}: for each $\ib \in I$, the canonical subgraph $H_\ib$ is the induced $k$-graph on the vertex set $\bigcup_{X \in \Part: i_X = 1} X$ whose edge set consists of all edges of $H$ with index $\ib$, and $\Part_\ib$ is the restriction of $\Part$ to $V(H_\ib)$ (so $H_\ib$ is $\Part_\ib$-partite). The canonical subgraphs $H_\ib$ will act as prototypes for the $k$-graphs on which we use the inductive hypothesis,
which will be induced subgraphs of the canonical subgraphs on randomly chosen sets of vertices. Note that by Proposition~\ref{properties}(\ref{properties:codegree}), at most $\mu_0' n^{k-1}$ $\Part_\ib$-partite $(k-1)$-sets 
$S \subseteq V(H_\ib)$ have $|H_\ib(S)| < n/\ell + 3\gamma n/4$: this inheritance of the codegree condition of $H$ 
is fundamental to our inductive approach. It also provides the analogue of condition $(\alpha 1)$ when we apply Proposition~\ref{properties} again to each $H_\ib$, which we will do to `boost' the part sizes and edge-detection parameter.

Since $\Part_\ib$ may not be robustly maximal with respect to $H_\ib$, 
we now need to identify robustly maximal partitions for each canonical subgraph. Let $S = s^{k^k}$.
We claim that there exist $i \in [S]$,  
and partitions $\Qart_\ib$ of $V(H_\ib)$ for each $\ib \in I$ such that $\Qart_\ib$ refines $\Part_\ib$ 
and is $(c, c, \mu'_{i-1}, \mu'_i)$-robustly maximal with respect to $H_\ib$.
To see this, we repeatedly apply Proposition~\ref{RobustExistence} to each $\ib \in I$. 
First, we choose some $\ib \in I$ and apply Proposition~\ref{RobustExistence} with 
parameters $\mu'_{\lambda S/s}$ for $0 \leq \lambda \leq s$ to obtain some $b \in [s]$ and a 
partition $\Qart_\ib$ of $V(H_\ib)$ such that $\Qart_\ib$ refines $\Part_\ib$ and 
is $(c, c, \mu'_{(b-1)S/s}, \mu'_{bS/s})$-robustly maximal with respect to $H_\ib$. We then 
repeat the following step. After $j$ steps, we will have chosen $j$ of these partitions to be
$(c, c, \mu'_{(b-1)S/s^j}, \mu'_{bS/s^j})$-robustly maximal for some $b \in [s^j]$. We then 
choose some $\ib$ we have not yet considered and apply Proposition~\ref{RobustExistence} with 
parameters $\mu'_{\lambda S/s^{j+1}}$ for $(b-1)s \leq \lambda \leq bs$ to obtain 
some $b' \in [s^{j+1}]$ and a partition $\Qart_\ib$ of $V(H_\ib)$ such that $\Qart_\ib$ 
refines $\Part_\ib$ and is $(c, c, \mu'_{(b'-1)S/s^{j+1}}, \mu'_{b'S/s^{j+1}})$-robustly 
maximal with respect to $H_\ib$.
Since $|I| \le k^k$ we can repeat this process for every $\ib$, which proves that the claim holds.

For simplicity of notation we now relabel $\mu'_{i-1}$ as $\mu_0$ and $\mu'_i$ as $\mu_6$, and introduce new constants $\mu_1, \dots, \mu_5$ and $\mu''$ such that 
$$\mu \ll \mu_0 \ll \mu_1 \ll \dots \ll \mu_5 \ll \mu_6 \ll \mu'' \ll \mu'.$$
For each $\ib \in I$, we write
\[ I_\ib := I_{\Qart_\ib}^{\mu_0}(H_\ib) \ \ \text{ and } \ \ L_\ib := L_{\Qart_\ib}^{\mu_0}(H_\ib).\] 
We will apply Proposition~\ref{properties} to $H_\ib$ with $\Part_\ib$ in place of $\Part'$, $\Qart_\ib$ in place of $\Part$, $\mu_0$ in place of both $\mu$ and $\eps$, and $1/\ell + 3\gamma/4$ and $\gamma/12$ in place of $d$ and $\psi$ respectively ($n, c$ and $k$ are unchanged). 
Indeed, condition $(\alpha 1)$ of Proposition~\ref{properties} holds by our earlier observation that at most $\mu_0' n^{k-1}$ $\Part_\ib$-partite $(k-1)$-sets 
$S \subseteq V(H_\ib)$ have $|H_\ib(S)| < n/\ell + 3\gamma n/4$, whilst condition $(\alpha 2)$ holds since $\Qart_\ib$ is $(c, c, \mu_0, \mu_6)$-robustly maximal with respect to $H_\ib$. 
So by Proposition~\ref{properties}(\ref{properties:full}) and~(\ref{properties:final}) we deduce that $I_\ib$ is full with respect to $\Part_\ib$, and
every part of $\Qart_\ib$ has size at least $n/\ell + 2\gamma n/3$.
The latter conclusion implies that $\Qart_\ib$ is in fact $(1/\ell+2\gamma/3, c, \mu_0, \mu_6)$-robustly maximal, and also that we can apply Proposition~\ref{properties} again with $1/\ell$ in place of $c$ (and all other variables as before); Proposition~\ref{properties}(\ref{properties:lattice}) then implies that 
\[I_{\Qart_\ib}^{1/2(k\ell)^k}(H_\ib) = I_\ib.\] 
Indeed, we have the boosting property that $I_{\Qart_\ib}^{\mu_0}(H_\ib)$ and $L_{\Qart_\ib}^{\mu_0}(H_\ib)$ 
are essentially independent of the edge-detection parameter, in that they
remain unchanged when $\mu_0$ is replaced by any constant between $\mu_0$ and $1/2(k\ell)^k$.

Finally, say that a vertex of $H_{\ib}$ is \emph{bad for $\ib$} 
if it lies in fewer than $dn^{k-1}$ edges $e \in H_\ib$ with $\ib_{\Qart_\ib}(e) \in I_\ib$. 
We say that a vertex is \emph{bad} if it is bad for some $\ib \in I$.
Since $|I| = r^{k-1}$, by Proposition~\ref{properties}(\ref{properties:vdegree}) (with $\mu_1/r^{k-1}$ in place of $\psi$),
the number of bad vertices is at most $\mu_1 n$.
Using assumption $(\beta 3)$, i.e.\ that every vertex is contained in at least $dn^{k-1}$
edges $e \in H$ with $\ib_\Part(e) \in L$, we may greedily choose a matching $M$ of such
edges with size at most $\mu_1 n$ which covers all of the bad vertices.
We now restrict attention to $V' := V \sm V(M)$.

\medskip

{\noindent \bf Step 3: The meet and join.}
The principal difficulty we must overcome in order to apply the inductive hypothesis 
is to ensure that an analogue of $(\beta 4)$ holds, that is, we require subgraphs 
in which the index vector of the entire vertex set lies in the appropriate robust edge-lattice.
In this step we introduce two additional partitions of $V$ that will be used to achieve this.

We begin by noting that for the subgraph chosen in $H_\ib$, the index vector of the vertex set 
with respect to $\Qart_\ib$ simply lists the number of vertices chosen from each part of $\Qart_\ib$.
Hence if two vertices $x$ and $y$ are contained in the same part of $\Qart_\ib$ for \emph{every} $\Qart_\ib$ 
such that $x, y \in V(H_\ib)$, then they are interchangeable for the current purpose.
With this in mind, we first let $\Qart^\cap$ be the `meet' of the partitions $\Qart_\ib$ for $\ib \in I$:
we say that two vertices in the same part $W$ of $\Part$ are in the same part of $\Qart^\cap$
if and only if they lie in the same part of $\Qart_\ib$ for every $\ib \in I$ with $i_W=1$.
This is not our final definition, as we will also require that every part of $\Qart^\cap$ has size at least $D$.

To achieve this, note that $\Part$ has at most $k^2$ parts, and each part $X$ of $\Part$ 
is partitioned into at most $k$ parts by each $\Qart_\ib$ such that $i_X=1$,
so the number of parts of $\Qart^\cap$ is at most $k^2 \cdot k^{|I|} = k^{2 + r^{k-1}} < K$.
Thus there is some $z \in [K + 1]$ such that no part of $\Qart^\cap$ 
has size between $(z-1)(kK)^{z-1}D$ and $z(kK)^{z}D$.
Let $B$ be the union of all parts of size at most $z(kK)^zD$.
Then $|B| \le K \cdot (z-1)(kK)^{z-1}D$. 
Similarly to the end of step 2, we greedily choose a matching $M'$ in $H$ 
of size at most $|B|$ that covers every vertex of $B$ 
so that every edge $e \in M'$ has $\ib_\Part(e) \in L$.   
We will now restrict attention to $V' \sm V(M')$.

Next we introduce the `join' $\Qart^\cup$ of the partitions $\Qart_\ib$, 
which will allow us to avoid the problem of `hidden' approximate divisibility barriers
(recall Construction~\ref{NestedConstruct}).
Let $G$ be the graph whose vertices are the parts of $\Qart^\cap$ not contained in $B$,
where $Z, Z' \in \Qart^\cap$ are adjacent in $G$ if they are contained in the same part of some $\Qart_\ib$. 
Then the parts of $\Qart^{\cup}$ are formed by taking the union of the parts of $\Qart^\cap$ in each component of $G$.

As an aid to memory, in the remainder of the proof we tend to use the consistent notation 
$W$, $X$, $Y$, $Z$ for general parts of $\Part$, $\Qart^\cup$, $\Qart_\ib$ (for $\ib \in I$), $\Qart^\cap$ respectively,
i.e.\ earlier letters in the alphabet denote (potentially) larger parts.

\medskip

{\noindent \bf Step 4: Restricting to good vertices.}
We now delete all of the at most $2k\mu_1n$ vertices covered by $M \cup M'$. 
To avoid introducing more complicated notation, all of our notation
is to be now understood as referring to the undeleted vertices,
e.g.\ $V$ now refers to $V \sm (M \cup M')$. However, all variables (including $n$) remain unchanged in value.
Note that any part $Z$ of $\Qart^\cap$ had size at least $z(kK)^{z}D$
before deleting the vertices covered by $M'$ (otherwise we would have deleted it).
Since $|M'| \leq |B| \le K \cdot (z-1)(kK)^{z-1}D$, we now have
\begin{align} \label{formerA2}
|Z| &\geq z(kK)^{z}D - k|M'| \geq D.
\end{align}

Since we have deleted a number of vertices from $H$, we now show that the properties of $H$ from the statement of the
lemma and from step 2 are preserved (albeit with weaker constants) after these vertex deletions.
Note that (A1--A3) in the following claim are the respective analogues of $(\beta 1$--$\beta 3)$
for $H_\ib$ with respect to $\Part_\ib$ and $\Qart_\ib$.

\begin{claim} The following hold for every $\ib \in I$.
\begin{enumerate}[({A}1)]
\item At most $\mu_0 n^{k-1}$ $\Part_\ib$-partite $(k-1)$-sets $S \subseteq V(H_\ib)$ 
have $|H_\ib(S)| < (1/\ell + \gamma/2)n$.
\item $\Qart_\ib$ is $(1/\ell, 2c, \mu_2, \mu_5)$-robustly maximal with respect to $H_\ib$ for each $\ib \in I$.
\item Every vertex of $H_\ib$ lies in at least $\tfrac{1}{2}dn^{k-1}$ edges $e \in H_\ib$ with $\ib_{\Qart_\ib}(e) \in I_\ib$.
\item $\ib_\Part(V) \in L$.
\item Every part of $\Part$, $\Qart^{\cup}$ and each $\Qart_\ib$ has size at least $(1/\ell + \gamma/2)n$.
\item $I_{\Qart_\ib}^{\mu'}(H_\ib) = I_{\Qart_\ib}^{\mu_1}(H_\ib) = I_\ib$.
\item $\Part$ is $(1/\ell, 2c, \mu_2, \mu'')$-robustly maximal with respect to $H$, and $I = I_\Part^{\mu'}(H)$.
\end{enumerate}
\end{claim}

For (A1), recall from step 2 that it was true with $3\gamma/4$ in place of $\gamma/2$.
Lemma~\ref{RobustInherit} implies (A2), since $\Qart_\ib$ was $(1/\ell+2\gamma/3, c, \mu_0, \mu_6)$-robustly maximal 
with respect to $H_\ib$ before any vertices were deleted. The first part of (A7) follows by the same argument.
For (A3), recall that $M$ covered all vertices which were  in fewer than $dn^{k-1}$ 
edges of $H_\ib$ with $\ib \in I_\ib$ and at most $2k\mu_1n \leq dn/2$ vertices were deleted.
For (A4), note that it was true before any vertices were deleted by $(\beta 4)$, 
and $\ib_\Part(V(M \cup M')) \in L$ since $\ib_\Part(e) \in L$ for every $e \in M \cup M'$.

For (A5), recall that each part of any $\Qart_\ib$ had size at least $(1/\ell + 2\gamma/3)n$ 
before the deletions, and at most $2k\mu_1n \leq \gamma n/6$ vertices were deleted. 
The bounds for $\Qart^\cup$ and $\Part$ follow,
since every part of $\Qart^\cup$ contains a part of $\Qart_\ib$ for some $\ib \in I$, 
and $\Qart^\cup$ refines $\Part$.
For (A6), we observed before any deletions that $I_{\Qart_\ib}^{1/2(k\ell)^k}(H_\ib) = I_{\Qart_\ib}^{\mu_0}(H_\ib) = I_\ib$, and $|V(H_\ib)|$ had size at least $n$ both before and after the deletions. 
So if $\vb \in I_\ib$ then there were at least $n^k/2(k\ell)^k$ edges of $H_\ib$ of index $\vb$, 
at most $2k\mu_1 n^k$ of which were deleted, so $\vb \in I_{\Qart_\ib}^{\mu'}(H_\ib)$.
On the other hand, if $\vb \notin I_\ib$ then there were at most $\mu_0(kn)^k < \mu_1 n^k$ 
edges of $H_\ib$ of index $\vb$, so $\vb \notin I_{\Qart_\ib}^{\mu_1}(H_\ib)$. Since $I_{\Qart_\ib}^{\mu'}(H_\ib) \subseteq I_{\Qart_\ib}^{\mu_1}(H_\ib)$, this proves (A6). The same argument proves the second part of (A7), completing the proof of the claim.

\medskip

{\noindent \bf Step 5: Choosing sizes.}
In this step we determine how many vertices each of our final random subgraphs
should choose from each part of $\Qart^\cap$ to ensure that the analogue of $(\beta 4)$ holds.
We accomplish this in three stages.
Firstly, we choose rough targets for the number of vertices to be contained in each subgraph.
Secondly, we determine how many vertices are to be chosen from each part of $\Qart^\cup$.
Thirdly, we determine how many vertices are to be chosen from each part of $\Qart^\cap$.

\begin{claim} \label{rhoclaim} There exist integers $\rho_\ib$ for each $\ib \in I$ which satisfy
\begin{enumerate}[({B}1)]
\item $\sum_{\ib \in I} \rho_\ib \ib = \ib_\Part(V)$,
\item $\rho_\ib \geq \eta n$ for each $\ib \in I$.
\end{enumerate}
\end{claim} 

To prove Claim~\ref{rhoclaim}, for each $X \in \Part$ we start by reserving $N := \lceil 2\eta n \rceil$ vertices of $X$
for each $\ib \in I$ with $i_X=1$. Recall that for any part $X$ of $\Part$ there are exactly $r^{k-2}$ vectors $\ib \in I$ such that $i_X =1$, so exactly $r^{k-2}N$ vertices are reserved from each part of $\Part$. 
Let $V'$ be the set of unreserved vertices and write $V'_j = V' \cap V_j$; 
recalling that each $V_j$ now has size $n - |M \cup M'|$ after the deletions, 
each $V'_j$ has size $n' := n - |M \cup M'| - r^{k-2}N$.
Also, by (A5) at least $n'/k$ vertices remain unreserved in each part of $\Part$. 
Since $I$ is full with respect to $\Part'$, by Proposition~\ref{applyfarkas} we deduce that $\ib_\Part(V') \in PC(I)$. 
That is, we may fix reals $\lambda_\ib \geq 0$ for each $\ib \in I$ 
so that $\sum_{\ib \in I} \lambda_\ib \ib = \ib_\Part(V')$. 

Let $a_\ib = \lfloor \lambda_\ib \rfloor$ for each $\ib \in I$. 
Then $\vb = \ib_\Part(V') - \sum_{\ib \in I} a_\ib \ib$ is a sum of $|I| \le r^{k-1}$
vectors of length at most $k$, so $\vb \in B(\0, kr^{k-1})$. 
We also have $\vb \in L$, since $\ib_\Part(V')$ was obtained by subtracting some vectors in $I$
from $\ib_\Part(V)$, and $\ib_\Part(V) \in L$ by (A4). 
So by Proposition~\ref{vector_as_small_sum} we may choose integers $b_\ib$ with $|b_\ib| \leq K$ for each $\ib \in I$ so that
$\sum_{\ib \in I} b_\ib \ib = \vb$. 

Let $\rho_\ib = a_\ib + b_\ib + N$ for each $\ib \in I$. 
Then (B2) holds as $a_\ib \geq 0$ and $|b_\ib| \leq K$ for each $\ib \in I$, and (B1) holds as
$$\sum_{\ib \in I} \rho_\ib \ib = \sum_{\ib \in I} a_\ib \ib + \sum_{\ib \in I} b_\ib \ib + \sum_{\ib \in I} N\ib 
= \ib_\Part(V') + \sum_{\ib \in I} N\ib = \ib_\Part(V). $$
This completes the proof of Claim~\ref{rhoclaim}.
\medskip

The vector $\nb^\cup_\ib$ with respect to $\Qart^{\cup}$ obtained in the next claim 
determines how many vertices the random subgraph for $\ib$ will take from each part of $\Qart^{\cup}$. 
Note that (C1) ensures that the sizes are correct to be a partition of $V$,
(C2) that there is no divisibility obstruction to our later choice of $\Qart^\cap$, and
(C3) that the sizes are in roughly equal proportion from each part of $\Qart^\cup$ 
which is a subset of a part of $\Part_\ib$, and zero for any other part of $\Qart^\cup$.

The proof of the claim proceeds through three stages. 
In the first stage, we use the constants $\rho_\ib$ and the part sizes of each $\Qart_\ib$ to determine provisional values for each $\nb_\ib^\cup$ which satisfy (C1) and (C3).
In the second stage we `snap' each $\nb_\ib^\cup$ onto a nearby lattice point so that (C2) is satisfied, preserving (C3) (although (C1) may no longer hold). 
In the final stage we make further adjustments to restore (C1) while preserving (C2) and (C3).

\begin{claim} \label{cupclaim}
There exist vectors $\nb^\cup_\ib$ with respect to $\Qart^{\cup}$ for $\ib \in I$ which satisfy 
\begin{enumerate}[({C}1)]
\item $\sum_{\ib \in I} \nb^\cup_\ib = \ib_{\Qart^{\cup}}(V)$,
\item $\nb^\cup_\ib \in L_{\Qart^{\cup}}^{\mu_6}(H_\ib)$ for every $\ib \in I$, and
\item For any $\ib \in I$, any part $W$ of $\Part$ and any part $X \subseteq W$ of $\Qart^{\cup}$ we have
$$ (n^\cup_\ib)_X = 
\begin{cases}
\rho_\ib \tfrac{|X|}{|W|} \pm K &\textrm{ if }~i_W = 1, \\ 
0 &\textrm{ otherwise.}
\end{cases}$$
\end{enumerate}
\end{claim} 

To prove the claim, for each $\ib \in I$ we choose a vector $\nb_\ib$ with respect to $\Qart_\ib$ by taking $(n_\ib)_Y$
to be either $\lfloor \rho_\ib |Y|/|W|\rfloor$ or $\lceil \rho_\ib |Y|/|W| \rceil$ for each $W \in \Part_\ib$ 
and $Y \in \Qart_\ib$ with $Y \subseteq W$. We make these choices so that $\sum_{Y \subseteq W} (n_\ib)_Y = \rho_\ib$ 
for each $W \in \Part$ with $i_W=1$; this is possible since 
$$\sum_{Y \in \Qart_\ib : Y \subseteq W} \frac{\rho_\ib |Y|}{|W|} = \frac{\rho_\ib |W|}{|W|} = \rho_\ib.$$
Now, for any $\ib \in I$, observe that this requirement implies that $\nb_\ib \in L_{\max}^{\Qart_\ib\Part_\ib}$. So by Proposition~\ref{properties0}(\ref{properties:index}) (with $\Qart_\ib$ and $\Part_\ib$ in place of $\Part$ and $\Part'$ respectively) we may choose parts $Y, Y'$ of
$\Qart_\ib$ so that $\nb^1_\ib := \nb_\ib - \ub_Y + \ub_{Y'} \in L_\ib = L_{\Qart_\ib}^{\mu'}(H_\ib)$, 
where the final equality follows from (A6). 
Let $\Qart^0_\ib$ be the partition of $V$ whose parts are those of $\Qart_\ib$ 
and those of $\Qart^\cup$ that are disjoint from $V(H_\ib)$.
Similarly, let $\nb^0_\ib$ be the vector with respect to $\Qart^0_\ib$ corresponding to $\nb^1_\ib$, 
that is, with additional zero co-ordinates corresponding to the parts of $\Qart^\cup$ that are disjoint from $V(H_\ib)$. 
Then $\nb^0_\ib \in L_{\Qart^0_\ib}^{\mu_6}(H_\ib)$. Finally, let $\nb'_\ib := (\nb^0_\ib \mid \Qart^\cup)$, so $\nb'_\ib \in L_{\Qart^\cup}^{\mu_6}(H_\ib)$ by Proposition~\ref{InverseRestriction}(i).

For any part $W$ of $\Part$ and any part $X \subseteq W$ of $\Qart^{\cup}$,
since by (A5) there are at most $k$ parts $Y \sub X$ of $\Qart_\ib$ we have
\begin{equation} \label{eq:cupclaim}
 (n'_\ib)_X = \sum_{Y \in \Qart_\ib : Y \sub X} \brac{\rho_\ib \tfrac{|Y|}{|W|} \pm 1} = 
\begin{cases}
\rho_\ib \tfrac{|X|}{|W|} \pm 2k &\textrm{ if }~i_W = 1,  \\ 
0 &\textrm{ otherwise.}
\end{cases}
\end{equation}

Now, crucially, we have $\ib_{\Qart^\cup}(V) \in L^{\mu_6}_{\Qart^\cup}(H)$. Indeed, $\Part$ is $(1/\ell, 2c, \mu_6, \mu'')$-robustly maximal with respect to $H$ by (A7), $\Qart^\cup$ is a refinement of $\Part$ with parts of size at least $n/\ell \geq 2c|V(H)|$ by (A5), and $(\ib_{\Qart^\cup}(V) \mid \Part) = \ib_\Part(V) \in L = L^{\mu''}_\Part(H)$ by (A4) and (A7). So Proposition~\ref{robmaxinverse} implies that $\ib_{\Qart^\cup}(V) \in L^{\mu_6}_{\Qart^\cup}(H)$, as desired.
Since $\nb'_\ib \in L_{\Qart^\cup}^{\mu_6}(H_\ib) \subseteq L_{\Qart^\cup}^{\mu_6}(H)$ for each $\ib \in I$,
we deduce that 
\[ \vb := \ib_{\Qart^\cup}(V) - \sum_{\ib \in I} \nb'_\ib \in L^{\mu_6}_{\Qart^\cup}(H).\]
Furthermore, for any $W \in \Part$ by (B1) we have $\sum_{\ib \in I: i_W=1} \rho_\ib = |W|$,
so for any $X \in \Qart^\cup$ with $X \subseteq W$ by (\ref{eq:cupclaim}) we have
\[ v_X = |X| - \sum_{\ib \in I : i_W = 1} \brac{ \rho_\ib \tfrac{|X|}{|W|} \pm 2k} = 0 \pm 2k|I|.\]
Since $|\Qart^\cup| \le k^2$ and $|I| \le r^{k-1}$ we have $\vb \in B(\0, 2k^3r^{k-1})$,
so we may apply Proposition~\ref{vector_as_small_sum} to obtain integers $a_{\ib'}$ 
with $|a_{\ib'}| \leq k^{-k} K$ for each $\ib' \in I^{\mu_6}_{\Qart^\cup}(H)$ such that 
$\sum_{\ib' \in I_{\Qart^\cup}^{\mu_6}(H)} a_{\ib'} \ib' = \vb$.

For each $\ib \in I$, we define $\nb^\cup_\ib := \nb'_\ib + \sum a_{\ib'} \ib',$
where the sum is taken over all $\ib' \in I_{\Qart^\cup}^{\mu_6}(H)$ with $(\ib' \mid \Part) = \ib$. 
Then $\nb^\cup_\ib$ is a linear combination of vectors in $L_{\Qart^\cup}^{\mu_6}(H_\ib)$, so (C2) holds.
Also, (C1) holds as
$$\ib_{\Qart^\cup}(V) - \sum_{\ib \in I} \nb^\cup_\ib 
= \ib_{\Qart^\cup}(V) - \sum_{\ib \in I} \nb'_\ib - \vb  = \0.$$
Finally, consider any $\ib \in I$, any part $W$ of $\Part$ with $i_W=1$ and any part $X \subseteq W$ of $\Qart^\cup$. Since there are at most $(k-1)^k$ vectors $\ib' \in I_{\Qart^\cup}^{\mu_6}(H)$ with $i'_X = 1$, by definition of $\nb^\cup_\ib$ we have 
\[|(n^\cup_\ib)_X -(n'_\ib)_X| \leq (k-1)^k \max_{\ib'} |a_{\ib'}| \leq (1-1/k)^k K;\] 
together with (\ref{eq:cupclaim}) we have (C3).
This completes the proof of Claim~\ref{cupclaim}. \medskip

In the final claim of this step we determine the required part sizes for the random subgraphs of the next step.
For proof, we start by choosing provisional values for $\nb^\cap_\ib$ to satisfy (D1) and (D3).
If (D2) fails then (C2) will imply that it can be remedied by transferrals
between pairs of parts of $\Qart$ that lie in the same part of $\Qart^\cup$.
By definition of $\Qart^\cup$ there is a path in the auxiliary graph $G$ between these parts.
By repeated `swaps' we can effectively move the required adjustment along the edges of the path, 
all the while preserving (D1) and (D3), until (D2) holds.
Throughout the proof, we identify the vectors $\nb^\cap_\ib$ and $\nb^\cup_\ib$
with their restrictions to the parts of $\Qart^{\cap}$ and $\Qart^{\cup}$
that are contained within parts of $\Part_\ib$ (they are zero on all other parts).

\begin{claim} \label{capclaim}
There exist vectors $\nb^\cap_\ib$ with respect to $\Qart^{\cap}$ for $\ib \in I$ which satisfy
\begin{enumerate}[({D}1)]
\item $\sum_{\ib \in I} \nb^\cap_\ib = \ib_{\Qart^{\cap}}(V)$,
\item $(\nb^\cap_\ib \mid \Qart_\ib) \in L_\ib$ for every $\ib \in I$, and
\item For any part $W$ of $\Part$ and any part $Z \subseteq W$ of $\Qart^{\cap}$ we have
$$ (n^\cap_\ib)_Z = 
\begin{cases}
\rho_\ib \frac{|Z|}{|W|} \pm 2K &\textrm{ if }~i_W = 1  \\ 
0 &\textrm{ otherwise.}
\end{cases}$$
\end{enumerate}
\end{claim}

To prove the claim, we start by choosing provisional values for $\nb^\cap_\ib$,
where for each $\ib \in I$, each part $X$ of $\Qart^{\cup}$ contained within a part of $\Part_\ib$,
and each part $Z \subseteq X$ of $\Qart^\cap$, we take $(n^\cap_\ib)_Z$ to be either 
$\lfloor (n^\cup_\ib)_X |Z|/|X| \rfloor$ or $\lceil (n^\cup_\ib)_X |Z|/|X| \rceil$, 
with choices made so that 
\begin{align}
\label{capclaimeq1} & \sum_{\ib \in I} (n^\cap_\ib)_Z = |Z| \text{ for each } Z \in \Qart^\cap, \text{ and } \\
\label{capclaimeq2} & \sum_{Z \subseteq X} (n^\cap_\ib)_Z = (n^\cup_\ib)_X \text{ for each } X \in \Qart^\cup.
\end{align}
To see that we can make such choices, fix $X \in \Qart^\cup$ and let $A$ be the matrix with
rows indexed by the index vectors $\ib \in I$ such that $X$ is contained within a part of $\Part_\ib$,
and columns indexed by the parts $Z \subseteq X$ of $\Qart^\cap$, 
where the $(\ib, Z)$ entry of $A$ is $(n^\cup_\ib)_X |Z|/|X|$.
Then the sum of column $Z$ is $|Z|$ by (C1),
and the sum of row $\ib$ is $(n^\cup_\ib)_X$.
Thus Theorem~\ref{Baranyai} implies that we can choose such values for $\nb^\cap_\ib$.
Note that Equation (\ref{capclaimeq2}) can be reformulated as 
\[(\nb^\cap_\ib \mid \Qart^\cup) = \nb^\cup_\ib.\]
Also, for any part $W$ of $\Part$ and any part $Z \subseteq W$ of $\Qart^\cap$,
it follows from (C3) that $(n^\cap_\ib)_Z$ is equal to 
$\rho_\ib {|Z|}/{|W|} \pm (K+1) \textrm{ if }~i_W = 1$ and is zero otherwise.

Equation (\ref{capclaimeq1}) implies that the current values for $\nb^\cap_\ib$ satisfy (D1) and (D3); we now modify them to satisfy (D2).
Consider $\ib_1, \ib_2 \in I$, and suppose that $Z$ and $Z'$ are parts of $\Qart^\cap$ which are subsets of distinct parts $Y_1$ and $Y'_1$ respectively of $\Qart_{\ib_1}$, but that $Z$ and $Z'$ are subsets of the same part $Y_2$ of $\Qart_{\ib_2}$.
We define an \emph{$(\ib_1, \ib_2, Z, Z')$-swap} as the operation of increasing $(n_{\ib_1}^\cap)_Z$ and $(n_{\ib_2}^\cap)_{Z'}$ each by one, and decreasing $(n_{\ib_2}^\cap)_Z$ and $(n_{\ib_1}^\cap)_{Z'}$ by one (all other co-ordinates of $\nb_{\ib_1}^\cap$ and $\nb_{\ib_2}^\cap$ remain unchanged, as do all other vectors $\nb^\cap_\ib$).
Clearly $\sum_{\ib \in I} \nb^\cap_\ib$ is unchanged by any $(\ib_1, \ib_2, Z, Z')$-swap, 
as is $(\nb^\cap_\ib \mid \Qart_\ib)$ for any $\ib \neq \ib_1, \ib_2$.
Further, the operation has no effect on $(\nb_{\ib_2}^\cap \mid \Qart_{\ib_2})$, 
since $Z$ and $Z'$ are contained in the same part of $\Qart_{\ib_2}$.
However, $(\nb_{\ib_1}^\cap \mid \Qart_{\ib_1})$ \emph{is} affected; 
specifically, an $(\ib_1, \ib_2, Z, Z')$-swap adds $\ub_{Y_1} - \ub_{Y'_1}$ to $(\nb_{\ib_1}^\cap \mid \Qart_{\ib_1})$.
By performing several sequences of swaps, we shall ensure that (D2) is satisfied for each $\ib \in I$ in turn. 

Fix some $\ib \in I$. Recall from (C2) that $\nb^\cup_\ib \in L_{\Qart^{\cup}}^{\mu_6}(H_\ib)$.
Since $\mu_0 \ll \mu_6$, Proposition~\ref{InverseRestriction}(ii) applied with $\Qart_\ib$ and $\Qart^\cup$ in place of $\Qart$ and $\Part$ respectively implies that 
there exists $\nb_\ib^* \in L_\ib$ such that $(\nb_\ib^* \mid \Qart^\cup) = \nb_\ib^\cup$,
where we identify $\Qart^\cup$ with its restriction to $V(H_\ib)$
and $\nb^\cup_\ib$ with its restriction to parts of $\Qart^\cup[V(H_\ib)]$.
We let $\db_\ib = \nb_\ib^* - (\nb_\ib^\cap \mid \Qart_\ib)$ and 
note that $(\db_\ib \mid \Qart^\cup) = \0$, since $(\nb^\cap_\ib \mid \Qart^\cup) = \nb^\cup_\ib$.
Hence we may write $\db_\ib = \sum_{\vb \in \mathcal{D}_\ib} \vb$ for some sequence $\mathcal{D}_\ib$ of transferrals, 
such that $Y$ and $Y'$ are parts of $\Qart_\ib$ which are contained in the same part of $\Qart^\cup$ 
for each transferral $\ub_Y - \ub_{Y'} \in \mathcal{D}_\ib$.
Then $(\nb_\ib^\cap \mid \Qart_\ib) + \sum_{\vb \in \mathcal{D}_\ib} \vb = \nb_\ib^* \in L_\ib$.
We now reduce $\mathcal{D}_\ib$ to a sequence of at most $k-2$ transferrals,
maintaining the property that $(\nb_\ib^\cap \mid \Qart_\ib) + \sum_{\vb \in \mathcal{D}_\ib} \vb \in L_\ib$.
To see that this is possible, recall that $I_\ib$ is full with respect to $\Part_\ib$,
and by (A5) every part of $\Qart_\ib$ has size at least $(1/\ell + \gamma/2)n$,
so $|L^{\Qart_i\Part_i}_{\max}/L_\ib| \le k-1$ by Proposition~\ref{properties0}~(\ref{properties:parts}) and~(\ref{parGroupSize}).
Thus the required reduction of $\mathcal{D}_\ib$ follows from Proposition~\ref{abeliangroup}.

For each $\ub_Y - \ub_{Y'} \in \mathcal{D}_\ib$, we carry out the following procedure.
Choose parts $Z$ and $Z'$ of $\Qart^\cap$ with $Z \subseteq Y$ and $Z' \subseteq Y'$.
Since $Y$ and $Y'$ are contained in the same part of $\Qart^\cup$, 
by definition of $\Qart^\cup$ there is a path in the auxiliary graph $G$ from $Z$ to $Z'$,
i.e.\ we have $Z = Z_0, \dots, Z_p = Z'$ and $Y_1, \dots, Y_p$ for some $p \leq |\Qart^\cap| \leq k^{2+k^k}$,
such that $Z_j \in \Qart^\cap$ for $0 \le j \le p$, and for each $j \in [p]$
there exists $\ib_j \in I$ such that $Y_j \in \Qart_{\ib_j}$ and $Z_{j-1}, Z_j \subseteq Y_j$.

Note that each $Z_j$ is contained in the same $W \in \Part$ as $Y$ and $Y'$,
and that by the definition of $\mathcal{D}_\ib$ we have $i_W = 1$.
Hence for each $0 \leq j \leq p$ there exists a part $Y^*_j \in \Qart_\ib$ which contains $Z_j$. 
Thus $Y^*_0 = Y$ and $Y^*_p = Y'$.

For each $j \in [p]$ in turn we apply swaps as follows. If $Y^*_{j-1} = Y^*_j$ then there is no swap. 
Otherwise, we perform an $(\ib, \ib_j, Z_{j-1}, Z_j)$-swap; as noted above, 
the only effect of this is to add $\ub_{Y^*_{j-1}} - \ub_{Y^*_j}$ to $(\nb^\cap_\ib \mid \Qart_\ib)$.
So the net effect of performing these swaps for every member of $\mc{D}_\ib$ is to add $\ub_Y - \ub_{Y'}$ to $(\nb^\cap_\ib \mid \Qart_\ib)$. 
By choice of $Y$ and $Y'$, after these modifications we have $(\nb^\cap_\ib \mid \Qart_\ib) \in L_\ib$,
and crucially, $(\nb_{\ib'}^\cap \mid \Qart_{\ib'})$ is unchanged for any $\ib' \neq \ib$ 
and $\sum_{\ib \in I} \nb^\cap_\ib$ is unchanged.

We proceed in this manner for every $\ib \in I$; then the vectors $\nb^\cap_\ib$ obtained 
at the end of this process must satisfy (D2). Since (D1) held before we made any modifications, 
and $\sum_{\ib \in I} \nb^\cap_\ib$ is preserved by each modification, we conclude that (D1) still holds. 
Finally recall that for any $\ib \in I$, any part $W$ of $\Part_\ib$ and any part $Z \subseteq W$ of $\Qart^\cap$
our provisional values satisfied $(n^\cap_\ib)_Z = \rho_\ib {|Z|}/{|W|} \pm (K+1)$.
Recall that $|I| = r^{k-1}$; now (D3) follows, as we performed at most $p \leq k^{2+k^k}$ swaps for each 
$\vb \in \bigcup_{\ib \in I} \mathcal{D}_\ib$, so there were at most $p(k-2)r^{k-1} < K$ swaps in total,
and no swap changed any co-ordinate of $\nb^\cap_\ib$ by more than one.
This completes the proof of Claim~\ref{capclaim}. \medskip

{\noindent \bf Step 6: The random selection.}
In this final step, we now partition $V$ into disjoint sets $T_\ib$ for $\ib \in I$, where for each $Z \in \Qart^\cap$
the number of vertices of $T_\ib$ taken from $Z$ is $(n^\cap_\ib)_Z$. To ensure that this is possible, 
we require that $\sum_{\ib \in I} \nb^\cap_\ib = \ib_{\Qart^\cap}(V)$ 
and that each co-ordinate of each $\nb^\cap_\ib$ is non-negative. 
The first of these conditions holds by (D1).
For the second, observe that for any $\ib \in I$, any part $W$ of $\Part_\ib$ and any part $Z \subseteq W$ of $\Qart^\cap$ 
we have $\rho_\ib \geq \eta n$ by (B2), $|Z| \geq D$ by~(\ref{formerA2}) and $|W| \leq n$. 
So (D3) implies that 
\begin{equation} \label{capsize}
(n^\cap_\ib)_Z \geq \rho_\ib |Z|/|W| - 2K \geq \eta |Z|/2 \ge 0. 
\end{equation}

We choose such a partition uniformly at random. 
That is, for each part $Z$ of $\Qart^\cap$, we choose uniformly at random 
a partition of $Z$ into sets $Z_\ib$ for $\ib \in I$ so that $|Z_\ib| = (n^\cap_\ib)_Z$ for each $\ib \in I$. 
Then for each $\ib\in I$ we let
\[ T_\ib := \bigcup_{Z \in \Qart^\cap} Z_\ib, \ \
\hat{H}_\ib = H[T_\ib], \ \ \hat{\Part}_\ib = \Part_\ib[T_\ib] \ \ \text{ and } \ \ \hQart_\ib = \Qart_\ib[T_\ib].\] 

Note that $T_\ib \subseteq V(H_\ib)$, as by (D3) $\nb^\cap_\ib$ is zero on parts of $\Qart^\cap$ not contained in $V(H_\ib)$,
so $\hat{H}_\ib$ is a $\hat{\Part}_\ib$-partite $k$-graph on the vertex set $T_\ib$. 
Also, since $(\nb^\cap_\ib \mid \Qart_\ib) \in L_\ib$ by (D2), and every $\ib' \in I_\ib$ has $(\ib' \mid \Part) = \ib$, 
we have $(\nb^\cap_\ib \mid \Part) = t_\ib \ib$ for some integers $t_\ib$. 
For any $W \in \Part_\ib$ with $i_W = 1$ we have 
\begin{equation}\label{eq:t_ib}
t_\ib = ( \nb^\cap_\ib \mid \Part)_W = \sum_Z (n^\cap_\ib)_Z \stackrel{\textrm{(D3)}}{=} 
\sum_Z \left(\frac{\rho_\ib|Z|}{|W|} \pm 2K\right) = \rho_\ib \pm 2K^2,
\end{equation} 
where both sums are taken over all parts $Z$ of $\Qart^\cap$ with $Z \subseteq W$,
and we recall that $\Qart^\cap$ has at most $K$ parts. 
So in particular we have $\eta n/2 \leq t_\ib \leq n$ for any $\ib \in I$ by (B2).

Let $\eps_* = \mu_2$, $\mu_* = \mu_3$, $\mu_*' = \mu_4$, $c_* = 6 c / \eta$, $d_*=d/4$, and $\gamma_* = \gamma/4$.
 
\begin{claim} \label{choiceclaim}
For any $\ib \in I$ the following properties each hold with high probability. 
\begin{enumerate}[({E}1)]
\item All but at most $\eps_* t_\ib^{k-1}$ $\Part_\ib$-partite $(k-1)$-sets $S \subseteq T_\ib$ 
have $|\hat{H}_\ib(S)| \ge (1/(\ell-1) + \gamma_*)t_\ib$.
\item $\hQart_\ib$ is $(c_*, c_*, \mu_*, \mu'_*)$-robustly maximal with respect to $\hat{H}_\ib$.
\item $L_{\hQart_\ib}^{\mu_*}(\hat{H}_\ib) = L_\ib$, and any vertex of $T_\ib$ is in at least $d_* t_\ib^{k-1}$ edges $e \in \hat{H}_\ib$  
with $\ib_{\hQart_\ib}(e) \in L_{\hQart_\ib}^{\mu_*}(\hat{H}_\ib)$. 
\end{enumerate}
\end{claim}

We first recall from (A2) that $\Qart_\ib$ is $(1/\ell, 2c, \mu_2, \mu_5)$-robustly maximal with respect to $H_\ib$,
and from (\ref{capsize}) that $(n^\cap_\ib)_Z \geq \eta |Z|/2$ for any $Z \in \Qart^\cap$ contained in a part of
$\Part_\ib$, so Lemma~\ref{RobustRandom} implies that (E2) holds with high probability.
For (E1), recall from (A1) that all but at most $\mu_0 n^{k-1}$ $\Part_\ib$-partite $(k-1)$-sets $S \subseteq V(H_\ib)$ 
have $|H_\ib(S)| \ge (1/\ell + \gamma/2)n$. So it suffices to show that for any such $S$ we have
$|H_\ib(S) \cap T_\ib| \geq (1/(\ell-1) + \gamma_*)t_\ib$ with high probability. Let $W$ be the part of $\Part$ such that $H_\ib(S) \subseteq W$.
Note that $|W| \le (1 - 1/\ell - \gamma/2)n$ as each part of $\Part'$ is partitioned into $r \ge 2$ parts by $\Part$,
and each part of $\Part$ has size at least $(1/\ell + \gamma/2)n$ by (A5). Now
\begin{align*}
\Exp[|H_\ib(S) \cap T_\ib|] 
&= \sum_{Z \subseteq W} \Exp[|H_\ib(S) \cap Z_\ib|] 
= \sum_{Z \subseteq W} \frac{(n^\cap_\ib)_Z |H_\ib(S) \cap Z|}{|Z|}\\
&\stackrel{\textrm{(D3)}}{\geq} \frac{\rho_\ib |H_\ib(S)|}{|W|}  - 2K^2 \\
&\stackrel{(\ref{eq:t_ib})}{\geq} \frac{t_\ib(1/\ell + \gamma/2)}{1 - 1/\ell - \gamma/2} - 4K^2 \\ 
&\geq (1 + \eps)\left(\frac{1}{\ell-1} + \gamma_*\right)t_\ib.
\end{align*}
Since $|H_\ib(S) \cap T_\ib|$ is a sum of independent hypergeometric random variables,
with high probability $|H_\ib(S) \cap T_\ib| \geq (1/(\ell-1) + \gamma_*)t_\ib$ by Corollary~\ref{SumOfHyper}.

For (E3), recall from (A6) that $L_\ib = L_{\Qart_\ib}^{\mu_1}(H_\ib) = L_{\Qart_\ib}^{\mu'}(H_\ib)$. For any $\ib' \in L_{\hQart_\ib}^{\mu_*}(\hat{H}_\ib)$ there are at least $\mu_* (kt_\ib)^k > \mu_1 |V(H_\ib)|^k$
edges $e \in \hat{H}_\ib \subseteq H_\ib$ with $\ib_{\Qart_\ib}(e) = \ib'$, so we have $L_{\hQart_\ib}^{\mu_*}(\hat{H}_\ib) \sub L_{\Qart_\ib}^{\mu_1}(H_\ib)$. Now suppose that $\ib' \in L_{\Qart_\ib}^{\mu'}(H_\ib)$, 
so there are at least $\mu'|V(H_\ib)|^k \geq \mu' n^k$ edges $e \in H_\ib$ with $\ib_{\Qart_\ib}(e) = \ib'$.
Recall that $\Qart^\cap$ has at most $K$ parts, so at most $K \eps (kn)^k$ of these edges contain a vertex from a part of $\Qart^\cap$ of size at most $\eps n$. 
Fix one of the remaining edges $e = \{x_1,\dots,x_k\}$ and for each $j \in [k]$
let $Z_j$ and $W_j$ be the parts of $\Qart^\cap$ and $\Part$ respectively which contain $x_j$. 
Since $|Z_j| \geq \eps n$ for each $j \in [k]$, the probability that $e \in \hat{H}_\ib$ equals
\begin{equation*}
\prod_{j \in [k]} \frac{(n^\cap_\ib)_{Z_j}}{|Z_j|} 
\stackrel{\textrm{(D3)}}{=} \prod_{j \in [k]} \brac{\frac{\rho_\ib}{|W_j|} \pm \frac{2K}{\eps n}}
\stackrel{(\ref{eq:t_ib})}{>}  \frac{t_\ib^{k}}{n^{k}} - \frac{3kK}{\eps n}.
\end{equation*}
Therefore the expected number of edges $e \in \hat{H}_\ib$ with $\ib_{\Qart_\ib}(e) = \ib'$
is at least $(\mu' n^k-K\eps (kn)^k)((t_\ib/n)^k - 3kK/\eps n) > \mu' t_\ib^k/2$.
By Corollary~\ref{ApplyAzuma}, with high probability there 
are at least $\mu' t_\ib^k/3 \geq \mu_* |V(\hat{H}_\ib)|^k$ such edges, so we have $\ib' \in L_{\hQart_\ib}^{\mu_*}(\hat{H}_\ib)$, proving that $L_{\hQart_\ib}^{\mu_*}(\hat{H}_\ib) = L_\ib$ as claimed.

Now consider any vertex $x \in T_\ib$ and let $E(x)$ denote the number of edges $e \in \hat{H}_\ib$ with $\ib_{\Qart_\ib}(e) \in L_\ib$ which contain $x$. 
To estimate $E(x)$, recall from (A3) that $x$ is in at least 
$\tfrac{1}{2}dn^{k-1}$ edges $e \in H_\ib$ with $\ib_{\Qart_\ib}(e) \in L_\ib$.
The same argument as above shows that $\Exp(E(x)) \geq \tfrac{1}{3}dt_\ib^{k-1}$
and so with high probability $E(x) \geq \tfrac{1}{4}dt_\ib^{k-1}$.
This implies (E3), as $L_{\hQart_\ib}^{\mu_*}(\hat{H}_\ib) = L_\ib$,
so completes the proof of the claim.

\medskip

To summarise, for each $\ib \in I$ we have the following: 
\begin{itemize}
\item $\hat{H}_\ib$ is a $\hat{\Part}_\ib$-partite $k$-graph on $T_\ib$ with parts of size $t_\ib$.
\item At most $\eps_* t_\ib^{k-1}$ $\hat{\Part}_\ib$-partite $(k-1)$-sets $S \subseteq T_\ib$ 
have $|\hat{H}_\ib(S)| < (1/(\ell-1) + \gamma_*)t_\ib$.
\item $\hQart_\ib$ is a partition of $T_\ib$ which refines $\hat{\Part}_\ib$ and is $(c_*, c_*, \mu_*, \mu'_*)$-robustly maximal with respect to $\hat{H}_\ib$. 
\item Any vertex of $T_\ib$ is in at least $d_* t_\ib^{k-1}$ edges $e \in \hat{H}_\ib$  
with $\ib_{\hQart_\ib}(e) \in L_\ib = L_{\hQart_\ib}^{\mu_*}(\hat{H}_\ib)$. 
\item $\ib_{\hQart_\ib}(T_\ib) = (\nb^\cap_\ib \mid \Qart_\ib) \in L_\ib = L_{\hQart_\ib}^{\mu_*}(\hat{H}_\ib)$ by (D2).
\end{itemize}
Since 
$$ 1/t_\ib \ll \eps_* \ll \mu_* \ll \mu'_* \ll c_*, d_* \ll \gamma_*, 1/k,$$
we conclude that $\hat{H}_\ib$ contains a perfect matching $M_\ib$ by our inductive hypothesis. 
Therefore $M \cup M' \cup \bigcup_{\ib \in I} M_\ib$ is a perfect matching in (the original) $H$.
This completes the proof of Lemma~\ref{kPartiteLemma}.
\end{proof}

\section{Proof of the structure theorem} \label{deduceThm}

In this section we use Lemma~\ref{NonPartiteLemma} to prove Theorem~\ref{PMNeccSuff}.
We start by proving an easier version in the first subsection, which is not algorithmic, 
but has a cleaner statement, and will be used in the proof of the full theorem.
In the second subsection we prove a weaker version of the algorithmic result,
which provides some details omitted from the extended abstract of this paper~\cite{KKMabs}.
The third and fourth subsections contain some technical preliminaries for the main result:
some analysis of how full lattices behave under merging of parts,
and some results on subsequence sums in abelian groups.
The final subsection contains the proof of Theorem~\ref{PMNeccSuff}.

\subsection{An easier version.}
In this subsection we prove the following easier version of Theorem~\ref{PMNeccSuff}.

\begin{thm} \label{WeakerPMNS} 
Under Setup~\ref{setup}, $H$ has a perfect matching if and only if every full pair $(\Part, L)$ for $H$ is soluble.
\end{thm}

While the statement of this version is more appealing than that of Theorem~\ref{PMNeccSuff},
it is impractical to use directly for our algorithm, as the number of full pairs for $H$ may be exponentially large.
The forward implication of Theorem~\ref{WeakerPMNS} is easy to prove: if $H$ has a perfect matching $M$ then $\ib_\Part(V(H) \sm V(M)) = \0 \in L$, and Lemma~\ref{EquivSol} implies that $(\Part, L)$ is soluble.
In fact the forward implication of Theorem~\ref{PMNeccSuff} now follows immediately; we simply observe from the definition that there is no $C$-certificate for $H$ for any $C \geq 0$.
We now consider the backward implication of Theorem~\ref{WeakerPMNS}.
We need the following non-partite analogue of Proposition~\ref{properties} (we omit the similar proof). 

\begin{prop} \label{NonParproperties} 
Suppose that $1/n \ll \mu, \eps \ll \psi, d, c, 1/k$, and let $V$ be a set of size $n$. 
Let $H$ be a $k$-graph on $V$ in which at most $\eps n^{k-1}$ $(k-1)$-sets $S \sub V$ have $d_H(S) < dn$, and let $\Part$ be a partition of $V$ with parts of size at least $cn$ 
such that $L^\mu_\Part(H)$ is transferral-free. Then the following properties hold. 
\begin{enumerate}[(i)]
\item\label{NonParproperties:full} $I_\Part^\mu(H)$ is full.
\item\label{NonParproperties:codegree} For any $\ib \in I^\mu_\Part(H)$ at most $\psi n^{k-1}$ $(k-1)$-sets $S \subseteq V$ 
with $\ib_\Part(S) = \ib - \ub_X$ for some $X \in \Part$ have $|H_\ib(S)| < (d - \psi) n$. 
\item\label{NonParproperties:final} Each part of $\Part$ has size at least $(d - \psi) n$.
\item\label{NonParproperties:lattice} $I^{dc^{k-1}/2k!}_\Part(H) = I^\mu_\Part(H)$.
\item\label{NonParproperties:vdegree} For any $\ib \in I^\mu_\Part(H)$ at most $\psi n$ vertices of $V(H_\ib)$ lie in fewer than $c^{k-2}dn^{k-1}/2(k-1)!$ edges of $H_\ib$.
\end{enumerate}
\end{prop} 

Next we show that we can modify a robustly maximal partition 
to arrange that every vertex belongs to many edges with index on the robust edge-lattice.
Note that it would not help to use the technique employed in previous sections 
of removing a small matching covering the bad vertices, 
as we will require the full vertex set to remain intact 
in order to use the fact that every full pair is soluble.

\begin{lemma} \label{MoveVertices} Suppose that $k \geq 3$ and $1/n \ll \eps, \mu \ll \mu' \ll c \ll \gamma \ll 1/k$. 
Let~$H$ be a $k$-graph on $n$ vertices such that
at most $\eps n^{k-1}$ $(k-1)$-sets $A \subseteq V(H)$ have $d_H(A) < (1+\gamma)n/k$
and every vertex is in at least $\gamma n^{k-1}$ edges of $H$. Also let $\Part$ be a partition of $V(H)$ that is $(c, c, \mu, \mu')$-robustly maximal with respect to $H$. 
Then there exists a partition $\Part'$ of $V(H)$ with at most $k-1$ parts such that
\begin{enumerate}[(i)]
\item $\Part'$ is $(4c, 2c, 4\sqrt{\mu}, \mu'/2)$-robustly maximal with respect to $H$, 
\item $L_{\Part'}^{6\sqrt{\mu}}(H) = L_{\Part'}^{4\sqrt{\mu}}(H) = L_\Part^\mu(H)$, and 
\item every vertex is in at least $\gamma n^{k-1}/3k$ edges $e \in H$ with $\ib_{\Part'}(e) \in L_{\Part'}^{4\sqrt{\mu}}(H)$.
\end{enumerate}
\end{lemma}

\begin{proof}
First we note by Proposition~\ref{NonParproperties}(iii) that every part of $\Part$
has size at least $n/k + \gamma n - cn > n/k$. So $|\Part| \leq k-1$, and $\Part$ is $(1/k, c, \mu, \mu')$-robustly maximal with respect to $H$.
Next let $B$ be the set of vertices that belong to
fewer than $\gamma n^{k-1}/2$ edges $e \in H$ with $\ib_\Part(e) \in L_\Part^\mu(H)$.
We claim that $|B| < \sqrt{\mu} n/2$. For otherwise, by the inclusion-exclusion principle 
the number of edges $e$ that intersect $B$ and have $\ib_\Part(e) \notin L_\Part^\mu(H)$ is at least
$(\sqrt{\mu} n/2) \cdot \gamma n^{k-1}/2 - \binom{\sqrt{\mu} n/2}{2}\binom{n}{k-2} \geq \sqrt{\mu} \gamma n^k/5$.
But this is a contradiction, as there are fewer than $k^k \mu n^k$ such edges in total, 
so we must have $|B| < \sqrt{\mu} n/2$. 

Now consider any vertex $v \in B$ and let $X$ be the part of $\Part$ containing $v$. 
For each edge $e \in H$ we let $g_X(e)$ be the unique part $X'$
such that $\ib_\Part(e) - \ub_X + \ub_{X'} \in L_\Part^\mu(H)$; this is well-defined as $I_\Part^\mu(H)$ is full
by Proposition~\ref{NonParproperties}(i), so we can apply Lemma~\ref{NonParproperties0}(\ref{NonParproperties:index}).
By the pigeonhole principle there exists $X(v) \in \Part$ such that $g_X(e) = X(v)$ 
for at least $\gamma n^{k-1}/2k$ edges $e \in H$ containing $v$. 
Let $\Part'$ be the partition obtained from $\Part$ by moving $v$ into $X(v)$ for each $v \in B$.
Then $\Part'$ is $\sqrt{\mu}$-close to $\Part$, so by Lemma~\ref{RobustInherit}, 
$\Part'$ is $(4c, 2c, 4\sqrt{\mu}, \mu'/2)$-robustly maximal with respect to $H$. Also, by Proposition~\ref{similarlattices} we have $L^{\mu^{1/3}}_\Part(H) \subseteq L^{6\sqrt{\mu}}_{\Part'}(H) \subseteq L^{4\sqrt{\mu}}_{\Part'}(H) \subseteq L^\mu_\Part(H)$; since $L^{\mu^{1/3}}_\Part(H) = L^\mu_\Part(H)$ by Proposition~\ref{NonParproperties}(iv) we have $L^\mu_\Part(H) = L^{4\sqrt{\mu}}_{\Part'}(H) = L^{6\sqrt{\mu}}_{\Part'}(H)$.
Finally, for any vertex $v \in V(H)$ the number of edges $e \in H$ containing $v$ with $\ib_{\Part'}(e) \in L_{\Part'}^{4\sqrt{\mu}}(H)$
is at least $\gamma n^{k-1}/2k - \sqrt{\mu} n^{k-1} \geq \gamma n^{k-1}/3k$. 
\end{proof}

\begin{proof}[Proof of Theorem~\ref{WeakerPMNS}] 
We have already proved the forward implication, so it remains to prove the backward implication.
Consider $H$ as in Setup~\ref{setup} and suppose that every full pair $(\Part, L)$ is soluble.
Introduce a new constant $c$ with $\eps \ll c \ll \gamma$, fix $s = \lfloor 1/c \rfloor$ and introduce further new constants $\mu_1, \dots, \mu_{s+1}$ such that 
$\eps \ll \mu_1 \ll \dots \ll \mu_{s+1} \ll c$. 
By Proposition~\ref{RobustExistence} there exists $t \in [s]$ and a partition $\Part$ of $V(H)$ 
that is $(c, c, \mu_t, \mu_{t+1})$-robustly maximal with respect to $H$; we write $\mu = \mu_t$ and $\mu' = \mu_{t+1}$.
Applying Lemma~\ref{MoveVertices} we obtain a partition $\Part'$ of $V(H)$ with at most $k-1$ parts 
which is $(4c, 2c, 4\sqrt{\mu}, \mu'/2)$-robustly maximal with respect to $H$, 
such that $L := L^{6\sqrt{\mu}}_{\Part'}(H) = L_{\Part'}^{4\sqrt{\mu}}(H) = L_\Part^\mu(H)$ and every vertex $v \in V(H)$ 
is in at least $\gamma n^{k-1}/3k$ edges $e \in H$ with $\ib_{\Part'}(e) \in L$.
Note that $(\Part', L)$ is a full pair for $H$ by Proposition~\ref{NonParproperties}(\ref{NonParproperties:full}),
so by assumption it has a solution $M$. 
Taking $V' := V \sm V(M)$ we then have $\ib_{\Part'}(V') \in L$. 
Also, by Proposition~\ref{similarlattices} we have 
$L = L^{6\sqrt{\mu}}_{\Part'}(H) \subseteq L_{\Part'[V']}^{5\sqrt{\mu}}(H[V']) \subseteq L_{\Part'}^{4\sqrt{\mu}}(H) = L$, 
so $L = L_{\Part'[V']}^{5\sqrt{\mu}}(H[V'])$, and by 
Lemma~\ref{RobustInherit} $\Part'[V']$ is $(3c, 3c, 5\sqrt{\mu}, \mu'/3)$-robustly 
maximal with respect to $H'$. We may therefore apply 
Lemma~\ref{NonPartiteLemma} to $H \sm V(M)$ with 
$\gamma/2$, $3c$, $5\sqrt{\mu}$, $\mu'/3$, $\gamma/4k$ and $2k^{k-1}\eps$ in place 
of $\gamma$, $c$, $\mu$, $\mu'$, $d$ and $\eps$ respectively to obtain a 
perfect matching in $H[V']$. 
Together with $M$ this gives a perfect matching in $H$. 
\end{proof}

\subsection{A slower algorithm.} 
As a warmup for Theorem~\ref{PMNeccSuff}, and to provide the details for the 
result given in the extended abstract of this paper~\cite{KKMabs},
we first prove a weaker result in which $2k(k-3)$ is replaced by $2k^{k+3}$;
this is sufficient for our algorithmic result, 
but the running time is of course significantly worse. (Actually, in the extended abstract we assumed instead that any full pair $(\Part^*, L^*)$ for which any matching of edges $e \in H$ with $\ib_{\Part^*}(e) \notin L^*$ has size less than $2k^{k+2}$ is soluble, but the proof shows that this gives the desired conclusion).
This weaker result follows from Theorem~\ref{WeakerPMNS} and the following result.

\begin{lemma} \label{AbstractSoluble} 
Suppose that $k \geq 3$ and $H$ is a $k$-graph such that every $2k^{k+3}$-far full pair for $H$ is soluble.
Then every full pair for $H$ is soluble.
\end{lemma}

\begin{proof}
Consider any full pair $(\Part,L)$ for $H$. Let $I$ be a full set in $L$, and let $S$ be the set of $k$-vectors $\ib$ such that
$H$ contains at least $2k^2$ disjoint edges $e \in H$ with $\ib_\Part(e) = \ib$.
Let $L'$ be the lattice generated by $I \cup S$ (so $L \subseteq L'$).
Consider the relation $\sim$ on $\Part$ defined by $X\sim Y$ if and only if $\ub_X - \ub_Y \in L'$; 
it is clear that $\sim$ is an equivalence relation.
Let $\Part^*$ be the partition of $V$ formed by taking unions of equivalence classes under $\sim$ 
and let $L^* = (L' \mid \Part^*)$. 

We claim that $(\Part^*,L^*)$ is a $2k^{k+3}$-far full pair for $H$.
To see this, we first show that $L^*$ is transferral-free. 
Suppose for a contradiction that $\ub_W - \ub_Z \in L^*$ for some distinct parts $W, Z \in \Part^*$.
By definition there exists $\vb \in L'$ such that $(\vb \mid \Part^*) = \ub_W - \ub_Z$.
Consider such $\vb$ that minimises $\sum_{X \in \Part} |v_X|$. Then we cannot have $v_X>0$ and $v_Y<0$
for distinct parts $X,Y \in \Part$ that are contained in the same part of $\Part^*$.
This implies that $\vb = \ub_X - \ub_Y$ for some $X,Y \in \Part$ with $X \sub W$ and $Y \sub Z$,
so $X \sim Y$, which contradicts the definition of $\Part^*$, so $L^*$ is transferral-free.
Next, the projections of vectors in $I$ show that $L^*$ is full.
Now, since there are at most $k^k$ possible values of $\ib_\Part(e)$ for an edge $e \in H$, by choice of $S$,  
any matching of edges $e \in H$ such that $\ib_{\Part^*}(e) \notin L^*$ has size less than $2k^{k+2}$. 
Therefore $(\Part^*,L^*)$ is a $2k^{k+3}$-far full pair for $H$, as claimed.

By assumption $(\Part^*,L^*)$ is soluble, so $H$ has a matching $M$ of size at most $k-2$ 
such that $\ib_{\Part^*}(V(H) \sm V(M)) \in L^*$. 
We claim that in fact $\ib := \ib_{\Part}(V(H) \sm V(M)) \in L'$. 
To see this, we apply Lemma~\ref{NonParproperties0}(\ref{NonParproperties:index})
to get $\ib - \ub_X + \ub_{X'} \in L$ for some parts $X,X'$ of $\Part$.
It follows that $\ib_{\Part^*}(V(H) \sm V(M)) - (\ub_X - \ub_{X'} \mid \Part^*) \in L^*$,
so $(\ub_X - \ub_{X'} \mid \Part^*) \in L^*$. 
Since $L^*$ is transferral-free, $(\ub_X - \ub_{X'} \mid \Part^*) = \0$, 
so $X$ and $X'$ are contained in the same part of $\Part^*$.
This implies $\ub_X - \ub_{X'} \in L'$, so $\ib \in L'$, as claimed.

Next let $G=L'/L$ and note that $|G| \le |G(\Part,L)| \le k-1$ by Lemma~\ref{GroupSize}.
There exists $r \in \N$ such that we can write $\ib = \vb + \sum_{j \in [r]} \ib_j$ for some $\vb \in L$, 
where either $\ib_j \in S$ or $-\ib_j \in S$ for $j \in [r]$.
Then $\ib + L = \sum_{j \in [r]} (\ib_j+L) \in G$, where without loss of generality $\ib_j \in S$ for each $j \in [r]$,
as for any $\ib' \in -S$ we can replace $\ib' + L$ by $|G| - 1$ copies of $-\ib' + L$.
By Proposition~\ref{abeliangroup}, without loss of generality $r \le k-2$.
Now by definition of $S$, we can greedily extend $M$ to a matching $M'$
where for each $j \in [r]$ we add an edge $e$ with $\ib_\Part(e) = \ib_j$.
Then $\ib_{\Part}(V(H) \sm V(M')) \in L$, so $(\Part,L)$ is soluble by Lemma~\ref{EquivSol}.
\end{proof}

\subsection{Merging.}
In Lemma~\ref{ImplySoluble} we will improve Lemma~\ref{AbstractSoluble}
by replacing $2k^{k+3}$ by $2k(k-3)$.
This will require a more careful analysis 
of the lattices that can be obtained by merging parts.

\begin{defin} \label{DefMerge}
Suppose $(\Part, L)$ is a full pair for a $k$-graph $H$ and write $G = G(\Part, L)$.
Fix an identification of $\Part$ with $G$ and an element $g_0 \in G$ so that $L = L(G,g_0)$ (using Theorem~\ref{FullStructure}).
For any subgroup $K$ of $G$, let $\Part_K$ be the partition obtained from $\Part$
by merging any two parts identified with elements in the same coset of $K$.
Let $L_K = (L \mid \Part_K)$.
\end{defin}

For example $\Part_0 = \Part$, $L_0=L$, $\Part_G = \{V(H)\}$ and $L_G=k\Z$.

\begin{lemma} \label{LemmaMerge}
Under the setup of Definition~\ref{DefMerge}, 
$\Part_K$ is well-defined, $(\Part_K,L_K)$ is a full pair, and 
\[L_K = L(G/K,g_0+K).\]
Furthermore, for $\ib \in L^\Part_{\max}$ we have $(\ib \mid \Part_K) \in L_K$
if and only if $R_G(\ib) \in K$. 
\end{lemma}

\begin{proof}
To see that $\Part_K$ is well-defined, recall from Remark~\ref{Translation}
that the identification $\pi: \Part \to G$ is determined up to translation by $G$;
we merge $X$ and $Y$ if and only if $\pi(X)-\pi(Y) \in K$,
and this property is invariant under any translation of $\pi$.
Note that there is an induced identification of $\Part_K$ with $G/K$.
Now we show that $L_K = L(G/K,g_0+K)$;
then $(\Part_K,L_K)$ is a full pair by Lemma~\ref{FullConstructProps}.

Recall from Lemma~\ref{FullConstructProps} that $L(G,g_0)$ is
the set of index vectors $\ib \in L_{\max}^{\Part}$ with
$\sum_{g \in G} i_g g = (k^{-1} \sum_{g \in G} i_g) g_0$.
For any such $\ib$ we have 
\begin{align*}
\sum_{h \in G/K} (\ib \mid \Part_K)_h h 
& = \sum_{h \in G/K} \left(\sum_{g \in h} i_g\right) h  = \sum_{g \in G} i_g (g+K) \\
&  = \left(k^{-1} \sum_{g \in G} i_g\right) (g_0+K) \\
&  = \left(k^{-1} \sum_{h \in G/K} (\ib \mid \Part_K)_h \right) (g_0+K).
\end{align*}
Thus $(\ib \mid \Part_K) \in L(G/K,g_0+K)$. Since $R_G(\ib)$ = $\sum_{g \in G} i_g g - (k^{-1} \sum_{g \in G} i_g) g_0$ for any $\ib \in L_{\max}^\Part$, this calculation
also establishes the `Furthermore' statement.

Conversely, consider any $\jb \in L(G/K,g_0+K)$,
so that $\sum_{h \in G/K} j_h h = (k^{-1} \sum_{h \in G/K} j_h) (g_0+K)$.
For each $h \in G/K$ fix a representative $\hat{h} \in h$
and consider the index vector $\ib = \sum_{h \in G/K} j_h \ub_{\hat{h}}$ 
with respect to $\Part$. Then $(\ib \mid \Part_K) = \jb$ and 
\begin{align*} 
\sum_{g \in G} i_g g & = \sum_{h \in G/K} j_h \hat{h}
\in \left(k^{-1} \sum_{h \in G/K} j_h\right) (g_0+K) \\
& = \left(k^{-1} \sum_{g \in G} i_g\right) (g_0+K).
\end{align*}
Let $g' = \sum_{g \in G} i_g g - (k^{-1} \sum_{g \in G} i_g)g_0$, so $g' \in K$, and
let $\ib' = \ib - \ub_{g'} + \ub_{0_G}$. 
Then
\[\sum_{g \in G} i'_g g = -g' + \sum_{g \in G} i_g g 
= (k^{-1} \sum_{g \in G} i_g)g_0 = (k^{-1} \sum_{g \in G} i'_g)g_0,\]
so $\ib' \in L(G,g_0)$, and $(\ib' \mid \Part_K) = (\ib \mid \Part_K) = \jb$.
\end{proof}

\subsection{Subsequence sums in finite abelian groups.}
Under the setup of Definition~\ref{DefMerge}, given a matching $M$ in $H$, 
we define $\Sart_G(M)$ to be the set of residues of submatchings of $M$;
that is, \[\Sart_G(M) = \{R_G(V(M')) \mid M' \subseteq M\}.\]
Note that $\ib_\Part(V(H) \sm V(M')) \in L$ if and only if $R_G(V(H)) = R_G(V(M'))$,
and if $R_G(V(H)) \in \Sart_G(M)$ then $(\Part, L)$ is soluble by Lemma~\ref{EquivSol}.
Thus we have the following result.

\begin{prop} \label{CompleteSums}
Under the setup of Definition~\ref{DefMerge}, 
if there is some matching $M$ such that $\Sart_G(M) = G$ then $(\Part, L)$ is soluble.
\end{prop}

To analyse $\Sart_G(M)$ we require some lemmas on subsequence sums in abelian groups.
Roughly speaking, we can obtain sufficient lower bounds provided 
that $|\Sart_G(M)|$ increases by at least $1$ when we add an edge to $M$.
Our first lemma gives structural information when this is not the case.
First we introduce some notation for sequences.

\emph{For the remainder of this subsection we fix any finite abelian group $G$.}

\begin{defin}
Let $a$ be a sequence in $G$.
\begin{enumerate}[(i)]
\item We write $\Sart(a)$ for the set of sums of all subsequences of $a$. 
\item For any subgroup $K$ of $G$ we write $a+K$
for the sequence of $K$-cosets of terms in $a$.
\item For $j \ge 1$ we write $a_j$ for the $j$th term of $a$ (when it exists).
\item For $j \ge 0$ we write $a^j = (a_1,\dots,a_j)$.
\item We write $a \circ b$ for the concatenation of two sequences $a$ and $b$.
\end{enumerate}
\end{defin}

\begin{lemma} \label{UnionOfCosets} 
Suppose $a$ is a sequence in  $G$,
and $x \in G$ is such that $\Sart(a \circ (x)) = \Sart(a)$.
Then $\Sart(a)$ is a union of cosets of $\bgen{x}$.
\end{lemma}

\begin{proof} 
Consider $y \in \Sart(a)$. It suffices to show that $y+rx \in \Sart(a)$ for every $r \in \N$.
We prove this by induction on $r$. The base case $r=0$ is given. For the induction step,
if $y+rx \in \Sart(a)$ then $y+(r+1)x = (y+rx)+x \in \Sart(a\circ (x)) = \Sart(a)$. 
\end{proof}

\begin{rem} \label{Prime}
Lemma~\ref{ImplySoluble} follows directly from Lemma~\ref{UnionOfCosets} 
for full pairs $(\Part, L)$ such that $|G(\Part, L)|=p$ is prime.
Indeed, consider a greedy construction of a matching $M$, where at each step we add
an edge that increases the size of $\Sart_G(M)$, if possible.
The construction terminates in at most $|G|-1$ steps,
because either $\Sart_G(M)=G$ or no edge increases the size of $\Sart_G(M)$.
If $\Sart_G(M)=G$ then $(\Part, L)$ is soluble by Proposition~\ref{CompleteSums}.
On the other hand, if $\Sart_G(M) \ne G$ we claim that 
for any edge $e$ disjoint from $M$ we have $R_G(e)=0_G$, i.e.\ $\ib_\Part(e) \in L$.
For otherwise, since $p$ is prime we have $\bgen{R_G(e)}=G$,
so by Lemma~\ref{UnionOfCosets} $\Sart_G(M)$ is a union of cosets of $G$, 
i.e.\ $\Sart_G(M)=G$, which contradicts our assumption.
Note also that $|M| \le |G|-2 \le k-3$, so $|V(M)| \le k(k-3)$.
Thus $(\Part, L)$ is a $k(k-3)$-far full pair,
so is soluble by the assumption of Lemma~\ref{ImplySoluble}.
\end{rem}

Our next lemma gives a monotonicity property that will allow us to extend matchings,
while maintaining the coset structure for subsequence sums.

\begin{lemma} \label{MonotoneGroup} 
Suppose $a$ is a sequence in  $G$
and $\Sart(a)$ is a union of cosets of some subgroup $K$.
Then $\Sart(a \circ b)$ is a union of cosets of $K$ for any sequence $b$ in $G$.
\end{lemma}

\begin{proof} 
Consider $y \in \Sart(a \circ b)$, and write $y = y_1 + y_2$, 
such that $y_1 \in \Sart(a)$ and $y_2 \in \Sart(b)$. 
Let $z \in y + K$, and write $z = y + h$ for $h \in K$.
Then $y_1 + h \in \Sart(a)$, and hence $z = (y_1 + h) + y_2 \in \Sart(a \circ b)$. 
\end{proof}

We deduce a property of sequences with the following minimality property.

\begin{defin}
Suppose $a$ is a sequence in $G$. We say that $a$ is \emph{minimal}
if $\Sart(a') \ne \Sart(a)$ for any proper subsequence $a'$ of $a$.
\end{defin}

\begin{lemma} \label{MinSequence}
Suppose $a$ is a minimal sequence in $G$ and $K$ is a subgroup of $G$.
Then $a$ has at most $|K|-1$ elements in $K$.
\end{lemma}

\begin{proof} 
Let $b$ and $c$ be the subsequences of $a$ consisting of its elements 
that are in $K$ and not in $K$ respectively.
It suffices to prove that $|\Sart(b^{j+1})|>|\Sart(b^j)|$ for $j \ge 0$,
as then $b$ has length at most $|\Sart(b)| - 1 \leq |K| - 1$.
Suppose for a contradiction that $|\Sart(b^{j+1})|=|\Sart(b^j)|$ for some $j \ge 0$.
Let $a'$ and $b'$ be obtained from $a$ and $b$ by deleting $b_{j+1}$.
We will show that $\Sart(a)=\Sart(a')$, contradicting minimality of $a$.
Consider any $y \in \Sart(a)$. We claim that $y \in \Sart(a')$.
We can assume $y$ is a subsequence sum using $b_{j+1}$.
Then $y = b_{j+1} + y_1 + y_2$ with $y_1 \in \Sart(b')$ and $y_2 \in \Sart(c)$.
Note that $\Sart(b^j)$ is a union of cosets of $\bgen{b_{j+1}}$ by Lemma~\ref{UnionOfCosets},
and the same is true of $\Sart(b')$ by Lemma~\ref{MonotoneGroup}.
Thus $b_{j+1} + y_1 \in \Sart(b')$, so $y \in \Sart(a')$, as required.
\end{proof}

Now we associate the following key subgroup with a sequence
(Lemma~\ref{UniqueMaximal} will show that it is well-defined).

\begin{defin} \label{Key}
Suppose $a$ is a sequence in  $G$. The \emph{key subgroup} $K(a)$ for $a$
is the unique subgroup $K$ such that $\Sart(a)$ is a union of cosets of $K$ 
and $K$ is maximal with this property.
\end{defin}

Note that the key subgroup always exists since every subset of $G$ is a union of cosets of $\{0\}$.
The following lemma shows that it is unique, and derives some of its properties.

\begin{lemma} \label{UniqueMaximal} 
Suppose $a$ is a sequence in $G$.
Then the key subgroup $K(a)$ is well-defined, and
\begin{enumerate}[(i)] 
\item if $|G|>1$ and $a$ has at least $|G|-1$ nonzero entries then $|K(a)|>1$,
\item $K(a+K(a))$ is the trivial subgroup of $G/K(a)$,
\item if $|G/K(a)|>1$ then $a$ has at most $|G/K(a)|-2$ elements not in $K(a)$.
\end{enumerate}
\end{lemma}

\begin{proof} 
To show that $K(a)$ is well-defined, 
it suffices to show that if $\Sart(a)$ is a union of cosets of $K$ and of $K'$, 
then $\Sart(a)$ is also a union of cosets of $K+K'$.
To see this, let $y$ be any element of $\Sart(a)$ and consider $z \in y + K+K'$.
Write $z = y + h + h'$, where $h \in K$ and $h' \in K'$.
Now $y + h \in \Sart(a)$ as $\Sart(a)$ is a union of cosets of $K$,  
so $z = (y+h) + h' \in \Sart(a)$ as $\Sart(a)$ is a union of cosets of $K'$.

To prove (i), we argue similarly to Remark~\ref{Prime}.
Either $|\Sart(a)| = |G|$ (so $K(a) = G$), or there exists $j \ge 0$ such that $a_{j+1} \neq 0$ and $\Sart(a^j) = \Sart(a^{j+1})$. In the latter case Lemma~\ref{UnionOfCosets} implies that $\Sart(a^j)$ is a union of cosets of $\bgen{a_{j+1}}$, 
and by Lemma~\ref{MonotoneGroup} the same holds for $\Sart(a)$. 

For (ii), suppose for a contradiction that $K'=K(a+K(a))$ 
is a non-trivial subgroup of $G/K(a)$.
Then $\{x + K(a) \mid x \in \Sart(a)\} = \bigcup_{t \in T} (t+K')$ for some $T \sub G/K(a)$.
Let $K^*$ be the set of $g \in G$ such that
$g$ is contained in some coset $z+K(a)$ with $z \in K'$.
Then $K^*$ is a subgroup of $G$ that strictly contains $K(a)$.
Now fix representatives $\hat{t} \in t+K(a)$ for $t \in T$
and note that $\Sart(a) = \bigcup_{t \in T} (\hat{t}+K^*)$,
which contradicts the definition of $K(a)$. 

To show (iii), suppose for a contradiction that $a$ 
has at least $|G/K(a)|-1$ elements not in $K(a)$.
Applying (i) in $G/K(a)$ we have $|K(a+K(a))|>1$,
which contradicts (ii), so we are done.
\end{proof}

\subsection{The full result.} \label{StrongPMNSSec}
Theorem~\ref{PMNeccSuff} follows from Theorem~\ref{WeakerPMNS} and the following result.

\begin{lemma} \label{ImplySoluble} 
Suppose that $k \geq 3$ and $H$ is a $k$-graph such that every $2k(k-3)$-far full pair for $H$ is soluble.
Then every full pair for $H$ is soluble.
\end{lemma}

First we require the following calculation.

\begin{lemma} \label{Calc} Suppose $t \ge 1$ and $\xb \in \Z^{t+1}$ with $x_i \ge 2$ for $i \in [t+1]$. Let
$$f(\xb) = \sum_{j \in [t]} x_1 \ldots x_j + \sum_{j \in [t+1]} x_j + \sum_{j \in [t]} x_{j+1} \ldots x_{t+1} - 4t - 2.$$
Then $f(\xb) \leq 2(x_1 \ldots x_{t+1} - 2)$.
\end{lemma}

\begin{proof}
First we consider the case $t=1$. Then $f(\xb) = 2x_1 + 2x_2 - 6$, 
so the required inequality is equivalent to $x_1 + x_2 - 1 \le x_1x_2$;
this holds as $(x_1 - 1)(x_2 - 1) \ge 0$.
Now suppose $t>1$. Note that the inequality holds if $x_i=2$ for all $i \in [t+1]$,
as then $f(\xb) = (2^{t+1}-2) + 2(t+1) + (2^{t+1}-2) - 4t-2 < 2(2^{t+1}-2)$. Also note that the inequality holds if $t=2$ and $\xb = (2,s,2)$ for some $s > 2$, as then $f(\xb) = 5s + 8 - 10 < 2(4s - 2)$.
For $r \in [t+1]$ write $\xb^{r \to 2}$ for the vector 
obtained from $x$ by setting $x_r$ equal to $2$. 
We will show that 
\[\frac{f(\xb)}{f(\xb^{r \to 2})} \le \frac{x_r}{2}
\le \frac{x_1 \ldots x_{t+1} - 2}{2x_1 \ldots x_{r-1} x_{r+1} \ldots x_{t+1} - 2},\]
except in the case where $t = r = 2$ and $\xb = (2, s, 2)$ for some $s > 2$. This suffices to prove the lemma, since we can replace each $x_r$ by $2$ in turn until we reach one of the cases considered above.
The second inequality is immediate from $x_r \ge 2$.
For the first, we rewrite it as
\begin{align*}
& 0 \le x_rf(\xb^{r \to 2}) - 2f(\xb) \\
& = (x_r - 2)\left(\sum_{j \in [r-1]} x_1 \ldots x_j + \sum_{j \in [t+1], j \neq r} x_j 
+ \sum_{r \leq j \leq t} x_{j+1} \ldots x_{t+1} - 4t - 2\right).
\end{align*}
If $t \ge 3$, or if $t=2$ and $r \ne 2$, this holds since there are $2t$ positive terms in the last bracket, 
each of which is at least $2$ and at least one of which is at least $4$.
Alternatively, if $t=r=2$ and either $x_1 \geq 3$ or $x_3 \geq 3$ we note that at least two terms in the bracket are at least $3$.
In any case we have the desired inequality.
\end{proof}

\begin{proof}[Proof of Lemma~\ref{ImplySoluble}.] 
We proceed by induction on $|\Part|$.
Let $(\Part, L)$ be a full pair and assume that every full pair $(\Part', L')$ with $|\Part'| < |\Part|$ is soluble.
We will show that $(\Part, L)$ is soluble. Write $G = G(\Part, L)$.
We will construct a sequence 
$G = K_0 \supseteq K_1 \supseteq K_2 \supseteq \ldots \supseteq K_{t_0} \supseteq 0$ of distinct subgroups of $G$, 
together with a sequence of vertex sets $\emptyset = S_0 \subseteq S_1 \subseteq S_2 \subseteq \ldots \subseteq S_{t_0}$, 
such that the following conditions hold:
\begin{enumerate}[(i)]
\item $|S_t|/k \leq \sum_{j \in [t]} |G/K_j| + \sum_{j \in [t]} |K_{j-1}/K_j| + \sum_{j \in [t]} |K_j| - 4t$.
\item Every edge $e \in H$ such that $R_G(e) \notin K_t$ intersects $S_t$.
\item There exists a solution for $(\Part_{K_t}, L_{K_t})$ using only vertices from $S_t$.
\end{enumerate}
Recalling that $\Part_{G}$ is the trivial partition of $V(H)$ into a single part, 
it is clear that $S_0 = \emptyset$ satisfies these conditions.
Given $K_t$ and $S_t$, we construct $K_{t + 1}$ and $S_{t + 1}$ as follows.

For any matching $M$, let $K(M)$ be the key subgroup of the sequence of residues of edges in $M$ (in any order).
Consider the set of all matchings $M$ in $H - S_t$ such that $\Sart_G(M \cup \{e\}) = \Sart_G(M)$ 
for every edge $e \in H - S_t$ which does not intersect $V(M)$.
This is a non-empty set, as it contains all maximal matchings in $H-S_t$.
We choose such an $M$ so that 
$\Sart_G(M)$ is maximal, and subject to this $M$ is minimal. Having chosen $M$, we set $K_{t+1} := K(M)$. 
 Note that the choice of $M$ implies that $\Sart_G(M') \neq \Sart_G(M)$ for any strict subset $M' \subset M$. Indeed, if $\Sart_G(M') = \Sart_G(M)$ then minimality of $M$ implies that some edge $e \in H - S_t$ which does not intersect $V(M')$ has $\Sart_G(M' \cup \{e\}) \neq \Sart_G(M') = \Sart_G(M)$, but then a maximal matching containing $M' \cup \{e\}$ contradicts the maximality of $\Sart_G(M)$.
Also observe that any edge $e$ with $R_G(e) \notin K_{t+1}$ intersects $S_t \cup V(M)$;
otherwise $\Sart_G(M)$ is a union of cosets of $\bgen{R_G(e)}$ by Lemma~\ref{UnionOfCosets},
which contradicts the definition of $K(M) = K_{t+1}$ as being the unique maximal subgroup of $G$ such that $\Sart_G(M)$ is a union of cosets of $K(M)$.
Thus $(\Part_{K_{t+1}}, L_{K_{t+1}})$ is a $|S_t \cup V(M)|$-far full pair for $H$
by Lemma~\ref{LemmaMerge}. 

First we consider the case that $K_{t+1} = K_t$.
By (iii) there is a solution $M_{sol}$ for $(\Part_{K_t}, L_{K_t})$ using only vertices from $S_t$.
Then $R_G(V(H) \sm V(M_{sol})) \in K_t$ by Lemma~\ref{LemmaMerge}.
Since $K(M)=K_{t+1}=K_t$ there is $M' \sub M$ with $R_G(V(M')) = R_G(V(H) \sm V(M_{sol}))$.
Now $M' \cup M_{sol}$ is a matching, as $M$ is disjoint from $S_t$,
and $R_G(V(H) \sm V(M' \cup M_{sol})) = 0_G$,
so $(\Part, L)$ is soluble by Lemma~\ref{EquivSol}.

Now we can assume that $K_{t+1}$ is a strict subgroup of $K_t$.
We will show that $|S_t \cup V(M)| \le 2k(k-3)$, 
so that we can apply the hypothesis of the lemma.
Define $\xb \in \Z^{t+1}$ by $x_j = |K_{j-1}/K_j|$ for $j \in [t]$ and $x_{t+1} = |K_t|$.
Then $x_j \ge 2$ for all $j \in [t+1]$ and $x_1 \dots x_{t+1} = |G| \le k-1$ by Lemma~\ref{GroupSize}.
Note that $M$ has at most $|K_{t+1}|-1$ elements $e$ with $R_G(e) \in K_{t+1}$ by Lemma~\ref{MinSequence}.
Also, since $R_G(e) \in K_t$ for all $e \in M$ by (ii), 
$M$ has at most $|K_t/K_{t+1}|-2$ elements $e$ with $R_G(e) \notin K_{t+1}$ by Lemma~\ref{UniqueMaximal}(iii).
Therefore $|M| \le |K_{t+1}| + |K_t/K_{t+1}| - 3$.

By (i) and Lemma~\ref{Calc} we have
\begin{align*} 
|S_t \cup V(M)|/k & \leq \sum_{j \in [t]} |G/K_j| + \sum_{j \in [t+1]} |K_{j-1}/K_j| + \sum_{j \in [t+1]} |K_j| - 4t - 3 \\
& = f(\xb) - (|K_t/K_{t+1}| - 1)(|K_{t+1}| - 1) \\
& \le 2(|G|-2) \le 2(k-3).
\end{align*}
Thus $(\Part_{K_{t+1}}, L_{K_{t+1}})$ is a $2k(k-3)$-far full pair for $H$,
so by assumption has a solution $M_{sol}$, which by definition has size at most $|G/K_{t+1}|-1$.

If $|K_{t+1}|=1$ then $M_{sol}$ is a solution for $(\Part, L)$.
Otherwise, we define $S_{t+1} = S_t \cup V(M) \cup V(M_{sol})$ and proceed to the next step.
The required conditions on $S_{t+1}$ hold as 
\[|S_{t+1}|/k \le |S_t|/k + |M| + |M_{sol}| \le |S_t|/k + |G/K_{t+1}| + |K_t/K_{t+1}| + |K_{t+1}| - 4,\]
every edge with $R_G(e) \notin K_{t+1}$ intersects $S_t \cup V(M)$,
and there is a solution for $(\Part_{K_{t+1}}, L_{K_{t+1}})$ using only vertices from $S_{t+1}$ (namely $M_{sol}$).
\end{proof}

\section{Deferred proofs} \label{DeferredProofsSection}

In this section we present the proofs of Lemmas~\ref{FindPMLemma},~\ref{RobustRandom} and~\ref{NonPartiteLemma}.
 
\subsection{Proof of Lemma~\ref{FindPMLemma}} 
Recall that we are given a $k$-graph $H$ on $n$ vertices that has a perfect matching,
and for each set $A$ of $k-1$ vertices of $H$ we let $t_A = \max(0, (1/k + \gamma)n - d_H(A))$.
Write $V=V(H)$, 
\[ \chi_1 = \sum_{A \in {V \choose k-1}} t_A^2 \ \ \text{ and } \ \
\chi_2 = \max\left( 0, \frac{n^{k-1}}{3k!} - \delta_1(H) \right).\]
We are given
\begin{enumerate}
\item[(i)] $\chi_1 < \eps \gamma^2 n^{k+1}/4 + 3kn^k$, and
\item[(ii)] $\chi_2 +  \tfrac{\chi_1}{\sqrt{\eps}\gamma^2 n^2} < \sqrt{\eps} n^{k-1}$.
\end{enumerate}
(We have slightly rephrased (ii) from the hypothesis of the lemma by including the maximum in
the definition of $\chi_2$; to see that it is equivalent, note that if $\chi_2=0$
then by (i) we have $\tfrac{\chi_1}{\sqrt{\eps}\gamma^2 n^2} < \sqrt{\eps} n^{k-1}$.)

Let $e$ be any edge of $H$.
By (i) there are at most $2\eps n^{k-1}$ $(k-1)$-sets $A \subseteq V \sm e$ such that 
$d(A) < (1/k + \gamma/2)n$, and by (ii) we have $\chi_2 < \sqrt{\eps} n^{k-1}$, so $\delta_1(H - e) \geq n^{k-1}/6k!$. 
Thus we can apply Theorem~\ref{polytime} to $H-e$ for every $e \in H$, with $\gamma/2$ and $3\eps$ in place of $\gamma$ and $\eps$, to construct the set $E'$ of edges $e$ such that $H-e$ contains a perfect matching; in total this takes time $O(n^{3k^2 - 7k})$. 

We will show that we can find an edge $e \in E'$ in time $O(n^{2k})$ such that $H-e$ satisfies the analogues of (i) and (ii). 
More precisely, for any $(k-1)$-set $A \sub V$ we define $t_A^e$ by
$$ t_A^e = \begin{cases} 
 \max(0, (1/k + \gamma)(n-k) - d_{H-e}(A)) & \mbox{if $A \sub V \sm e$}, \\
 0 & \mbox{otherwise}.
\end{cases}$$
 We also define 
\[ \chi_1^e = \sum_{A \in {V \sm e \choose k-1}} (t_A^e)^2 \ \ \text{ and } \ \
\chi_2^e = \max\left(0, \frac{(n-k)^{k-1}}{3k!} - \delta_1(H-e)\right).\]
Then we need to find $e \in E'$ so that
\begin{enumerate}
\item[(i)${}^e$] $\chi_1^e < \eps \gamma^2 (n-k)^{k+1}/4 + 3k(n-k)^k$, and
\item[(ii)${}^e$] $\chi_2^e +  \tfrac{\chi_1^e}{\sqrt{\eps}\gamma^2 (n-k)^2} < \sqrt{\eps} (n-k)^{k-1}$.
\end{enumerate}

First we claim that for each $A \in {V \choose k-1}$, 
\begin{equation} \label{FindPMEqn0}
(t^e_A)^2 \leq t_A^2 + 2t_A(|H(A) \cap e| - 1 - k\gamma) + k^2.
\end{equation}
Note that if $t_A = 0$ then~(\ref{FindPMEqn0}) follows from $t^e_A \leq k$.
On the other hand, if $t_A > 0$, then 
\begin{align*}
(t^e_A)^2 &\leq (t_A + |H(A) \cap e| - 1 - k\gamma)^2 \\
&= t_A^2 + 2t_A(|H(A) \cap e| - 1 - k\gamma) + (|H(A) \cap e| - 1 - k\gamma)^2 \\
&\leq t_A^2 + 2t_A(|H(A) \cap e| - 1 - k\gamma) + k^2,
\end{align*}
so (\ref{FindPMEqn0}) holds. Now to find the desired edge $e$, we split into two cases. 
 
\medskip \noindent \textbf{Case 1:} \emph{$\chi_2 > 0$.} 
In this case, by definition there must exist a vertex $x$ of degree at most $n^{k-1}/3k!$, 
and we can find such a vertex in time $O(n^{k})$. 
Since $H$ contains a perfect matching, there is some $e \in E'$ containing $x$.
We now delete $e$ and show that conditions (i)$^e$ and (ii)$^e$ hold.
First we define $\mc{A} = \{A \in {V \choose k-1}: x \in A\}$ and
note that $\sum_{A \in \mc{A}} d_H(A) = (k-1)d_H(x) \leq n^{k-1}/3(k-1)!$.
Since $|\mc{A}| = {n-1 \choose k-2}$ we have
\[\sum_{A \in \mc{A}} t_A \geq \left(\frac{1}{k} + \gamma\right)n{n-1 \choose k-2} - \frac{n^{k-1}}{3(k-1)!} \geq \frac{n}{2k}{n-1 \choose k-2}.\]
Now by the Cauchy-Schwartz inequality,
we obtain $\sum_{A \in \mc{A}} t_A^2 \geq \tfrac{n^2}{4k^2}{n-1 \choose k-2}$.
Moreover, by~(\ref{FindPMEqn0}) and then Cauchy-Schwartz we have
\begin{align*}
\sum_{A \in {{V \sm e} \choose k-1}} & \brac{(t^e_A)^2 - t_A^2} \leq \sum_{A \in {{V \sm e} \choose k-1}} \brac{2t_A(k-1) + k^2} \\
&\leq 2(k-1)\sqrt{{n \choose k-1} \chi_1} + k^2{n-k \choose k-1} \leq \sqrt{\eps} n^k,
\end{align*}
where the final inequality holds by (i). Since $\mc{A} \subseteq \binom{V}{k-1} \sm \binom{V \sm e}{k-1}$, we have
\begin{equation} \label{chi1e}
\chi_1^e \leq \chi_1 - \frac{n^2}{4k^2}{n-1 \choose k-2} + \sqrt{\eps} n^k \leq \chi_1 - \frac{n^k}{7k!},
\end{equation}
so $\chi_1^e < \eps \gamma^2 n^{k+1}/4 + 3kn^k - n^k/7k! \leq \eps \gamma^2 (n-k)^{k+1}/4 + 3k(n-k)^k$, which proves (i)${}^e$.

Note also that $\chi_2^e \leq \chi_2 + k{n \choose k-2}$, so by (\ref{chi1e}) we have
\begin{align*}
& \brac{\frac{\chi_1^e}{\sqrt{\eps} \gamma^2 (n-k)^2} + \chi_2^e}
-  \brac{\frac{\chi_1}{\sqrt{\eps} \gamma^2 n^2} + \chi_2} \\
& \le \frac{\chi_1}{\sqrt{\eps} \gamma^2 n^2} \cdot \frac{2kn-k^2}{(n-k)^2} 
 - \frac{n^k}{7k!\sqrt{\eps} \gamma^2 (n-k)^2} + k{n \choose k-2} \\
& \le k \sqrt{\eps} n^{k-2} - \frac{n^{k-2}}{\gamma} + k{n \choose k-2} < -n^{k-2}.
\end{align*}
Now by (ii) we deduce that
\[ \frac{\chi_1^e}{\sqrt{\eps} \gamma^2 (n-k)^2} + \chi_2^e  
< \sqrt{\eps} n^{k-1} - n^{k-2} < \sqrt{\eps} (n-k)^{k-1}, \]
which proves (ii)${}^e$. 

\medskip \noindent \textbf{Case 2:} \emph{$\chi_2 = 0$.}
We claim that we can find in time $O(n^{2k})$ an edge $e \in E'$ such that
\begin{equation} \label{FindPMEqn1}
\sum_{A \in {V \choose k-1}} \Big( 2t_A(|H(A) \cap e| - 1 - k\gamma) - |A \cap e| t_A^2 \Big) \leq -\frac{k(k+1)\chi_1}{n}.
\end{equation}
To see this, consider a perfect matching $M$ in $H$
(this exists by hypothesis of the lemma, although the algorithm does not have access to it).
Note that
\begin{align*}
&\sum_{e \in M} \sum_{A \in {V \choose k-1}} 2t_A(|H(A) \cap e| - 1 - k\gamma) \\
& = \sum_{A \in {V \choose k-1}} 2t_A \sum_{e \in M} (|H(A) \cap e| - 1 - k\gamma) \\
& = \sum_{A \in {V \choose k-1}} 2t_A\left(\left(\frac{1}{k} + \gamma\right)n - t_A - \frac{(1 + k \gamma)n}{k}\right) \\
& = \sum_{A \in {V \choose k-1}} -2t_A^2 = -2\chi_1.
\end{align*}
(For the second equality, note that terms with $t_A = 0$ contribute zero to each side of the equation, whilst terms with $t_A > 0$ are equal by definition of $t_A$ and the fact that $M$ has size $n/k$).
Note also that 
\[\sum_{e \in M} \sum_{A \in {V \choose k-1}} |A \cap e| t_A^2 = \sum_{A \in {V \choose k-1}} (k-1) t_A^2 = (k-1)\chi_1.\]
We deduce that
\[ \sum_{e \in M} \sum_{A \in {V \choose k-1}} 
\Big( 2t_A(|H(A) \cap e| - 1 - k\gamma) - |A \cap e| t_A^2 \Big) = -(k+1)\chi_1,\]
so by averaging there is some $e \in M$ that satisfies~(\ref{FindPMEqn1}).
Note further that $e \in E'$, so we can find such an edge $e$ in time $O(n^{2k})$ 
simply by checking~(\ref{FindPMEqn1}) for every edge of $E'$.  

We will show that (i)${}^e$ and (ii)${}^e$ hold for this choice of $e$.
Let $D(e)$ be the family of $(k-1)$-sets $A \subseteq V$ which intersect $e$.
Applying~(\ref{FindPMEqn0}), we obtain
\begin{align} \label{FindPMEqn2}
\chi_1^e - \chi_1 &= \sum_{A \in {V \choose k-1}} ((t_A^e)^2 - t_A^2) 
= \sum_{A \in {V\sm e \choose k-1}} ((t_A^e)^2 - t_A^2) - \sum_{A \in D(e)} t_A^2 \nonumber \\
&\leq \sum_{A \in {V\sm e \choose k-1}} 2t_A(|H(A) \cap e| - 1 - k\gamma) - \sum_{A \in D(e)} t_A^2 + k^2 {n \choose k-1}.
\end{align}
We will show that the first and second terms of (\ref{FindPMEqn2}) are close to those of (\ref{FindPMEqn1}).
For the first, note that $|D(e)| = \binom{n}{k-1}-\binom{n-k}{k-1} \le \frac{k^2}{n}\binom{n}{k-1}$, so
\begin{align} \label{FindPMEqn3} 
& \sum_{A \in {V\sm e \choose k-1}} 2t_A(|H(A) \cap e| - 1 - k\gamma)
- \sum_{A \in {V \choose k-1}} 2t_A(|H(A) \cap e| - 1 - k\gamma) \nonumber \\
& \le \sum_{A \in D(e)} n \le k^2 \binom{n}{k-1}.
\end{align}
For the second, observe that if $A \in \binom{V}{k-1}$ is chosen uniformly at random then $|A \cap e|$ is hypergeometric with mean $k(k-1)/n$, so
\[\sum_{A \in D(e)} |A \cap e| = \sum_{A \in \binom{V}{k-1}} |A \cap e| = \frac{k(k-1)}{n} {n \choose k-1}. \]
We also have 
\[|D(e)| = {n \choose k-1} - {n-k \choose k-1} \geq \brac{ \frac{k(k-1)}{n} - \frac{k^4}{2n^2} } {n \choose k-1},\] 
so $\sum_{A \in D(e)} (|A \cap e| - 1) \leq \frac{k^4}{2n^2} {n \choose k-1}$
and so $\sum_{A \in D(e)} (|A \cap e| - 1) t_A^2 \leq k^2 {n \choose k-1}$.
It follows that 
\begin{equation} \label{FindPMEqn4}
\sum_{A \in D(e)} t_A^2 \geq \sum_{A \in {V \choose k-1}} |A \cap e| t_A^2 - k^2{n \choose k-1}.
\end{equation}
Substituting~(\ref{FindPMEqn3}) and~(\ref{FindPMEqn4}) into~(\ref{FindPMEqn2}) and then applying~(\ref{FindPMEqn1}), we obtain
\begin{align*}
\chi_1^e &\leq \chi_1 - \frac{k(k+1)\chi_1}{n} + k^2{n \choose k-1} + k^2{n \choose k-1}  + k^2{n \choose k-1} \\
&\leq \brac{1-\frac{k(k+1)}{n}}\left(\frac{\eps \gamma^2 n^{k+1}}{4} + 3kn^k\right) + 3k^2{n \choose k-1} \\
&\leq \eps \gamma^2 (n-k)^{k+1}/4 + 3k(n-k)^k,
\end{align*}
using (i) for the second inequality, and $n^{k+1} - (n-k)^{k+1} \le k(k+1)n^k$ and $n^k - (n-k)^k \le k^2 n^{k-1}$ for the third. This proves (i)${}^e$.
Further, since $\chi_2 = 0$ we have $\chi_2^e \leq k{n - 1 \choose k-2}$, so
\[\frac{\chi_1^e}{\sqrt{\eps} \gamma^2 (n - k)^2} + \chi_2^e 
< \frac{\sqrt{\eps} (n - k)^{k-1}}{2} + k{n - 1 \choose k-2} < \sqrt{\eps} (n - k)^{k-1},\]
which proves (ii)${}^e$.
\qed

\subsection{Weak hypergraph regularity}

The proofs of Lemma~\ref{RobustRandom} and~\ref{NonPartiteLemma} use the weak hypergraph regularity lemma.
We state this after the following definitions.

\begin{defin}
Suppose that $\Part$ partitions a set $V$ into $r$ parts $V_1, \dots, V_r$ and $G$ is a $\Part$-partite $k$-graph on $V$. 
For $A \in \binom{[r]}{k}$ we write $G_A$ for the induced $k$-partite subgraph of $G$ with parts $V_i$ for $i \in A$.
\begin{enumerate}[(i)]
\item The \emph{density} of $G_A$ is $d(G_A) = \frac{|G_A|}{\prod_{i \in A} |V_i|}$. 
\item For $\eps>0$ and $A \in \binom{[r]}{k}$, we say that the $k$-partite subgraph $G_A$ is
\emph{$\eps$-vertex-regular} if for any sets $V'_i \subseteq V_i$ with $|V'_i| \ge \eps |V_i|$ for $i \in A$,
writing $V' = \bigcup_{i \in A} V'_i$, we have $d(G_A[V']) = d(G_A) \pm \eps$.  
\item We say that a partition $\Part$ of $V(G)$ is \emph{$\eps$-regular} if all but at most $\eps |V|^k$ edges of $G$
  belong to $\eps$-vertex-regular $k$-partite subgraphs.
\item The \emph{reduced $k$-graph} $R^d_\Part$ is the $k$-graph whose vertices are the parts of $\Part$ and whose edges are all $k$-sets of
parts of $\Part$ that induce an $\eps$-vertex-regular $k$-partite subgraph of $G$ of density at least $d$.
\end{enumerate}
\end{defin}

Note that we use the same notation $d(\cdot)$ for density as for degree, but there should be no confusion. We shall use the following formulation of the weak hypergraph regularity lemma to obtain regular partitions. 

\begin{thm} (Weak hypergraph regularity lemma) \label{weakrl}
Suppose that $1/n \ll 1/m_0 \ll \eps, 1/r \ll 1/k$. Suppose that $G$ is a $k$-graph on $n$ vertices and that $\Part$
is a partition of $V = V(G)$ into at most $r$ parts of size at least $\eps n$. Then there is an $\eps$-regular partition $\Qart$ of $V$ into $m \leq m_0$ parts such that
\begin{enumerate}[(i)]
\item each part of $\Qart$ has size $\lfloor n/m \rfloor$ or $\lceil n/m \rceil$, and
\item $\Qart$ is $\eps$-close to a refinement $\Part'$ of $\Part$.
\end{enumerate}
\end{thm} 

\begin{proof}
Taking $1/m_0 \ll \eps' \ll \eps$, the most commonly-used form of the weak hypergraph regularity lemma (see~\cite{C}) states that there exist an integer $m \leq m_0$, an exceptional set $V_0 \sub V$ of size $|V_0| \leq \eps' n$ and a partition $\Part'$ of $V \sm V_0$ into parts $V_1, \dots, V_m$ of equal size which is $\eps'$-regular (with respect to $H[V \sm V_0]$) and has the property that $V_j$ is a subset of some part of $\Part$ for any $j \in [m]$. By distributing the vertices of $V_0$ as equally as possible among the parts of $\Part'$ we obtain the desired partition $\Qart$. 
\end{proof}

\subsection{Proof of Lemma~\ref{RobustRandom}} 

Before giving the proof of Lemma~\ref{RobustRandom}, we give a brief sketch of the idea.
A rough statement of the lemma is that robust maximality is preserved by random selection,
i.e.\ if some partition $\Part$ of $V(H)$ is robustly maximal with respect to $H$
and $S$ is a suitable random subset of $V(H)$ then with high probability $\Part[S]$ 
is also robustly maximal (with slightly weaker parameters) with respect to $H[S]$.
It is not hard to see that $H[S]$ has a transferral-free robust edge-lattice with respect to $\Part[S]$.
The main difficulty is to show that with high probability there is no strict refinement $\Part^\circ$ of $\Part[S]$
(with large parts) that has a transferral-free robust edge-lattice.
Since there are many possible refinements $\Part^\circ$ to consider, 
a straightforward union bound on the probability will not suffice.
Instead, we use the weak hypergraph regularity lemma.
We show that any such refinement $\Part^\circ$ gives rise to a partition
of the reduced $k$-graph, which in turn gives rise a refinement $\Part^*$ of $\Part$ with a transferral-free robust edge-lattice. However, this contradicts the robust maximality of $\Part$,
so no such refinement $\Part^\circ$ can exist.
 
\begin{proof}[Proof of Lemma~\ref{RobustRandom}.] 
Recall that we are given a $k$-graph $H$ on $n$ vertices and
a partition $\Part$ of $V(H)$ that is $(c, c', \mu, \mu')$-robustly maximal with respect to $H$.
We are also given a partition $\Part'$ of $V(H)$ that refines $\Part$ and integers $(n_Z)_{Z \in \Part'}$ such that $\eta |Z| \le n_Z \le |Z|$. 
We choose $S \sub V(H)$ uniformly at random subject to the condition that $|S \cap Z| = n_Z$ for each $Z \in \Part'$.
First we note that $|S| \geq \eta n$, and for each part $X \in \Part$ that $|S \cap X| \ge \eta |X| \ge \eta cn$.
Next we show that $L_{\Part[S]}^{\mu/c}(H[S])$ is transferral-free. 
It suffices to prove that $I_{\Part[S]}^{\mu/c}(H[S]) \subseteq I_\Part^\mu(H)$. 
To see this, note that if $\ib \in I_{\Part[S]}^{\mu/c}(H[S])$ then there are at least $(\mu/c)(\eta n)^k \geq \mu n^k$ edges $e \in H[S] \sub H$ with $\ib_{\Part[S]}(e) = \ib$, so $\ib \in I_\Part^\mu(H)$. 

It remains to show that with high probability there is no refinement $\Part^\circ$ of $\Part[S]$ with parts of size at least $(3c'/\eta)|S|$, such that $L_{\Part^\circ}^{(\mu')^3}(H[S])$ is transferral-free. Suppose that we have such a partition $\Part^\circ$; then we will obtain a contradiction using events that hold with high probability.
Introduce new constants $\eps, N_0, N_1$ such that $1/n \ll 1/N_1 \ll 1/N_0 \ll \eps \ll \mu$. 
We apply Theorem~\ref{weakrl} to $H$ to obtain an $\eps$-regular partition $\Qart$ of $V(H)$ which has $n_R$ parts of
almost equal size, for some $N_0 \le n_R \le N_1$, and which is $\eps$-close to a refinement of $\Part$. Let $n_0 = n/n_R$, so every cluster has size $\lfloor n_0 \rfloor$ or $\lceil n_0 \rceil$. We henceforth omit the floor and ceiling signs as these do not affect the argument.
Let $R = R^{\mu'/3}_\Qart$ be the $\mu'/3$-reduced $k$-graph on $\Qart$.

For any $Y \in \Qart$, note that $|Y \cap S|$ is a sum of independent hypergeometric random variables, with
\[ \Exp[|Y \cap S|] = \sum_{Z \in \Part'} \frac{n_Z |Y \cap Z|}{|Z|} \ge \sum_{Z \in \Part'} \eta |Y \cap Z| = \eta n_0.\]
Hence by Corollary~\ref{SumOfHyper} with high probability $|Y \cap S| \geq \eta n_0/2$ for every $Y \in \Qart$. 
Next we choose a partition of $V(R)$ that is `representative' of $\Part^\circ$.

\begin{claim} \label{RobustRandomClaim} There exists a partition $\Sart$ of $V(R)$, whose parts $(X_\Sart)_{X \in \Part^\circ}$ correspond to those of $\Part^\circ$, such that
\begin{enumerate}[(i)]
\item Each part $X_\Sart \in \Sart$ has size at least $3c'n_R/2$, and
\item $|X \cap Y| \geq (c')^2 n_0$ whenever $Y \in X_\Sart$. 
\end{enumerate}
\end{claim}

To prove the claim, we choose $\Sart$ randomly as follows. Let $P$ be the set of pairs $(X,Y)$ such that $X \in \Part^\circ$, $Y \in V(R) = \Qart$ and $|X \cap Y| \geq c'|S \cap Y|$. We independently assign each $Y \in V(R)$ to a part $X_\Sart$ such that $(X, Y) \in P$ with probability proportional to $|X \cap Y|$, so with probability 
\[\frac{|X \cap Y|}{\sum_{X: (X,Y) \in P} |X \cap Y|} \ge \frac{|X \cap Y|}{|S \cap Y|}.\]
Note that whenever $Y \in X_\Sart$ we have $|X \cap Y| \geq c'(\eta n_0/2) \geq (c')^2 n_0$, so (ii) is satisfied. 
For (i), consider any $X \in \Part^\circ$, and note that $\sum_{Y \in V(R)} |X \cap Y| = |X| \ge (3c'/\eta)|S| \ge 3c'n$. Then
$$\sum_{Y \in V(R)} \frac{|X \cap Y|}{|S \cap Y|} 
\geq \sum_{Y \in V(R)} \frac{|X \cap Y|}{n_0} \geq \frac{3c'n}{n_0} = 3c'n_R, \ \text { and }$$
\begin{align*}
\Exp[|X_\Sart|] & \geq \sum_{Y: (X,Y) \in P} \frac{|X \cap Y|}{|S \cap Y|} \\
&= \sum_{Y \in V(R)} \frac{|X \cap Y|}{|S \cap Y|} - \sum_{Y: (X,Y) \notin P} \frac{|X \cap Y|}{|S \cap Y|} \\
&\geq 3c'n_R - c'n_R = 2c'n_R.
\end{align*}
Thus (i) holds with high probability by Lemma~\ref{VaryingChernoff}, which proves the claim. (Note that here we used the phrase `with high probability' to mean with probability $1 - e^{-\Omega(n_R^c)}$ for some $c > 0$ as $n_R \to \infty$, that is, with $n_R$ large instead of our usual definition with $n$ large.)

\medskip

We now define a partition $\Qart^*$ of $V$ by 
\[\Qart^* = \{\bigcup_{Y \in X} Y \mid X \in \Sart\}.\]
So $\Qart$ refines $\Qart^*$, and by Claim~\ref{RobustRandomClaim}(i) each part of $\Qart^*$ has size at least $3c'n_Rn_0/2 = 3c'n/2$.   
Recall that we chose $\Qart$ to be $\eps$-close to a refinement of $\Part$, so we can obtain a refinement of $\Part$ from $\Qart$ by changing the part of a set $B$ of at most $\eps n$ vertices of $V$. Furthermore, any cluster $Y$ was assigned to $X_\mc{S}$ for some $X \in \Part^\circ$ such that $|X \cap Y| \geq (c')^2n_0$ by Claim~\ref{RobustRandomClaim}(ii). Since $\Part^\circ$ is a refinement of $\Part$, we deduce that we may obtain a refinement of $\Part$ from $\Qart^*$ by possibly changing the part of
\begin{enumerate}[(i)]
\item the at most $\eps n$ vertices in $B$, and
\item the at most $n_0 \cdot (\eps n / (c')^2 n_0) < \mu' n/5$ vertices which lie in clusters $Y$ with $|Y \cap B| \geq (c')^2n_0$.
\end{enumerate}
We conclude that $\Qart^*$ is $(\mu'/4)$-close to a refinement $\Part^*$ of $\Part$; 
in particular, $\Part^*$ then has parts of size at least $c'n$.

To finish the proof of the lemma, it suffices to prove the following claim.
Indeed, since $L_{\Part^\circ}^{(\mu')^3}(H[S])$ is transferral-free,
it will imply that the same is true of $L_{\Part^*}^{\mu'}(H)$,
contradicting the robust maximality of $\Part$.

\begin{claim} \label{RobustRandomClaim2} 
$I_{\Part^*}^{\mu'}(H) \sub I_{\Qart^*}^{\mu'/2}(H) \sub I_{\Sart}^{\mu'/10}(R) \sub I_{\Part^\circ}^{(\mu')^3}(H[S])$.
\end{claim}

The first inequality holds by Proposition~\ref{similarlattices} since $\Qart^*$ is $(\mu'/4)$-close to $\Part^*$. For the second inequality, consider $\ib \in I_{\Qart^*}^{\mu'/2}(H)$. By definition of $R$, there are at most $\eps n^k + n_R^k \cdot (\mu'/3) n_0^k \leq 2\mu' n^k/5$ edges $e \in H$ with $\ib_{\Part^*}(e) = \ib$ which do not lie in $k$-graphs corresponding to edges of $R$. There are at least $\mu' n^k/2$ edges $e \in H$ with $\ib_{\Qart^*}(e) = \ib$, so at least $\mu' n^k/10$ of these lie in $k$-graphs corresponding to edges of $R$. Thus
there are at least $\mu' n^k/10n_0^k = \mu'n_R^k/10$ edges $e \in R$ such that $\ib_{\Sart}(e) = \ib$, i.e.\ $\ib \in I_{\Sart}^{\mu'/10}(R)$.

For the final inequality consider $\ib \in I_{\Sart}^{\mu'/10}(R)$. Then there are at least $\mu' n_R^k/10$ edges $e \in R$ with $\ib_{\Sart}(e) = \ib$. 
Consider such an edge $e = \{Y_1,\dots,Y_k\}$, let $(X_j)_\Sart$ be the part of $\Sart$ containing $Y_j$ for each $j \in [k]$, and let $F$ be the $k$-partite subgraph of $H$ with vertex classes $(X_j \cap Y_j)_{j \in [k]}$. By definition of $R$, since each part of $F$ has size $(c')^2 n_0$ by Claim~\ref{RobustRandomClaim}(ii), we have $|F| \ge (\mu'/3 - \eps)((c')^2 n_0)^k \geq 10(\mu')^2 (|S|/n_R)^k$. 
Summing over all choices of $e$, there are at least $(\mu')^3 |S|^k$ edges $e' \in H[S]$ with $\ib_{\Part^\circ}(e') = \ib$, 
i.e.\ $\ib \in I_{\Part^\circ}^{(\mu')^3}(H[S])$.
This completes the proof of the claim, and so of the lemma.
\end{proof}

\subsection{Proof of Lemma~\ref{NonPartiteLemma}}
Recall that we are given a $k$-graph $H$ on a set $V =V(H)$ of size $kn$ and a partition $\Part$ of $V(H)$ such that
(i) at most $\eps n^{k-1}$ $(k-1)$-sets $S \subseteq V$ have $d_H(S) < (1+\gamma)n$,
(ii) $\Part$ is $(c, c, \mu, \mu')$-robustly maximal with respect to $H$,
(iii) any vertex is in at least $dn^{k-1}$ edges $e$ with $\ib_\Part(e) \in L_\Part^\mu(H)$, and
(iv) $\ib_\Part(V) \in L_\Part^\mu(H)$.

We choose uniformly at random a partition $\Part' = (V_1, \ldots, V_k)$ of $V$ into $k$ parts each of size $n$. Let $H'$ be the induced $\Part'$-partite subgraph of $H$ and let $\hPart$ be the common refinement of $\Part$ and $\Part'$.
We also set $c_* = c/2k$, $d_* = d/2k^k$ and $\gamma_* = \gamma/2k$, and introduce a new constant $\mu_*$ with $\mu \ll \mu_* \ll \mu'$. Note that $\mu \ll \mu_* \ll c_*, d_* \ll \gamma_*, 1/k$.
We will show that the following conditions hold with high probability. 
\begin{enumerate}[({F}1)]
\item At most $\eps n^{k-1}$ $\Part'$-partite $(k-1)$-sets $S \sub V$ have $|H'(S)| < (1/k + \gamma_*)n$,
\item $\hPart$ is $(c_*, c_*, \mu, \mu_*)$-robustly maximal with respect to $H'$,
\item any vertex is in at least $d_* n^{k-1}$ edges $e \in H'$ with $\ib_{\hPart}(e) \in L_{\hPart}^{\mu}(H')$, and
\item $\ib_{\hPart}(V) \in L_{\hPart}^{\mu}(H')$.
\end{enumerate}
We then apply Lemma~\ref{kPartiteLemma} to $H'$ with $\hPart$ in place of $\Part$ and $c_*, \mu_*, d_*, \gamma_*$ in place of $c, \mu', d, \gamma$ to obtain a perfect matching in $H'$, which is also a perfect matching in $H$.

First we note that (F1) is immediate from (i) and Corollary~\ref{SumOfHyper}, and (F3) follows from (iii) and Corollary~\ref{ApplyAzuma}, similarly to the proof of (E3) in Claim~\ref{choiceclaim}. 
Next, since each part of $\Part$ has size at least $c|V| = ckn$, by Corollary~\ref{SumOfHyper} with high probability each part of $\hPart$ has size at least $cn/2 = c_*|V|$ and by Proposition~\ref{InverseRestriction}(i), for every $\ib \in L_{\hPart}^{\mu}(H') \subseteq L_{\hPart}^{\mu}(H)$ we have $(\ib \mid \Part) \in L_\Part^\mu(H)$, so $L_{\hPart}^{\mu}(H')$ is transferral-free. 
This gives part of (F2); it also allows us to apply Proposition~\ref{properties} and deduce that $L_{\hPart}^{\mu}(H')$ is full with respect to $\Part'$. 

Now we claim that $\ib \in L_{\hPart}^{\mu}(H')$ for every index vector $\ib$ with respect to $\hPart$ such that $(\ib \mid \Part')$ is a multiple of $\1$ and $(\ib \mid \Part) \in L_\Part^{\mu}(H)$. Note that this will imply (F4), as $(\ib_{\hPart}(V) \mid \Part') = n \cdot \1$ and $(\ib_{\hPart}(V) \mid \Part) = \ib_\Part(V) \in  L_\Part^{\mu}(H)$ by (iv). To prove the claim, we apply Proposition~\ref{properties0}(\ref{properties:index}) to find parts $X$ and $X'$ of $\hPart$ such that $\ib - \ub_X + \ub_{X'} \in L_{\hPart}^\mu(H')$. Now $(\ub_X - \ub_{X'} \mid \Part) = (\ib \mid \Part) - (\ib - \ub_X + \ub_{X'} \mid \Part) \in L_\Part^{\mu}(H)$, so $X$ and $X'$ must be contained in the same part of $\Part$, as $L_\Part^{\mu}(H)$ is transferral-free. Also, since $H'$ is $\Part'$-partite, $X'$ is contained in the same part of $\Part'$ as $X$, so $X'=X$ and $\ib \in L_{\hPart}^{\mu}(H')$, as claimed.

To finish the proof of (F2), we need to show that with high probability there is no  strict refinement $\hPart^*$ of $\hPart$ with parts of size at least $c_*|V|$ such that $L_{\hPart^*}^{\mu_*}(H')$ is transferral-free; this will occupy the remainder of the proof. Suppose that we have such a partition $\hPart^*$; then we will obtain a contradiction using events that hold with high probability. By (F1) we can apply Proposition~\ref{properties} to deduce that $L_{\hPart^*}^{\mu_*}(H')$ is full and every part of $\hPart^*$ has size at least $(1/k+\gamma_*/2)n$. Also, by Proposition~\ref{properties0}(i) there is some integer $\ell$ such that every part of $\Part'$ is refined into $\ell$ parts by $\hPart^*$; note that $\ell<k$. 

Now we use an argument which is similar in spirit to that of Lemma~\ref{RobustRandom} (but more complicated). Let $N_0, N_1$ satisfy $1/n \ll 1/N_1 \ll 1/N_0 \ll \eps$. We apply Theorem~\ref{weakrl} to obtain an $\eps$-regular partition $\Qart$ of $V$ which is $\eps$-close to a refinement of $\Part$ and has $n_R$ parts of almost equal size, for some $N_0 \le n_R \le N_1$. Let $n_0 = kn/n_R$, so every cluster has size $\lceil n_0 \rceil$ or $\lfloor n_0 \rfloor$; we omit the floor and ceiling signs since they do not affect the argument. Also let $R = R^{\mu'/4}_\Qart$ be the $\mu'/4$-reduced $k$-graph of $H$ on $\Qart$. For any $e \in R$ we write 
\[H_e = \{e' \in H: e' \cap X \ne \es \ \forall X \in e \} \]
for the $k$-partite subgraph of $H$ corresponding to $e$. 
We note for future reference the following claim, which shows that $R$ inherits a codegree condition from $H$.

\begin{claim} \label{NonParClaim0}
At most $\gamma^{-1}\eps n_R^{k-1}$ $(k-1)$-sets $A \sub V(R)$ have $|R(A)|<n_R/k$. 
\end{claim}

To see this, recall that at most $\eps |V|^k$ edges of $H$ do not lie in $\eps$-vertex-regular $k$-partite subgraphs of $H$ formed by parts of $\Qart$; call these \emph{irregular} edges. So there are at most $k^2\sqrt{\eps} n_R^{k-1}$ $(k-1)$-sets $A \subseteq V(R)$ for which more than $\sqrt{\eps} n_0^{k-1} n$ irregular edges intersect each member of $A$. Fix any other $A$ for which $|R(A)| < n_R/k$. Then we have $\sum_{B \in V_A} |H(B)| < n_0^k n_R/k + (\mu'/4) n_0^{k-1} n + \sqrt{\eps} n_0^{k-1} n$, 
where $V_A$ denotes the set of $(k-1)$-sets with one vertex from each cluster of $A$. 
Let $m_A$ be the number of $(k-1)$-sets $B \in V_A$ such that $|H(B)| < (1 + \gamma)n$. 
Then $\sum_{B \in V_A} |H(B)| \geq (n_0^{k-1} - m_A)(1 + \gamma)n$, and so 
\begin{align*}
m_A & \geq n_0^{k-1} - \frac{n_0^k n_R/k +  (\mu'/4)  n_0^{k-1} n + \sqrt{\eps} n_0^{k-1}n}{(1 + \gamma)n} \\
& = n_0^{k-1} \frac{\gamma - \mu'/4 - \sqrt{\eps}}{1 + \gamma} \geq \frac{1}{2}\gamma n_0^{k-1}. 
\end{align*}
Since by (i) there are at most $\eps n^{k-1}$ $(k-1)$-sets $B \subseteq V$ with $|H(B)| < (1+\gamma)n$,
there can be at most $\eps n_R^{k-1}/2\gamma$ such sets $A$. Together with the at most  $k^2 \sqrt{\eps} n_R^{k-1}$ sets $A$ considered earlier we have a total of at most $\gamma^{-1} \eps n_R^{k-1}$ $(k-1)$-sets $A \subseteq V(R)$ with $|R(A)| < n_R/k$, proving the claim.

\medskip

In the next claim we choose for each cluster of $R$ a representative part of $\hPart^*$ within each part of $\Part'$. We encode these choices by $k$-vectors with respect to $\hPart^*$ which are $\Part'$-partite, by which we mean that they correspond to $\Part'$-partite $k$-sets. 

\begin{claim} \label{NonParClaim1}
There are $\Part'$-partite $k$-vectors $\ib^*(Y) \in \Z^{\hPart^*}$ for each $Y \in V(R)$ such that
\begin{enumerate}[({G}1)]
\item $|Y \cap Z| \geq c_*n_0$ whenever $i^*(Y)_Z = 1$, and
\item $R_Z := |\{Y: i^*(Y)_Z=1\}| \ge n_R/k$ for every $Z \in \hPart^*$.
\end{enumerate}
\end{claim}

The proof of the claim is similar to that of Claim~\ref{RobustRandomClaim}. For each $i \in [k]$ we let $P_i$ be the set of pairs $(Y,Z)$ such that $Y \in V(R) = \Qart$, $Z \in \hPart^*$, $Z \sub V_i$ and $|Y \cap Z| \geq c_*n_0$. For each $Y \in V(R)$ and $i \in [k]$ we let $i^*(Y)_Z = 1$ for a part $Z$ with $(Y, Z) \in P_i$ chosen with probability proportional to $|Y \cap Z|$, and let $i^*(Y)_Z = 0$ otherwise, where all random choices are made independently. Thus (G1) is satisfied by choice of $P_i$. For (G2), first note that 
$\Prob[i^*(Y)_Z = 1] \geq \frac{|Y \cap Z|}{|Y \cap V_i|}$ for every $Y$ and $Z \subseteq V_i$ such that $|Y \cap Z| \geq c_* n_0$. Also, since $|Y \cap V_i|$ is distributed hypergeometrically, with high probability $|Y \cap V_i| = (1 \pm \eps)n_0/k$ for every $Y \in \Qart$ and $V_i \in \Part'$ by Corollary~\ref{SumOfHyper}. Now for any $V_i \in \Part'$ and any $Z \in \hPart^*$ with $Z \subseteq V_i$ we have 
$$\sum_{Y \in V(R)} \frac{|Y \cap Z|}{|Y \cap V_i|} \geq \sum_{Y \in V(R)} \frac{|Y \cap Z|}{(1 + \eps)n_0/k}
= \frac{|Z|}{(1 + \eps)n_0/k}, \ \ \text{ and }$$ 
$$\sum_{Y: (Y,Z) \notin P_i} \frac{|Y \cap Z|}{|Y \cap V_i|} \leq n_R \cdot \frac{c_*n_0}{(1-\eps)n_0/k} < 2kc_*n_R.$$
Since every part $Z$ of $\hPart^*$ has size at least $(1/k+\gamma_*/2)n$, we obtain
$$\Exp[R_Z] \geq \sum_{Y: (Y,Z) \in P_i} \frac{|Y \cap Z|}{|Y \cap V_i|} \geq (1 + \gamma_*/3)n_R/k.$$
Thus (G2) holds with high probability by Lemma~\ref{VaryingChernoff}, which completes the proof of the claim. 
 
\medskip

In the next claim we show that certain `bad' occurrences are rare. We will consider a cluster to be bad if it is not well-represented by its choice of vector in Claim~\ref{NonParClaim1}. We will consider a vertex to be bad if it belongs to a good cluster but not to the part assigned to this cluster in Claim~\ref{NonParClaim1}. We will consider an edge of the reduced $k$-graph to be bad if it either contains a bad cluster or contains two clusters whose vectors are `incompatible'. More precisely, we make the following definitions. 

\begin{defin} $ $
\begin{enumerate}[(i)]
\item We call a cluster $Y \in V(R)$ \emph{bad} if there exists $Z \in \hPart^*$ such that $\ib^*(Y)_Z = 0$, but $|Y \cap Z| \geq \mu_*^{1/4} n_0$; otherwise we call $Y$ \emph{good}. 
\item We call a vertex \emph{bad} if it is contained in $Y \cap Z$ for some good cluster $Y$ 
and some $Z \in \hPart^*$ such that $\ib^*(Y)_Z = 0$.
\item We call an edge $e$ of $R$ \emph{bad} if either $e$ contains a bad cluster, or $e$ contains two clusters $Y_1, Y_2$ such that $\ib^*(Y_1)$ and $\ib^*(Y_2)$ are neither identical nor orthogonal to each other. 
\end{enumerate}
\end{defin}

\begin{claim} \label{NonParClaim2}
$H$ has at most $k^3 \mu_*^{1/4} n$ bad vertices and $R$ has at most $\sqrt{\mu_*} n_R^k$ bad edges.
\end{claim}

The first part of the claim is immediate from the definitions, using $|\hPart^*| \le k^2$ and $n_0n_R=kn$. 
For the second, we will show that if $e$ is bad then the number of edges $e' \in H_e$ with $\ib_{\hPart^*}(e') \notin L_{\hPart^*}^{\mu_*}(H')$ is at least $\mu_*^{1/3} n_0^k$. The claim then follows from the fact that $H'$ has at most $\ell^k \mu_* (kn)^k$ such edges in total.

First we consider $e \in R$ that is bad because it contains a bad cluster $Y_1$. Then there is $Z \in \hPart^*$ such that $\ib^*(Y_1)_Z = 0$ but $|Y_1 \cap Z| \geq \mu_*^{1/4} n_0$. Without loss of generality $Z \subseteq V_1$. Let $Y_2, \ldots, Y_k$ be the remaining clusters of $e$, and let $Z_j$ be the parts of $\hPart^*$ such that $Z_j \subseteq V_j$ and $\ib^*(Y_j)_{Z_j} = 1$ for each $j \in [k]$. Now we consider two induced subgraphs of $H_e$: let $F_1$ have parts $(Y_j \cap Z_j)_{j \in [k]}$ and $F_2$ have parts $Y_1 \cap Z$ and $(Y_j \cap Z_j)_{2 \leq j \leq k}$. Let $\ib_1$ and $\ib_2$ be the index vectors with respect to $\hPart^*$ of edges in $F_1$ and $F_2$ respectively. Note that $\ib_1$ and $\ib_2$ are $\Part'$-partite $k$-vectors, so the edges of $F_1$ and $F_2$ are $\Part'$-partite. Also observe that $\ib_1 - \ib_2 = \ub_{Z_1} - \ub_{Z}$, and hence at least one of $\ib_1$ and $\ib_2$ does not lie in $L_{\hPart^*}^{\mu_*}(H')$, since $Z$ and $Z_1$ are distinct and $L_{\hPart^*}^{\mu_*}(H')$ is transferral-free by our choice of $\hPart$. Thus it suffices to show that each of $F_1$ and $F_2$ contains at least $\mu_*^{1/3} n_0^k$ edges. To see this, note that by (G1) and the choice of $Z$, each of $F_1$ and $F_2$ contains at least $c_*n_0$ vertices from $k-1$ vertex classes and at least $\mu_*^{1/4} n_0$ vertices from the remaining vertex class. Furthermore $H_e$ is $\eps$-regular with density at least $\mu'/4$ by definition of the reduced $k$-graph $R$. Thus each of $F_1$ and $F_2$ contains at least $(\mu'/4 - \eps)(c_*)^{k-1}\mu_*^{1/4} n_0^k > \mu_*^{1/3} n_0^k$ edges, as required.

The other case is that $e \in R$ is bad because it contains two clusters $Y_1, Y_2$ such that $\ib^*(Y_1), \ib^*(Y_2)$ are neither identical nor orthogonal to each other. Then without loss of generality there exist $Z_1, Z'_1, Z_2 \in \hPart^*$ with $Z_1, Z'_1 \subseteq V_1$ and $Z_2 \subseteq V_2$, such that $\ib^*(Y_1)_{Z_1} = \ib^*(Y_2)_{Z'_1} = \ib^*(Y_1)_{Z_2} = \ib^*(Y_2)_{Z_2} = 1$. Let $Y_3, \ldots, Y_k$ be the remaining clusters of $e$ and select $Z_j \in \hPart^*$ such that $Z_j \subseteq V_j$ and $\ib^*(Y_j)_{Z_j} = 1$ for each $3 \leq j \leq k$. Now we consider two induced subgraphs of $H_e$: let $F_1$ have parts $(Y_j \cap Z_j)_{j \in [k]}$ and $F_2$ have parts $Y_2 \cap Z'_1, Y_1 \cap Z_2$ and $(Y_j \cap Z_j)_{3 \leq j \leq k}$. Similarly to the previous case, each of $F_1$ and $F_2$ contains at least $\mu_*^{1/3} n_0^k$ edges, each of which is $\hPart$-partite, and for one of them the index vector is not in $L_{\hPart^*}^{\mu_*}(H')$.
This completes the proof of the claim. 

\medskip

In the next claim we use the index vectors from Claim~\ref{NonParClaim1} 
that are well-represented to define a partition of $V$.
More precisely, we let
\[ R_\ib = \{Y \in V(R) \mid \ib^*(Y) = \ib\} \ \ \text{ and } \ \ 
T = \{\ib \in \Z^{\hPart^*} \mid |R_\ib| \geq \mu_*^{1/5} n_R\}.\]
Then we define $\Part^* = \{X_\ib^*\}_{\ib \in T}$, where $X_\ib^* = \bigcup_{Z \in \hPart^*, i_Z = 1} Z$ for each $\ib \in T$. 

\begin{claim} \label{NonParClaim3}
$\Part^*$ is a partition of $V$ which strictly refines $\Part$ into parts of size at least $c|V|$.
\end{claim}

To see this, we start by showing that any distinct vectors $\ib_1$, $\ib_2$ of $T$ are orthogonal. For suppose otherwise, and note that the number of $(k-1)$-sets of clusters that intersect $R_{\ib_1}$ and $R_{\ib_2}$ is at least $k!^{-1} (n_R)^{k-3} |R_{\ib_1}| |R_{\ib_2}| \geq k!^{-1} \mu_*^{2/5} (n_R)^{k-1}$. By Claim~\ref{NonParClaim0} all but at most $\gamma^{-1} \eps n_R^{k-1}$ of these $(k-1)$-sets $A$ satisfy $|R(A)| \ge n_R/k$. But then we obtain at least $(k!^{-1}\mu_*^{2/5} - \gamma^{-1} \eps) n_R^{k-1} (n_R/k) > \sqrt{\mu_*} n_R^k$ bad edges, which contradicts Claim~\ref{NonParClaim2}. 

Thus for every $Z \in \hPart^*$ there is at most one $\ib \in T$ such that $i_Z = 1$, so the parts of $\Part^*$ are pairwise disjoint. Furthermore, for each $Z \in \hPart^*$ there must be some $\ib \in T$ such that $i_Z = 1$. Indeed, if $i_Z = 0$ for every $\ib \in T$ then the number of $Y \in V(R)$ with $i^*(Y)_Z = 1$ is at most $\ell^k \mu_*^{1/5} n_R$, which contradicts (G2). Thus $\Part^*$ is a partition of $V$. Clearly each part of $\Part^*$ has size at least $c|V|$, since this is true of $\hPart^*$. 

Finally, to see that $\Part^*$ refines $\Part$, recall that $\Qart$ was chosen to be $\eps$-close to a refinement of $\Part$. So there are at most $\eps |V| / c_*n_0 < \mu_*^{1/5}n_R$ clusters $Y$ of $\Qart$ which intersect more than one part of $\Part$ in at least $c_*n_0$ vertices. Now, consider any $\ib \in T$. By definition of $T$ we have $|R_\ib| \geq \mu_*^{1/5}n_R$, so we may choose $Y \in R_\ib$ for which there is a unique part $X \in \Part$ such that $|X \cap Y| \geq c^*n_0$. We claim that $Z \sub X$ for any $Z \in \hPart^*$ with $Z \sub X_\ib^*$. Indeed, since $\hPart^*$ is a refinement of $\Part$ we have $Z \sub X'$ for some $X' \in \Part$. But $i_Z=1$ by definition of $X_\ib^*$, so $|Y \cap Z| \geq c_*n_0$ by (G1). By choice of $Y$ this implies that $X=X'$, as claimed. This completes the proof of Claim~\ref{NonParClaim3}.
 
\medskip

Our final claim will complete the proof of the lemma. Indeed, it implies that $\ib \in L_{\hPart^*}^{\mu_*}(H')$ for any index vector $\ib$ with respect to $\hPart^*$ such that $\ib' = (\ib \mid \Part^*) \in L_{\Part^*}^{\mu'}(H)$ and $(\ib \mid \Part')$ is a multiple of $\1$. 
We can apply this with $\ib' \in L_{\Part^*}^{\mu'}(H)$ equal to a transferral $\ub_{X_1^*} - \ub_{X_2^*}$, 
which must exist since $\Part$ is $(c, c, \mu, \mu')$-robustly maximal with respect to $H$ and $\Part^*$ strictly refines $\Part$ into parts of size at least $c|V|$. Letting $Z_1, Z_2$ be the intersections of $X_1^*$ and $X_2^*$ respectively with $V_1$ 
we deduce that $\ub_{Z_1} - \ub_{Z_2} \in L_{\hPart^*}^{\mu_*}(H)$, 
contradicting our assumption that $L_{\hPart^*}^{\mu_*}(H)$ was transferral-free. 
Thus it remains to prove the following claim.

\begin{claim} \label{LastClaim}
For any $\Part'$-partite $k$-vector $\ib$ with respect to $\hPart^*$ such that $(\ib \mid \Part^*) \in I_{\Part^*}^{\mu'}(H)$
we have $\ib \in I_{\hPart^*}^{\mu_*}(H')$.
\end{claim}

To see this, let $J_\ib$ be the set of good edges $e \in R$ such that $e \sub \bigcup_{{\ib'} \in T} R_{\ib'}$ for which
$H_e$ contains some edge $e'$ with no bad vertex and $\ib_{\Part^*}(e') = (\ib \mid \Part^*)$.
We claim that $|J_\ib| \ge \mu' n_R^k/2$.
To see this, note that since $(\ib \mid \Part^*) \in I_{\Part^*}^{\mu'}(H)$, 
there are at least $\mu' (kn)^k$ edges $e' \in H$ such that $\ib_{\Part^*}(e') = (\ib \mid \Part^*)$.
Of these, at most $\eps (kn)^k + \mu' (kn)^k/4$ are not contained in $\bigcup_{e \in R} H_e$,
and at most $k^2 \mu_*^{1/4} (kn)^k$ contain a bad vertex.
Thus there are at least $2\mu' n_R^k/3$ edges of $R$ 
that contain at least one of the remaining edges $e'$.
Of these edges of $R$, at most $\sqrt{\mu_*} n_R^k$ are bad, 
and at most $k^k \mu_*^{1/5} n_R^k$ intersect $\bigcup_{\ib \notin T} R_\ib$.
We conclude that $|J_\ib| \ge \mu' n_R^k/2$, as claimed.

Next we claim that for any $e \in J_\ib$ there are 
at least $(\mu')^2n_0^k$ edges $e^* \in H_e$ with $\ib_{\hPart^*}(e^*) = \ib$.
To see this, fix some $e' \in H_e$ with no bad vertex, and consider any $v \in e'$. Since $v$ is not bad we have $\ib^*(Y_v)_{Z_v}=1$ for the $Y_v\in e$ and $Z_v \in \hPart^*$ with $v \in Y \cap Z$, and since each cluster $Y \in e$ lies in $R_{\ib'}$ for some $\ib' \in T$, we have $\ib^*(Y_v) \in T$.
Recall that the vectors in $T$ are orthogonal, so for each $Z \in \hPart^*$ there is a unique $\ib^Z \in T$ such that $i^Z_Z=1$. So we must have $\ib^*(Y_v)=\ib^{Z_v}$, that is, $Y_v \in R_{\ib^{Z_v}}$.
Since $\ib_{\Part^*}(e') = (\ib \mid \Part^*)$, there must be exactly $(\ib \mid \Part^*)_{X_{\ib'}^*}$ 
clusters of $e$ in $R_{\ib'}$ for each $\ib' \in T$.
Thus we can order the clusters of $e$ as $(Y_1,\dots,Y_k)$
such that $\ib^*(Y_j)_{Z_j} = 1$ for each $j \in [k]$,
where $Z_j \sub V_j$ is such that $i_{Z_j}=1$ for $j \in [k]$.
Now the sets $Y_j \cap Z_j$ for $j \in [k]$ each have size at least $c_*n_0$ by (G1), 
and so by definition of $R$ they induce a subgraph of $H_e$ 
with at least $(\mu'/4 - \eps)(c_*n_0)^k > (\mu')^2n_0^k$ edges, as claimed.

Summing over $e \in J_\ib$, we obtain at least $(\mu')^2n_0^k \cdot \mu' n_R^k/2 \geq \mu_* (kn)^k$ 
edges $e^* \in H$ with $\ib_{\hPart^*}(e^*) = \ib$, and each such edge lies in $H'$ since $\ib$ is $\hPart$-partite. Thus $\ib \in I_{\hPart^*}^{\mu_*}(H')$, 
which completes the proof of Claim~\ref{LastClaim}, and so of the lemma. 
\qed 

\section{Concluding Remarks} \label{PartiteSection}

The only case where we did not resolve the complexity status of $\PM(k, \delta)$ is when $\delta=1/k$;
this has been solved while this paper was under review by Han~\cite{H2}:
he showed that there is a polynomial time algorithm, 
using some theory developed in our paper and a lattice-based absorbing method. 

We will conclude our paper with some remarks on 
multipartite versions of our results, 
tightness of the parameters, 
and other degree conditions.

\subsection{Multipartite analogues}
The following multipartite versions of our main results may be proved similarly to the non-partite versions.

\begin{thm} \label{PartiteMain} 
Fix $k \geq 3$ and $\gamma > 0$. Then there is an algorithm with running time $O(n^{3k^2 - 7k + 1})$, which given any $k$-partite $k$-graph with parts of size $n$ such that every partite $(k-1)$-set has degree at least $(1/k + \gamma)n$, finds either a perfect matching or a certificate\footnote{A $C$-certificate for a $k$-partite $k$-graph is defined in an analogous way as for a general $k$-graph.} that no perfect matching exists.
\end{thm}

\begin{thm} \label{PartitePMNeccSuff} 
Suppose that $k \ge 3$ and $1/n \ll \eps \ll \gamma \ll 1/k$.
Let $\Part'$ partition a vertex set $V$ into $k$ parts each of size $n$.
Suppose that $H$ is a $\Part'$-partite $k$-graph $H$ on $V$
such that $\delta_1(H) \geq \gamma n^{k-1}$, and
at most $\eps n^{k-1}$ $\Part'$-partite $(k-1)$-sets $A \sub V(H)$ have $d_H(A) < (1/k + \gamma)n$.
Then $H$ has a perfect matching if and only if there is no $2k(k-3)$-certificate for $H$.
\end{thm} 

Let $\mathbf{PPM}(k, \delta)$ denote the problem of deciding whether there is a perfect matching in a given $k$-partite $k$-graph with parts of size $n$ such that every partite $(k-1)$-set has degree at least $\delta n$. Theorem~\ref{PartiteMain} implies that $\mathbf{PPM}(k, \delta)$ can be decided in polynomial time when $\delta > 1/k$. On the other hand, a similar argument to that of Szyma\'nska~\cite{Szy09} shows that $\mathbf{PPM}(k, \delta)$ is NP-complete for $\delta < 1/k$.

\subsection{Tightness of parameters}
We believe that a stronger version of Lemma~\ref{ImplySoluble} is true, in which $2k(k-3)$ is replaced by $k(k-3)$, as in the case for full pairs $(\Part,L)$ where $|G(\Part, L)|$ is prime (see Remark~\ref{Prime}).
This would immediately imply the following improved version of Theorem~\ref{PMNeccSuff}.

\begin{conj} \label{StrongerPMNS} 
Under Setup~\ref{setup}, $H$ has a perfect matching if and only if there is no $k(k-3)$-certificate for $H$.
\end{conj}

A proof of this conjecture would allow us to improve the running time in Theorem~\ref{main}.
We also conjecture that a similar strengthening of Theorem~\ref{PartitePMNeccSuff} holds.
Conjecture~\ref{StrongerPMNS}, if true, would be best-possible.
To see this, consider the following construction, which is similar to Construction~\ref{nopm}.

\begin{construct} \label{Generalnopm}
Let $\Part = (A_1, \ldots, A_{k-1})$ be a partition of a set of $n$ vertices 
with $|A_j| = n/(k-1) \pm 2$ and $\sum_{j \in [k-1]} j|A_j| = k-2$ modulo $k-1$. 
Let $B$ be a subset of $A_1$ of size $k(k-2) - 1$.
Let $H$ be the $k$-graph with vertex set $\bigcup_{j \in [k-1]} A_j$ whose edges are
\begin{enumerate}
\item any $k$-set $e$ with $\sum_{j \in [k-1]} j|e \cap A_j| = 0$ modulo $k-1$, and
\item any $k$-set of vertices in $B$.
\end{enumerate}
\end{construct}

Consider the full pair $(\Part, L)$ for $H$, where $L$ is the lattice of index vectors $\ib$ such that $\sum_{j \in [k-1]} i_j j = 0$ modulo $k-1$, and note that $(\Part, L)$ is a $k(k-3)$-far full pair for $H$.
Since any matching $M$ contains at most $k-3$ edges contained in $B$, 
we have $\sum_{j \in [k-1]} j|V(M) \cap A_j| \neq k-2$ for any matching $M$ in $H$ and so $H$ cannot contain a perfect matching.
On the other hand, $(\Part, L)$ is not a $(k(k-3) - 1)$-far full pair for $H$ since removing $k(k-3) - 1$ vertices from $B$ still leaves $k$ vertices which form a single edge $e$ with $\ib_\Part(e) \notin L$.
It is not hard to show that indeed no $(k(k-3) - 1)$-far full pair for $H$ exists.

We also remark that the following construction shows that the constant $1/k$ in Setup~\ref{setup} is best-possible.

\begin{construct} \label{spacebar} (Space Barrier)
Let $V$ be a set of size $n$ and fix $S \sub V$ with $|S|<n/k$. 
Let $H$ be the $k$-graph whose edges are all $k$-sets that intersect $S$.
\end{construct}

Note that any matching in Construction~\ref{spacebar} has at most $|S|$ edges, so is not perfect.

\subsection{Other degree conditions}
It is natural to ask whether similar results could be obtained 
for the decision problem for perfect matching in $k$-graphs
under the weaker complex degree sequence assumption used in~\cite{KM11p}.
However, our inductive argument in the key lemma depends crucially
on the inheritance of the codegree condition by subgraphs,
and this inheritance need not hold for a complex degree sequence condition.
Thus an alternative proof strategy would be needed for such a result.

It is also natural to consider the decision problem of determining the existence of a perfect matching in $k$-graphs 
satisfying a minimum $\ell$-degree condition for any $\ell \in [k-1]$.
However, one should note that the $\ell$-degree threshold for the existence 
of a perfect matching is open in general. One can read~\cite{HPS} as suggesting
that either a space barrier or a divisibility barrier is always extremal for such a problem,
and this was formalised as Conjecture 3.6 in~\cite{RR10}.
Let $\mathbf{PM}(k,\ell,\delta)$ denote the problem of deciding whether there is a perfect matching 
in a given $k$-graph on $n$ vertices with minimum $\ell$-degree at least $\delta \binom{n}{k-\ell}$.
Szyma\'nska~\cite{Szy} showed that $\mathbf{PM}(k,\ell,\delta)$ is NP-complete
when $\delta$ is less than the space barrier threshold,
i.e.\ $\delta < 1 - (1-1/k)^{k-\ell}$.
It is then natural to make the following conjecture. 

\begin{conj} \label{vdegconj}
$\mathbf{PM}(k,\ell,\delta)$ is in P for $\delta > 1 - (1-1/k)^{k-\ell}$.
\end{conj}

Theorem~\ref{main} establishes the case $\ell=k-1$ of Conjecture~\ref{vdegconj}. Beyond this, the conjecture is immediate for those cases where the space barrier threshold is (asymptotically) equal to the threshold at which a perfect matching is guaranteed. The conjecture of H\`an, Person and Schacht mentioned above would therefore imply all cases of Conjecture~\ref{vdegconj} with $(1-1/k)^{k-\ell} \leq 1/2$, and some cases of the latter conjecture are implied by partial results for the former. Specifically, Conjecture~\ref{vdegconj} holds in the case $k = 3, \ell = 1$ by a result of H\`an, Person and Schacht~\cite{HPS}, in the case $k=4, \ell = 1$ by a result of Lo and Markstr\"om~\cite{LM}, and in the cases $k=5, \ell = 1$ and $k=6, \ell = 2$ by results of Alon, Frankl, Huang, R\"odl, Ruci\'nski and Sudakov~\cite{Alon+}. To our knowledge all other cases remain open. 

\section*{Acknowledgements}

We thank the anonymous referees for many helpful comments.

\end{document}